\documentclass[11pt,twoside,a4paper]{article}
\usepackage[english]{babel}
\usepackage{a4wide}
\usepackage{graphicx}
\usepackage{amsmath,amssymb,amsthm, amsgen}
\usepackage{listings}
\usepackage{color}
\usepackage[usenames,dvipsnames]{xcolor} 
\usepackage{rotating} 
\usepackage{hhline}
\usepackage{multirow}
\usepackage{enumerate} 
\usepackage[bookmarks]{}
\usepackage{amsfonts}   
\usepackage{dsfont}
\usepackage{tikz}
\usepackage{booktabs}
\usepackage{longtable} 
\usepackage{enumitem}
\usepackage{stmaryrd} 

\newtheorem{thm}{Theorem}[section]
\newtheorem{prop}[thm]{Proposition}
\newtheorem{lem}[thm]{Lemma}
\newtheorem{ass}[thm]{Assumption}
\newtheorem{cor}[thm]{Corollary}
\newtheorem{defn}[thm]{Definition}

\theoremstyle{definition}
\newtheorem{rem}[thm]{Remark}
\newtheorem{ex}[thm]{Example}

\newcommand{\fext}{g} 

\newcommand{\loc}{\operatorname{loc}}

\newcommand{\pv}{\mathrm{pv}}
\newcommand{\qtext}[1]{\quad \text{#1} \quad}
\newcommand{\qand}{\qtext{and}}
\newcommand{\sign}{\operatorname{sign}}
\newcommand{\supp}{\operatorname{supp}}

\newcommand{\weakto}{ \rightharpoonup }
\newcommand{\xto}[1]{\xrightarrow{#1}}

\newcommand{\e}{\varepsilon}
\newcommand{\N}{\mathbb N}

\newcommand{\R}{\mathbb R}

\newcommand{\V}{\mathbb V}
\newcommand{\Z}{\mathbb Z}

\newcommand{\bb}{\mathbf b}

\newcommand{\bx}{\mathbf x}

\newcommand{\cM}{\mathcal M}
\newcommand{\cN}{\mathcal N}

\newcommand{\cZ}{\mathcal Z}

\renewcommand{\o}{\overline}

\DeclareFontFamily{U}{mathx}{\hyphenchar\font45}
\DeclareFontShape{U}{mathx}{m}{n}{
      <5> <6> <7> <8> <9> <10>
      <10.95> <12> <14.4> <17.28> <20.74> <24.88>
      mathx10
      }{}
\DeclareSymbolFont{mathx}{U}{mathx}{m}{n}
\DeclareFontSubstitution{U}{mathx}{m}{n}
\DeclareMathAccent{\widecheck}{0}{mathx}{"71}

\numberwithin{equation}{section}

\title{Discrete-to-continuum limits of interacting particle systems in one dimension  with collisions}

\author{Patrick van Meurs\footnote{
Faculty of Mathematics and Physics, Kanazawa University, Kakuma, Kanazawa 920-1192, Japan. 
e-mail: pjpvmeurs@staff.kanazawa-u.ac.jp
}}

\date{}

\long\def\comm#1{{\color{orange}#1}}
\long\def\comm#1{}

\begin{document}

\maketitle

\begin{abstract}
We study a class of interacting particle systems in which $n$ signed particles move on the real line. This class includes the overdamped dynamics (i.e.\ `velocity = force') of charged particles interacting by the Coulomb potential. Particles of opposite sign collide typically in finite time. Upon collision, pairs of colliding particles are removed from the system. 

In a recent paper by Peletier, Po\v{z}\'ar and the author, the specific case of the Coulomb potential was studied. Global well-posedness of the corresponding particle system was shown and a discrete-to-continuum limit (i.e.\ $n \to \infty$) to a nonlocal PDE for the signed particle density was established. Both results rely on innovative use of techniques in ODE theory and viscosity solutions.

The particle system with the Coulomb potential has a limited range of applications. The aim of the present paper is to broadly widen this range by adding an external potential and by considering a large class of interaction potentials which, moreover, may scale with $n$. Several key steps in the proofs for the Coulomb potential do not carry over; new techniques are developed to establish these steps.
\end{abstract}
{\textbf{Keywords}}: {Interacting particle systems, asymptotic analysis, viscosity solutions.}

\noindent{\textbf{MSC}}: {74H10, 35D40, 34E18.}



\noindent Declarations of interest: none.


\section{Introduction}
\label{s:intro}

The recent paper \cite{VanMeursPeletierPozar22} introduces several new techniques by which it is possible to describe the dynamics of $n$ signed particles interacting by the Coulomb potential, and to pass to the limit $n \to \infty$. Motivated by various applications, our aim is to further develop these techniques such that similar well-posedness and discrete-to-continuum limit results hold for many other potentials. After briefly revisiting the setting and main results of \cite{VanMeursPeletierPozar22} in Section \ref{s:intro:vMPP}, we introduce in Section \ref{s:intro:aim} the class of particle systems which we consider in this paper, and list in Section \ref{s:intro:lit} the applications.

\subsection{{\cite{VanMeursPeletierPozar22}}: charged particles and the limit $n \to \infty$}
\label{s:intro:vMPP}

In \cite{VanMeursPeletierPozar22} the starting point is the interacting particle system formally given by
   \begin{equation} \label{Pn:vMPP}
   \left\{ \begin{aligned}
     &\frac{dx_i}{dt} = \frac{1}{n} \sum_{ \substack{ j = 1 \\ j \neq i} }^n \frac{b_i b_j}{x_i - x_j}
     && t \in (0,T), \ i = 1,\ldots, n \\
     &+ \text{annihilation upon collision}
     &&
   \end{aligned} \right.
\end{equation}
subject to an initial condition,
where $n \geq 2$ is the number of particles, $\bx := (x_1, \ldots, x_n)$ are the particle positions in $\R$, and $\bb := (b_1, \ldots, b_n)$ are the particle charges in $\{-1,+1\}$. One can think of the particles as being electrically charged, i.e.\ particles repel or attract depending on their charges. The evolution can be though of as the overdamped limit of Newtonian point masses. 

Figure \ref{fig:trajs} illustrates the particle dynamics. The most intricate feature is that particle collision occur in finite time. At a collision the right-hand side in \eqref{Pn:vMPP} is not defined. In order to show how to overcome this with the `annihilation rule', it is instructive to consider first the case of $n=2$ particles with opposite charges. Then,
\begin{equation} \label{sol:n2:vMPP}
  x_1(t) = -\frac12 \sqrt{1 - 2t},
  \qquad 
  x_2(t) = \frac12 \sqrt{1 - 2t}
\end{equation}
is a solution to \eqref{Pn:vMPP} until the collision time $\tau = \frac12$. The trajectories form a parabola when drawn as in Figure \ref{fig:trajs}. Also for general $n$ it seems that the shape of trajectories shortly before collisions are approximately parabolic, but there is no sufficient evidence in \cite{VanMeursPeletierPozar22} to guarantee this. The main difficulty is that collisions are not restricted to $2$ particles; Figure \ref{fig:trajs} illustrates collisions between three particles, and it seems to be possible that any number of particles can collide with each other. 

Next we formally describe how collisions are resolved by the annihilation rule and how the particle dynamics are continued after collisions. Whenever a pair of particles with opposite charge collide, they are removed from the system. The remaining particles continue to evolve by the system of ODEs \eqref{Pn:vMPP}. The removal of pairs is done sequentially. As a consequence, not all colliding particles are necessarily removed; at any $3$-particle collision, for instance, one particle survives.

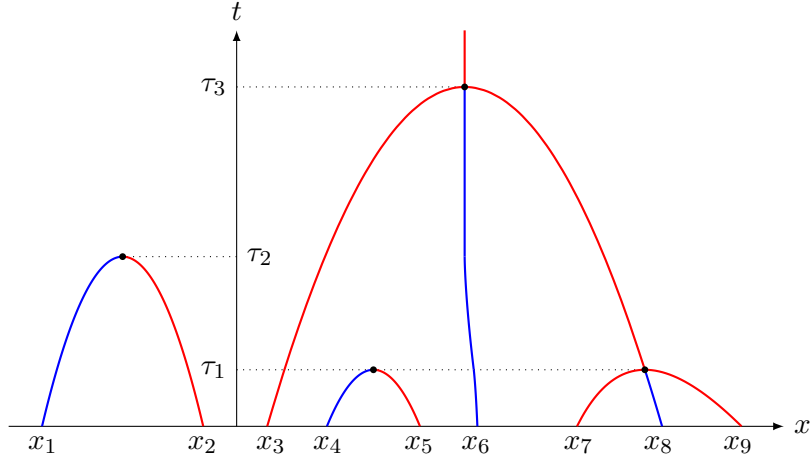
\begin{figure}[ht]
\centering
\begin{tikzpicture}[scale=1.5, >= latex]
\def \sqtwo {1.414}
\def \rr {0.03}

\draw[->] (0,0) -- (0,3.5) node[above] {$t$};
\draw[->] (-2,0) -- (4.8,0) node[right] {$x$};

\draw[dotted] (0,.5) node[left]{$\tau_1$} -- (3.581,.5);
\draw[dotted] (0,1.5) node[right]{$\tau_2$} -- (-1,1.5);
\draw[dotted] (0,3) node[left]{$\tau_3$} -- (2,3);

\draw (-1.7,0) node[below] {$x_1$};
\draw (-.3,0) node[below] {$x_2$};
\draw (.3,0) node[below] {$x_3$};
\draw (.8,0) node[below] {$x_4$};
\draw (1.6,0) node[below] {$x_5$};
\draw (2.1,0) node[below] {$x_6$};
\draw (3,0) node[below] {$x_7$};
\draw (3.7,0) node[below] {$x_8$};
\draw (4.4,0) node[below] {$x_9$};

\begin{scope}[shift={(2,3)},scale=1] 
    \draw[thick, red] (0,0) -- (0,.5);
    \draw[thick, blue] (0,-1.5) -- (0,0);
    \draw[domain=-1.732:1.581, smooth, thick, red] plot (\x,{-\x*\x});           
    \draw[domain=1.581:1.732, smooth, thick, blue] plot (\x,{-\x*\x});           
    \fill[black] (0,0) circle (\rr); 
\end{scope}

\begin{scope}[shift={(2,1.5)},scale=1, rotate = 270] 
    \draw[domain=0:1, smooth, thick, blue] plot (\x,{.08*\x*sqrt(\x)}); 
    \draw[domain=1:1.5, smooth, thick, blue] plot (\x,{.08*\x*sqrt(\x) - .1*(\x - 1)*sqrt(\x - 1)});         
\end{scope}

\begin{scope}[shift={(3.581,.5)},scale=1] 
    \draw[domain=-.707:.707, smooth, thick, red] plot ({2*exp(-\x/2)-2},{-\x*\x});
    \fill[black] (0,0) circle (\rr);  
\end{scope}

\begin{scope}[shift={(1.2,.5)},scale=1] 
    \draw[domain=0:.408, smooth, thick, red] plot (\x,{-3*\x*\x});
    \draw[domain=-.408:0, smooth, thick, blue] plot (\x,{-3*\x*\x});     
    \fill[black] (0,0) circle (\rr);     
\end{scope}

\begin{scope}[shift={(-1,1.5)},scale=1] 
    \draw[domain=-.707:0, smooth, thick, blue] plot (\x,{-3*\x*\x});  
    \draw[domain=0:.707, smooth, thick, red] plot (\x,{-3*\x*\x}); 
    \fill[black] (0,0) circle (\rr);         
\end{scope}

\end{tikzpicture} \\
\caption{A sketch of solution trajectories to \eqref{Pn}. Trajectories of particles with positive charge are colored red; those with negative charge blue. The black dots indicate the collision points.}
\label{fig:trajs}
\end{figure}

The two main results of \cite{VanMeursPeletierPozar22} are as follows. The first is the well-posedness of \eqref{Pn:vMPP}, i.e.\ the existence and uniqueness of solutions as well as stability with respect to perturbations of the initial condition. This well-posedness result is nontrivial for three reasons:
\begin{enumerate}
  \item at a collision time $\tau$, the right-hand side of the ODE blows up, and thus it is not clear why the limit $\lim_{t \uparrow \tau} x_i(t)$ exists,
  \item collisions between more than two particles are possible, and
  \item for the ODE system to be restarted after the annihilation rule is applied, it is necessary that no two particles are at the same location, as otherwise the right-hand side of the ODE is not defined.
\end{enumerate}
The key a priori estimate in \cite{VanMeursPeletierPozar22} to overcome these challenges states that the minimal distance between any neighboring particles of the same sign is increasing in time. As a consequence, any collection of colliding particles must have \textit{alternating} charges. This feature of the particle configuration gives sufficient control for overcoming the three challenges mentioned above.

The second main result of \cite{VanMeursPeletierPozar22} is the limit passage as $n \to \infty$. The limiting PDE is given formally by
\begin{equation} \label{P:vMPP}
  \partial_t \kappa = -\partial_x (|\kappa| H[\kappa] ),
  \qquad H[\kappa](x) := \pv \int_\R \frac{\kappa(y)}{x-y} \, dy,
\end{equation}
where $\kappa = \kappa(t,x)$ is the signed particle density. The connection with the ODE system is that the signed empirical measure
\begin{equation*} 
  \kappa_n := \frac1n \sum_{i=1}^n b_i \delta_{x_i}
\end{equation*}
converges to $\kappa$ as $n \to \infty$. In \eqref{P:vMPP} $H[\kappa]$ is the Hilbert transform; for $\kappa$ smooth enough it is simply the convolution with the Coulomb force (hence the principle-value integral). Alternatively, the Hilbert transform can be expressed in terms of the half-Laplacian as follows: 
\begin{equation} \label{HalfL}
  H[\kappa] = (- \Delta)^{\tfrac12} u, \qquad u := \int \kappa(x) \, dx.
\end{equation}
The appearance of the absolute value in the right-hand side in \eqref{P:vMPP} is not very common in PDEs. Its presence has two consequences:
\begin{enumerate}
  \item the positive charge density $\kappa_+$ and the negative charge density $\kappa_-$ (i.e.\ the positive and negative parts of $\kappa$) act in opposite direction to the velocity field $H[\kappa]$, and
  \item at points where $\supp \kappa_+$ and $\supp \kappa_-$ touch (i.e.\ where $\kappa$ changes sign), positive and negative charge may cancel out in time (i.e.\ $\int_{\R} |\kappa(t,x)| \, dx$ is typically decreasing in time).
\end{enumerate} 

The limit passage of the particle system as $n \to \infty$ was not obvious for two main reasons:
\begin{enumerate}
  \item the PDE is nonlocal and nonlinear, and contains the two nonsmooth components given by $H[\kappa]$ and  $|\kappa|$. It is not even obvious why this PDE should have a meaningful solution concept;
  \item the convergence of $\kappa_n$ to $\kappa$ is naturally done in the space of measures. However, the right-hand side of the PDE is not well-defined on this space.
\end{enumerate}
The first challenge was overcome in \cite{BilerKarchMonneau10}. There, it is shown that when integrating the PDE over space and working with the function $u = \int \kappa(x) \, dx$ instead of $\kappa$, the resulting PDE for $u$ satisfies a comparison principle. 
This allows one to work with viscosity solutions; a framework in which one can effectively replace $u$, which is possibly of low regularity, with regular test functions. \cite{BilerKarchMonneau10} uses this framework to establish existence, uniqueness and properties of solutions. Finally, the solution $\kappa$ to \eqref{P:vMPP} is then simply defined from the viscosity solution $u$ of the integrated equation as $\kappa := \partial_x u$.

Regarding the second challenge, in \cite{ForcadelImbertMonneau09} an attempt was made to prove the convergence of the ODE system by developing a similar notion of viscosity solutions for a `spatially integrated' version of the ODE system. Then, the limit passage can be carried out effectively with regular test functions for which the right-hand side of the PDE is well-defined. The authors succeeded (among other achievements) in proving the convergence as $n \to \infty$ in the \textit{single}-charge case (i.e.\ $b_i = 1$ for all $i$). In this case the particles do not collide. This allows for regularizating the Coulomb force. However, for the case of signed particles a regularized Coulomb force can cause complications such as the formation of persisting particle clusters which influence the dynamics of surrounding particles; see \cite[Chapter 9]{VanMeurs15} for simulations. Thirteen years later, this difficulty was overcome in \cite{VanMeursPeletierPozar22}. Instead of working with a regularization, the authors modify the notion of viscosity solution. The idea is technical; reducing the class of test functions at points where collisions may occur in order to handle the singularity at collisions.

\subsection{Setting and main results of the present paper}
\label{s:intro:aim}

In this paper, we consider a generalized version of \eqref{Pn:vMPP}. Formally, it is given by
\begin{align} \label{Pn} \tag{$P_n$}
  \left\{ \begin{aligned}
     &\frac{dx_i}{dt} 
       = - \frac1{n} \sum_{j \neq i} b_i b_j V_{\alpha_n}' (x_i - x_j) - b_i U'(x_i),
     && t \in (0,T), \ i = 1,\ldots, n \\
     &+ \text{annihilation upon collision,}
     &&
   \end{aligned} \right.
\end{align}
where  
\begin{equation} \label{Val}
  V_{\alpha_n}(x) := \alpha_n V(\alpha_n x)
\end{equation}
is a rescaling of an interaction potential $V$, $\alpha_n > 0$ is a parameter and $U$ is an external potential. 
The previous system \eqref{Pn:vMPP} is covered by \eqref{Pn} with the choices $V(x) = -\log|x|$, $U=0$ and $\alpha_n = 1$. Figure \ref{fig:VU} illustrates typical choices of $V$ and $U$ which we have in mind. 

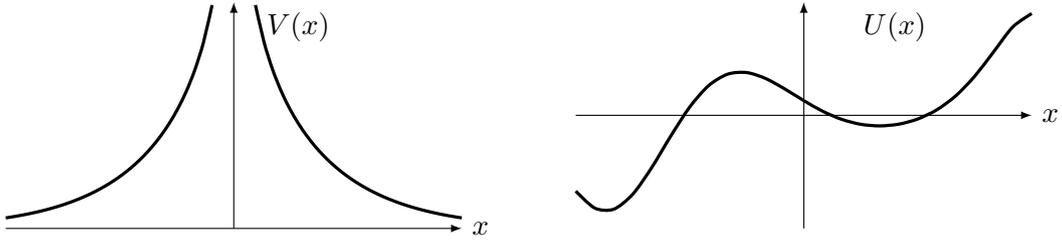
\begin{figure}[h]
\centering
\begin{tikzpicture}[scale=1.5, >= latex]    
\def \w {2}
        
\draw[->] (0,0) -- (0,\w);
\draw[->] (-\w,0) -- (\w,0) node[right] {$x$};
\draw[domain=0.19:\w, smooth, very thick] plot (\x,{ \x * cosh(\x) / sinh(\x) - ln( 2 * sinh(\x) ) });
\draw[domain=-\w:-0.19, smooth, very thick] plot (\x,{ \x * cosh(\x) / sinh(\x) - ln( -2 * sinh(\x) ) });
\draw (0.2, \w) node[anchor = north west]{$V(x)$};

\begin{scope}[shift={(2.5*\w, .5*\w)},scale=1]
\draw[->] (0,-.5*\w) -- (0, .5*\w);
\draw[->] (-\w,0) -- (\w,0) node[right] {$x$};
\draw[domain=-\w:\w, smooth, very thick] plot ( \x, { .5*sin(240*(\x - exp(.6*\x)) + 45) + .2*\x } );
\draw (.8, .5*\w) node[below]{$U(x)$};
\end{scope}
\end{tikzpicture} \\
\caption{Typical examples of $V$ and $U$.}
\label{fig:VU}
\end{figure}

The system \eqref{Pn} generalizes \eqref{Pn:vMPP} with the three new objects $U, V, \alpha_n$. We motivate each of them by applications in Section \ref{s:intro:lit}. 
We interpret $-U'$ as an external force field. We take $U \in C^1(\R)$ such that $U'$ is Lipschitz continuous. Note that $U$ need not be bounded; for instance, $U(x) = x^2$ meets the requirements. While it is common in particle systems to consider $U$ as a confining potential, this interpretation does not hold for \eqref{Pn}; any $U$ which confines the positive particles will disperse the negative particles, and vice versa. 

On $V$ we temporarily assume (for the sake of the introduction) that it is even, that it has a singularity at $0$, that $V' < 0 < V''$ on $(0,\infty)$, and that it is integrable on $\R$. The evenness assumption implies that the force which $x_i$ exerts on $x_j$ equals the negative of the force which $x_j$ exerts on $x_i$. The singularity at $0$ allows us, shortly before collision, to neglect any external or nonlocal forces. By $V' < 0$ on $(0,\infty)$ we have that particles of the same sign repel and particles of opposite sign attract. By $V'' > 0$ on $(0,\infty)$, the particle interaction force $|V'|$ decreases as the distance between the particles increases. We expect that this monotonicity prevents particles from clustering together without colliding. 
Finally, the integrability is needed for the scaling with $\alpha_n$.

The parameter $\alpha_n$ determines the part of the potential $V$ (the singularity, the tail, or both) which is relevant for the dynamics as $n$ gets large. The rescaling $V_{\alpha_n}$ is such that it conserves the integral, i.e.
\begin{equation*} 
  \int_\R V_{\alpha_n}(x) \, dx = \int_\R V(x) \, dx
  \qquad \text{for all } \alpha_n > 0.
\end{equation*}
We note that the scaling in \eqref{Pn} with respect to $n$ is such that the particle distances $|x_i - x_j|$ typically range from $O(\frac1n)$ for $|i-j| = O(1)$ to $O(1)$ for $|i-j| = O(n)$. Hence, when $\alpha_n = O(1)$ as $n \to \infty$, the relevant part of $V(x)$ is a neighborhood of size $O(1)$ around $x=0$, and thus the tail behavior of $V$ is of no importance. However, when $\alpha_n = O(n)$, then the singularity of $V$ is only briefly visited prior to collisions, but is otherwise irrelevant. If $1 \ll \alpha_n \ll n$, then both the singularity and the tail behavior of $V$ contribute to the dynamics. Indeed, in previous studies 
\cite{GeersPeerlingsPeletierScardia13,
VanMeursMuntean14,
vanMeurs18}
it has been shown that different asymptotic behavior of $\alpha_n$  leads to different limiting models for the particle density. Similarly to those studies we separate the following three cases as $n \to \infty$:
\begin{align} \label{alphanm}
  \begin{cases}
    \alpha_n \to \alpha
    &\text{if } m=1 \\
    1 \ll \alpha_n \ll n
    &\text{if } m=2 \\
    \frac{ \alpha_n } n \to \beta
    &\text{if } m=3.
  \end{cases} 
\end{align}
Here, $\alpha, \beta > 0$ are positive constants. We reserve the symbol $m$ as the label for separating these three cases. Note that by focussing on the cases in \eqref{alphanm} we leave out the regimes $\alpha_n \ll 1$ and $\alpha_n \gg n$. We do this for two reasons. First, these regimes require a detailed description of either the singularity or the tail of $V$. Second, we are not aware of literature on these regimes other than \cite{GeersPeerlingsPeletierScardia13}. Each of the three regimes in \eqref{alphanm}, however, appears in various studies; see Section \ref{s:intro:lit}. 

The two main questions in this paper are:
\begin{enumerate}
  \item For which $V$ is \eqref{Pn} well-posed? 
  \item For which $V$ does the limit passage of \eqref{Pn} hold? 
\end{enumerate}
When answering these questions, we strive for a proper balance between, on the one hand, keeping the assumptions on $V$ as weak as possible, and, on the other hand, keeping the assumptions simple and preventing the proofs from becoming too technical.

Our answers are given respectively by Theorems \ref{t:Pn} and \ref{t}, which are the main two results of this paper. The assumptions on $V$ for Theorems \ref{t:Pn} and \ref{t} are given by respectively Assumptions \ref{a:UV:fg} and \ref{a:UV}. In Section \ref{s:disc:weak:V} we discuss how these assumptions can be weakened. Here, we note that the assumptions on $V$ allow for a large class of singularities and tail behaviors. In fact, the well-posedness of \eqref{Pn} requires hardly any bound on the singularity or the tail. In particular, no integrability is required. This is interesting, because the strength of the singularity of $V$ is important for the shape of the particle trajectories prior to collision (see Example \ref{ex:Pn=2:expl} below); the parabolic shape observed in \cite{VanMeursPeletierPozar22} is characteristic only for a logarithmic singularity. The limit passage in Theorem \ref{t} does require bounds on the strength of the singularity and the fatness of the tails, but these bounds are hardly stronger than $V$ having to be integrable on $\R$.

Next we describe the limiting PDEs for the signed particle density $\kappa$. These PDEs are formally given by
\begin{align} \label{P1} \tag{$P^1$}
  \partial_t \kappa &= \partial_x (|\kappa| (V_\alpha' * \kappa) + |\kappa| U') && \text{if } m=1, \\\label{P2} \tag{$P^2$}
  \partial_t \kappa &= \partial_x \big( f_2(\kappa) \partial_x \kappa + |\kappa| U' \big) && \text{if } m=2, \\\label{P3} \tag{$P^3$}
  \partial_t \kappa &= \partial_x \big( f_3(\kappa) \partial_x \kappa + |\kappa| U' \big) && \text{if } m=3, 
\end{align}
where 
\begin{equation} \label{f23}
  f_2(\kappa) := \|V\|_{L^1(\R)} |\kappa|, \qquad
  f_3(\kappa) := \left\{ \begin{aligned}
    &\frac{\beta^3}{\kappa^2} \Psi \Big( \frac \beta \kappa \Big)
    &&\text{if } \kappa \neq 0  \\
    &0
    &&\text{if } \kappa = 0
  \end{aligned} \right.
\end{equation}
and  
\begin{equation} \label{Psi}
  \Psi (x) := \sum_{k=1}^\infty k^2 V''(kx)\
  \qquad \text{for all } x \in \R \setminus \{0\}. 
\end{equation}
Note that \eqref{P1} is a generalization of \eqref{P:vMPP} (`$*$' denotes the convolution on $\R$). Equations \eqref{P2} and \eqref{P3}, however, are quite different from \eqref{P1}: they are \textit{local} PDEs. Similar, but simpler, PDEs were  obtained in \cite{VanMeursMuntean14} for the case of single-sign particles (in this case the framework of Wasserstein gradient flows of convex energies applies). To see where the localness comes from, we recall that the scaling of $V$ by $\alpha_n$ preserves the integral. Moreover, if $\alpha_n \to \infty$, then $V_{\alpha_n} \weakto \|V\|_{L^1(\R)} \delta_0$ in the weak topology of measures. This shows that the nonlocal interactions vanish in the tails of $V$ as $n \to \infty$, and motivates the constant $\|V\|_{L^1(\R)}$ in $f_2$. The expression of $f_3$ is more complicated. Roughly speaking, in the case $m=3$, each $k$th-neighbor interaction in \eqref{Pn} is of size $O(1)$; its contribution in the limit $n \to \infty$ leads to the $k$th term in \eqref{Psi}. 

Our proof methods for Theorems \ref{t:Pn} and \ref{t} are inspired by those in \cite{VanMeursPeletierPozar22}: for Theorem \ref{t:Pn} we use a careful analysis of the singularities of the ODE system and for Theorem \ref{t} we rely on a modified notion of viscosity solutions. At several steps in the proof, however, we require a completely different approach from that in \cite{VanMeursPeletierPozar22}. We explain this difference after introducing the mathematical setting: see the text below Corollary \ref{c:Pn} and the text below Theorem \ref{t:conv}.

\subsection{Applications and literature}
\label{s:intro:lit}

In \cite{VanMeursPeletierPozar22} three applications are given for establishing well-posedness and the limit $n \to \infty$ of \eqref{Pn:vMPP}: systems of vortices (see \cite{Serfaty07II,SmetsBethuelOrlandi07}), arrays of dislocations \cite{ForcadelImbertMonneau09,GarroniVanMeursPeletierScardia19,HeadLouat55,VanMeursMorandotti19}, and in general a contribution to interacting particle systems with multiple species; see e.g.\ 
\cite{BerendsenBurgerPietschmann17,
DiFrancescoEspositoSchmidtchen20ArXiv,
DiFrancescoFagioli13,
DiFrancescoFagioli16,
Zinsl16}. Obviously, with our generalization \eqref{Pn} of \eqref{Pn:vMPP} we aim to contribute to the third motivation. In addition, the freedom in the choice of $V,U,\alpha_n$ provides various new important applications. Here we mention several of these. In Section \ref{s:disc} we show that Theorems \ref{t:Pn} and \ref{t} capture almost all of these applications. 

\paragraph{Dislocation structures.} Dislocations are defects in the atomic lattice of metals. The textbooks \cite{HirthLothe82,HullBacon11} treat dislocations as the main topic. Here, we provide a brief description. Under mechanical loading, dislocations can move through the atomic lattice. If they do so in large quantities, then the metal deforms plastically. A long-standing open question is to understand precisely how plasticity on a macroscopic scale arises from the microscopic dynamics of dislocations. Under idealistic modeling assumptions, dislocation dynamics can be described in a one-dimensional setting by \eqref{Pn:vMPP}, where $x_i$ is the position of a single dislocation and $b_i$ is its orientation. The limit passage $n \to \infty$ describes the connection between the micro- and macroscopic scale. Hence, the results in \cite{VanMeursPeletierPozar22} have contributed to the understanding of plasticity. 
  
Other than single dislocations, the points $x_i$ may also represent dislocation \textit{structures} such as dipoles (see \cite{HallChapmanOckendon10,ChapmanXiangZhu15}) or dislocation walls (i.e.\ vertically periodic arrays of dislocations; see \cite{RoyPeerlingsGeersKasyanyuk08,GeersPeerlingsPeletierScardia13,VanMeursMuntean14,vanMeurs18}). Both settings require the generalization of \eqref{Pn:vMPP} to \eqref{Pn}: in the case of dipoles, $V(x) = 1/x^2$, and in the case of dislocation walls,
\begin{equation} \label{Vwall}
  V(x) = x \coth x - \log |2 \sinh x|,  
\end{equation}
which has a logarithmic singularity at $0$ and exponentially decaying tails. In the literature to date on these dislocation structures the setting is restricted to avoid collisions. The results in the present paper lift these restrictions, which contributes towards understanding plasticity. 

\paragraph{Riesz potentials.} The (extended) Riesz potential with parameter $a > -1$ is given by 
\begin{equation} \label{VR}
  V(x) =
  \left\{ \begin{aligned}
    &-|x|^{-a}
    && -1 < a < 0 \\
    & -\log|x|
    && a=0 \\
    & |x|^{-a}
    && 0 < a.
  \end{aligned} \right.
\end{equation}
Here, we have extended the usual range of $a$ from $(0,1)$ to $(-1,\infty)$. The choice of this extension to negative $a$ is natural from the observation that 
\[
  -V'(x) = c_a \frac{ \sign(x) }{|x|^{1+a}} 
\]  
for some explicit constant $c_a > 0$. For $a \in (-1,1)$ Riesz potentials are connected with the fractional Laplacian:
  \[
    V' * \kappa = -C_a (-\Delta)^{\tfrac{1+a}2} u,
  \]
where $C_a > 0$ is a certain constant and $u = \int \kappa(x) \, dx$ as in \eqref{HalfL}. For applications of particle systems such as \eqref{Pn} with Riesz potentials to physics and approximation theory we refer to \cite{DauxoisRuffoArimondoWilkens02,BorodachovHardinSaff19}. 
   
  In the mathematics literature the Riesz potential appears in the context of \eqref{Pn} or \eqref{P1} in the following papers: 
\begin{itemize} 
    \item in \cite{BilerKarchMonneau10}, for any $a \in (-1,1)$ a viscosity solution concept for \eqref{P1} is built,
    \item in \cite{Duerinckx16,NguyenRosenzweigSerfaty21,DuerinckxSerfaty20}, for any $a \in [0,1)$  the limit $n \to \infty$ of \eqref{Pn} is established also in higher dimensions, but for the single-sign case,
    \item in the paper series by Patrizi, Valdinoci et al.\ (see e.g.\ \cite{PatriziValdinoci15,
  PatriziValdinoci16,
  PatriziValdinoci17,
  PatriziSangsawang23,
  VanMeursPatrizi22ArXiv}), for any $a \in (-1,1)$ a regularization of \eqref{Pn} is considered in the form of a phase-field model. 
\end{itemize}  
The third setting is closely related to that in this paper. In that setting the regularization of \eqref{Pn} is constructed by smearing out the particles (considered as point masses) over a small distance $\delta > 0$ and then by integrating the resulting density in space, similarly to the manner in which $u$ is obtained from $\kappa$ in \eqref{HalfL}. This results in a phase-field model in which phase transitions (up or down) correspond to (positive or negative) particles. For $a = 0$ the phase-field model is known in physics as the Peierls-Nabarro model for dislocations. Together with the well-posedness results in \cite{VanMeursPeletierPozar22} for the special case $a=0$ it became possible to connect the Peierls-Nabarro model to both \eqref{Pn:vMPP} (see \cite{VanMeursPatrizi22ArXiv}) and to \eqref{P:vMPP} (see \cite{PatriziSangsawang23}) by means of proving limit passages as $n \to \infty$ or $\delta \to 0$. We expect that the well-posedness result in the present paper (see Theorem \ref{t:Pn}) is of equal importance for the extension of these connections to any $a \in (-1,1)$. 
    
\paragraph{The external potential $U$.} We have added the potential $U$ to \eqref{Pn} for two reasons. First, with the term $U$ we can capture more phenomena, such as a driving force which pushes a group of positive particles and a group of negative particles in opposite directions. In particular, in the application to dislocations, $U$ can capture an external loading applied to the metal. Such a term appears in several of the papers on dislocations mentioned above. Second, in the paper series by Patrizi, Valdinoci et al.\ cited above, the key method of proof is the construction of viscosity sub- and supersolutions to the phase-field equation. This construction relies on a perturbed version of \eqref{Pn:vMPP} in which not only the initial particle positions are perturbed, but also an additional, small potential $U(x) = \delta x$ is added. Our well-posedness result on \eqref{Pn} (see Theorem \ref{t:Pn}) shows that this perturbed particle system converges to the unperturbed system \eqref{Pn} as the perturbation parameter $\delta $ tends to $0$. The convergence is uniform in time across collisions, which provides a stronger tool for constructing the viscosity solutions.
 
\paragraph{The parameter $\alpha_n$.} The case $m=1$ in which $\alpha_n \to \alpha > 0$ (often simply $\alpha_n = \alpha = 1$) is most common in the literature on particle systems; the corresponding limit $n \to \infty$ is called the mean-field limit. However, the case $\alpha_n \to \infty$ also appears in various studies. We mention three of these.

First, for the application to dislocation walls (see \eqref{Vwall}) we show in Section \ref{s:disc} that Theorem \ref{t} guarantees the convergence of \eqref{Pn} to ($P^m$) in all three scaling regimes in \eqref{alphanm}. This extends the results in  \cite{VanMeursMuntean14} where a similar limit passage has been established for single-sign particles. 

Second, \cite{Oelschlager90} derives the limit $n \to \infty$ of \eqref{Pn} in the cases $m=2,3$ in higher spatial dimensions, but only in the single-sign case and for nonsingular $V$. Within the one-dimensional setting, the present paper is a large extension of this result. 

Third, the case $m=3$ coincides with the setting in studies on arrays of atoms; see e.g.\ 
\cite{BraidesGelli04,
BlancLeBrisLions07,
HallHudsonVanMeurs18,
Hudson13,
HudsonOrtner14,
SchaeffnerSchloemerkemper18} for several of such studies.

\subsection{Organization of the paper}
\label{s:intro:org}

In Section \ref{s:ODE} we treat the well-posedness of \eqref{Pn} (Theorem \ref{t:Pn}) and in Section \ref{s:t} we establish the convergence of \eqref{Pn} to each of the three PDEs \eqref{P1}, \eqref{P2} or \eqref{P3} (Theorem \ref{t}). Section \ref{s:disc} contains a discussion of these two theorems. Sections \ref{s:HJ} and \ref{s:conv} contain the framework and the essence of the convergence statement in Theorem \ref{t}. In Section \ref{s:HJ} we provide a proper meaning to the integrated versions of the PDEs \eqref{P1}, \eqref{P2} and \eqref{P3} in terms of viscosity solutions. We also recast an integrated version of \eqref{Pn} in the framework of viscosity solutions such that the convergence statement in Theorem \ref{t} can be translated to this framework. This translated convergence statement is given by Theorem \ref{t:conv} and proven in Section \ref{s:conv}.

\newpage 

\section{Well-posedness of the ODE system \eqref{Pn} }
\label{s:ODE}

In this section we state and prove Theorem \ref{t:Pn} on the well-posedness of the particle system \eqref{Pn} for fixed $n \geq 2$. Before introducing \eqref{Pn} rigorously, we make several a priori observations:
\begin{itemize}
  \item As mentioned before, since $n$ is fixed, $\alpha_n > 0$ is a fixed constant;
  \item In between collisions, the system of ODEs in \eqref{Pn} is well-defined. In fact, it is the gradient flow of the energy
  \begin{equation*} 
    E(\bx; \bb) = \frac1{n^2} \sum_{i>j} b_i b_j V_{\alpha_n} (x_i - x_j) + \frac1n \sum_i b_i U(x_i). 
    \end{equation*}
    However, this gradient flow structure is of limited use. Indeed, if $V$ is singular at $0$, then along a solution $\bx(t)$ of \eqref{Pn} the energy diverges to $-\infty$ as $t$ approaches a time at which two particles of opposite sign collide;
  \item Since the variational framework is of limited use, we rewrite \eqref{Pn} in terms of the forces
\[
  f := -V_{\alpha_n}' 
  \qand 
  \fext := -U'.
\]
This yields
\begin{align} \label{Pn:ito:forces} \tag{$P_n'$}
  \frac d{dt} x_i 
  = \frac1{n} \sum_{j \neq i} b_i b_j f (x_i - x_j) + b_i \fext(x_i);
\end{align}
  \item \eqref{Pn:ito:forces} is invariant under relabeling the indices of particles with the same sign; 
  \item If $g = 0$, then \eqref{Pn:ito:forces} is invariant under swapping all signs (i.e.\ replacing $b_i$ with $-b_i$ for all $i$). We will often use this to fix the sign of one particle;
  \item The net charge $\sum_{i=1}^n b_i$ of the surviving particles is conserved under \eqref{Pn:ito:forces}.
\end{itemize} 
In addition, it is instructive to consider the case $n=2$ as we did in \eqref{sol:n2:vMPP}:

\begin{ex}[Derivative of the Riesz potential] \label{ex:Pn=2:expl}
Consider the interaction force given by $f(x) = \frac1{2+a} \sign (x) |x|^{-1-a}$ for some $a \in (-1, \infty)$. Take $g = 0$, $n=2$ and $b_1 b_2 = -1$. Then, the solution to \eqref{Pn} is given by
\begin{align} \notag
  x_2(t) + x_1(t) 
  &= x_2(0) + x_1(0), \\ \label{power:law}
  x_2(t) - x_1(t) 
  &= (c_0 - t)^{\tfrac1{2+a}}, \qquad c_0 := (x_2(0) - x_1(0))^{2+a}.
\end{align}
\end{ex}

For general $n$ we do not have sufficient evidence to state that the shape of the particle trajectories close to collision are similar to the power law in Example \ref{ex:Pn=2:expl} for any $f$ with a power-law type singularity. Nevertheless, we show in Theorem \ref{t:Pn} that several such trajectories can be bounded, at least from one side, by the power law in \eqref{power:law}.

Let us introduce some notation:
\begin{itemize}
  \item For a function $f$ of one variable, we set $f(t-) := \lim_{s \uparrow t} f(s)$ and $f(t+) := \lim_{s \downarrow t} f(s)$ as respectively the left and right limit;
  \item $a \vee b := \max \{a,b\}$ and $a \wedge b := \min \{a,b\}$;
  
  \item We reserve $C, c > 0$ to denote generic positive constants which do not depend on the relevant parameters and variables. We think of $C$ as possibly large and $c$ as possibly small. The values of $C,c$ may change from line to line, but when they appear multiple times in the same display their values remain the same. When more than one generic constant appears in the same display, we use $C', C'', c'$ etc.\ to distinguish them. When we need to keep track of certain constants, we use $C_0, C_1$ etc.
\end{itemize}

Next we introduce the assumptions on $g$ and $f$:
\begin{ass}[$g$ and $f$] \label{a:fg}
$g : \R \to \R$ is Lipschitz continuous. $f$ satisfies
\begin{enumerate}[label=(\roman*)]
  \item \label{a:fg:odd} $f : \R \setminus \{0\} \to \R$ is odd;
  \item \label{a:fg:reg} $f \in C^2((0,\infty))$;
  \item \label{a:fg:mon} (monotonicity). $f \geq 0$, $f' \leq 0$ and $f'' \geq 0$ on $(0,\infty)$;
  \item \label{a:fg:singLB} (singularity lower bound). $x f'(x) \to -\infty$ as $x \downarrow 0$;
  \item \label{a:fg:singUB} (singularity upper bound). There exists $a \in (-1,\infty)$ and $C > 0$ such that $f(x) \leq \dfrac C{x^{1+a}}$ for all $x \in (0, 1)$.
\end{enumerate} 
\end{ass} 

For future reference, we translate Assumption \ref{a:fg} in terms of $U$ and $V$:

\begin{ass}[$U$ and $V$] \label{a:UV:fg}
$U \in C^1(\R)$ is such that $U'$ is Lipschitz continuous. $V$ satisfies:
\begin{enumerate}[label=(\roman*)]
  \item $V : \R \setminus \{0\} \to \R$ is even;
  \item $V \in C^3((0,\infty))$;
  \item (monotonicity). $V' \leq 0$, $V'' \geq 0$ and $V''' \leq 0$ on $(0,\infty)$;
  \item (singularity lower bound). $x V''(x) \to \infty$ as $x \downarrow 0$;
  \item (singularity upper bound). There exists $a \in (-1,\infty)$ and $C > 0$ such that $|V'(x)| \leq \dfrac{C}{x^{1+a}}$ for all $x \in (0, 1)$.
\end{enumerate}
\end{ass} 

In the remainder of this section we work solely with Assumption \ref{a:fg}. First, we mention several consequences of Assumption \ref{a:fg}. From \ref{a:fg:mon} and \ref{a:fg:singLB} we obtain that $f(0+) = \infty$ and that $f$ and $f'$ are bounded on $[\delta,\infty)$ for any $\delta > 0$. Hence, in \ref{a:fg:singUB} the range of $x$ away from $0$ is of little importance. \comm{p.107 top}  

Next we motivate Assumption \ref{a:fg}. By the Lipschitz continuity of $f$ and $g$, \eqref{Pn:ito:forces} has a unique solution up to the first time when particles get arbitrarily close. We will use $f(0+) = \infty$ to show that the interaction forces between particles close to collision outweigh the contributions of $g$ and the other particles. The upper bound on the singularity in \ref{a:fg:singUB} is actually not necessary for existence and uniqueness of solutions; we impose it to get several explicit estimates on the solutions in terms of the exponent $a$.

The monotonicity assumption goes one step further than convexity of $U$. To motivate it, we first note from  \eqref{Pn:ito:forces} that the force field exerted by a particle $x_j$ on another particle $x_i > x_j$ is
\[
  \frac{b_i b_j}n f(x_i - x_j).
\]
The zeroth-order monotonicity (i.e.\ $f \geq 0$ on $(0,\infty)$) implies that the sign of this force does not depend on the distance between $x_j$ and $x_i$. An important consequence of the first-order monotonicity (i.e.\ $f' \leq 0$ on $(0,\infty)$) is illustrated in Figure \ref{fig:dip} (left): the force exerted by the negative particle tends to increase the distance between two positive particles on the right, independently of the distances between the particles. Finally, by the second-order monotonicity (i.e.\ $f'' \geq 0$ on $(0,\infty)$) this observation also holds when the negative particle is replaced by a dipole; see Figure \ref{fig:dip} (right) and Lemma \ref{l:f2}.

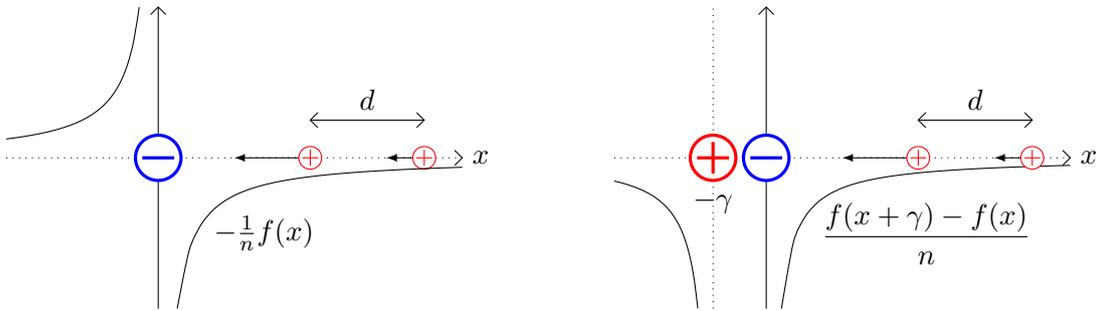
\begin{figure}[ht]
\centering
\begin{tikzpicture}[scale=1, >= latex]
\def \sqtwo {1.414}
\def \rr {0.15} 
\def \a {0.1};
\def \r {0.3};

\draw (0,-2) -- (0,2); 
\draw (-\a, 2-\a) -- (0,2) -- (\a, 2-\a);
\draw[dotted] (-2,0) -- (4,0) node[right] {$x$};
\draw (4-\a, -\a) -- (4,0) -- (4-\a, \a);

\filldraw[draw=blue, very thick, fill=white] (0, 0) circle (\r); 
\draw[blue, very thick] (-.7*\r, 0) -- (.7*\r, 0);

\draw (2,.5) --  (3.5,.5) node[midway, above]{$d$};
\draw (3.5-\a, .5-\a) -- (3.5,.5) -- (3.5-\a, .5+\a);
\draw (2+\a, .5-\a) -- (2,.5) -- (2+\a, .5+\a);

\begin{scope}[shift={(2,0)},scale=1] 
\draw[<-] (-1,0) -- (0,0);
\filldraw[draw=red, fill=white] (0, 0) circle (\rr); 
\draw[red] (-.7*\rr, 0) -- (.7*\rr, 0);
\draw[red] (0, -.7*\rr) -- (0, .7*\rr);
\end{scope}

\begin{scope}[shift={(3.5,0)},scale=1] 
\draw[<-] (-.5,0) -- (0,0);
\filldraw[draw=red, fill=white] (0, 0) circle (\rr); 
\draw[red] (-.7*\rr, 0) -- (.7*\rr, 0);
\draw[red] (0, -.7*\rr) -- (0, .7*\rr);
\end{scope}

\draw[domain=.25:4, smooth] plot (\x,{-.5/\x});
\draw[domain=-2:-.25, smooth] plot (\x,{-.5/\x});
\draw (.6,-1) node[right] {$-\frac1nf(x)$};

\begin{scope}[shift={(8,0)},scale=1] 
\draw (0,-2) -- (0,2); 
\draw[dotted] (-.7,-2) -- (-.7,2);
\draw (-\a, 2-\a) -- (0,2) -- (\a, 2-\a);
\draw[dotted] (-2,0) -- (4,0) node[right] {$x$};
\draw (4-\a, -\a) -- (4,0) -- (4-\a, \a);

\draw[domain=.2:4, smooth] plot (\x,{-.4/\x});
\draw[domain=-2:-.9, smooth] plot (\x,{.4/(\x + .7)});
\draw (.6,-1) node[right] {$\dfrac{ f(x+\gamma) - f(x) }n$};

\filldraw[draw=blue, very thick, fill=white] (0, 0) circle (\r); 
\draw[blue, very thick] (-.7*\r, 0) -- (.7*\r, 0);

\begin{scope}[shift={(-.7,0)},scale=1] 
\filldraw[draw=red, very thick, fill=white] (0, 0) circle (\r); 
\draw[red, very thick] (-.7*\r, 0) -- (.7*\r, 0);
\draw[red, very thick] (0, -.7*\r) -- (0, .7*\r);
\draw (0,-\r) node[below]{$-\gamma$};
\end{scope}

\draw (2,.5) --  (3.5,.5) node[midway, above]{$d$};
\draw (3.5-\a, .5-\a) -- (3.5,.5) -- (3.5-\a, .5+\a);
\draw (2+\a, .5-\a) -- (2,.5) -- (2+\a, .5+\a);

\begin{scope}[shift={(2,0)},scale=1] 
\draw[<-] (-1,0) -- (0,0);
\filldraw[draw=red, fill=white] (0, 0) circle (\rr); 
\draw[red] (-.7*\rr, 0) -- (.7*\rr, 0);
\draw[red] (0, -.7*\rr) -- (0, .7*\rr);
\end{scope}

\begin{scope}[shift={(3.5,0)},scale=1] 
\draw[<-] (-.5,0) -- (0,0);
\filldraw[draw=red, fill=white] (0, 0) circle (\rr); 
\draw[red] (-.7*\rr, 0) -- (.7*\rr, 0);
\draw[red] (0, -.7*\rr) -- (0, .7*\rr);
\end{scope}
\end{scope}
\end{tikzpicture} \\
\caption{Left: the graph sketches the force field (given by $-\frac1n f(x)$) that a negative particle at $x=0$ exerts on $\R$. The effect of this force field is further illustrated by force vectors acting on two positive particles in the vicinity. Right: similar as the left figure, but now for a dipole (with force field $\frac1n f(x+\gamma) - \frac1n f(x)$). Both figures: the  force field yields a positive contribution to $\frac d{dt} d$. }
\label{fig:dip}
\end{figure}

\begin{lem} \label{l:f2}
Let $f$ satisfy Assumption \ref{a:fg}. For all $\gamma > 0$, the function $x \mapsto f(x+\gamma) - f(x)$ is nondecreasing on $(0,\infty)$.
\end{lem}

\begin{proof}
For any $y > x > 0$, we have by $f'' \geq 0$ on $(0,\infty)$ that
\begin{align*}
  \big( f(y+\gamma) - f(y) \big)
  - \big( f(x+\gamma) - f(x) \big)
  &= \int_0^{\gamma} f'(y + z) \, dz
    - \int_0^{\gamma} f'(x + z) \, dz \\
  &= \int_0^{\gamma} \int_{x+z}^{y+z} f''(\xi) \, d\xi dz
  \geq 0.
\end{align*}
\end{proof}

Next we construct a rigorous solution concept for \eqref{Pn} which allows for collisions. First we encode the removal of particles. There are several choices on how to do this. We follow the approach in \cite{VanMeursPeletierPozar22} to switch $b_i$ from $\pm 1$ to $0$ at the time when $x_i$ gets annihilated. Consequently, we consider $\bb$ as a time-dependent unknown. We call a particle $x_i$ at time $t$ \textit{charged} if $|b_i(t)| = 1$ and \textit{neutral} if $b_i(t) = 0$. We call two particles $i < j$ \emph{neighbors}  if they are charged and any particle in between them is neutral, i.e.\ 
\[ b_i(t) b_j(t) \neq 0
\quad \text{and} \quad 
\forall \, i < k < j : b_k(t) = 0. \]
Note that for neutral particles, the right-hand side in \eqref{Pn:ito:forces} vanishes (we employ the convention to define $f(0) := 0$), and thus they remain stationary after collision. In addition, note that the presence of neutral particles does not affect the force exerted on the charged particles. Finally, we do not consider any collision rule when a charged particle happens to move across a neutral particle.

Next we recall from \cite{VanMeursPeletierPozar22} the state space $\cZ_n$ for the solution pair $(\bx, \bb)$. Let the particles initially be numbered from left to right. For charged particles the dynamics preserves the numbering; this is a consequence of the annihilation rule. When adding the neutral particles, however, there is no reason for the numbering to be preserved. This motivates the definition
\begin{equation} \label{Zn}
  \cZ_n := \{ (\bx, \bb) \in \R^n \times \{-1,0,+1\}^n : \text{if } i < j \text{ and } b_i b_j \neq 0, \text{ then } x_i < x_j \}.
\end{equation}
Note that for any $(\bx, \bb) \in \cZ_n$ no two charged particles can be at the same position.

The solution concept to \eqref{Pn} is the same as in \cite{VanMeursPeletierPozar22}. It is as follows:

\begin{defn}[Solution to $(P_n)$] \label{d:Pn}
Let $n \geq 2$, $T > 0$ and $(\bx^\circ, \bb^\circ) \in \cZ_n$. $(\bx, \bb) : [0,T] \to \cZ_n$ is a solution to $(P_n)$ if there exists a finite set $S \subset (0,T]$ such that
\begin{enumerate}[label=(\roman*)]
  \item (Regularity) $\bx \in C([0,T]) \cap C^1((0,T) \setminus S)$, and $b_1, \ldots, b_n : [0,T] \to \{-1, 0, 1\}$ are right-continuous; 
  \item (Initial condition) $(\bx(0), \bb(0)) = (\bx^\circ, \bb^\circ)$;
  \item (Annihilation rule) \label{d:Pn:ann:rule} Each $b_i$ jumps at most once. If $b_i$ jumps at $t \in [0,T]$, then $t \in S$, $|b_i(t-)| = 1$ and $b_i(t) = 0$. Moreover,
  \begin{equation} \label{annih:rule}
    \sum_{i: x_i(t) = y} b_i (t) = \sum_{i: x_i(t) = y} b_i(t-)
    \qquad \text{for all } t \in (0,T] \text{ and all } y \in \R;
  \end{equation}
  \item (ODE of $\bx$) On $(0,T) \setminus S$, $\bx$ satisfies the ODE in \eqref{Pn}.
\end{enumerate}
\end{defn}

We take a moment to parse Definition \ref{d:Pn}. We call a point $(\tau, y) \in S \times \R$ a  \emph{collision point} if the second sum in \eqref{annih:rule} contains at least two non-zero summands. We call the time $\tau$ of a collision point a \emph{collision time}. The set of all collision times $\{ \tau_1, \ldots, \tau_K \}$ is finite, where $K \leq n^+ \wedge n^-$ and $n^\pm$ are the numbers of positive/negative particles at time $0$ as determined by $\bb^\circ$. From Theorem \ref{t:Pn} it will turn out that the minimal choice for $S$ is $\{ \tau_1, \ldots, \tau_K \}$.

In Definition \ref{d:Pn} annihilation is encoded by the combination of the annihilation rule in \ref{d:Pn:ann:rule} and the requirement that $(\bx(t), \bb(t)) \in \cZ_n$. Indeed, Definition \ref{d:Pn}\ref{d:Pn:ann:rule} limits the choice of jump points for $b_i$, while the separation of particles implied by $(\bx(t), \bb(t)) \in \cZ_n$ requires particles to annihilate upon collision.

\begin{rem}[Uniqueness modulo relabeling]
\label{r:unique-modulo-renumbering}
As we shall see in Theorem~\ref{t:Pn} below, solutions according to Definition~\ref{d:Pn} are unique, but only  up to relabeling of particle indices. The nonuniqueness of the labeling can easily be seen from Figure \ref{fig:trajs}. In that figure, at the three-particle collision between $x_7$, $x_8$ and $x_9$, there is a choice for the surviving particle to be $x_7$ or $x_9$. Either choice will lead to a different solution in terms of Definition \ref{d:Pn}. However, for either choice the union of trajectories $x_i$ is the same.  
\end{rem}

\medskip
  
In preparation for stating Theorem \ref{t:Pn} on the well-posedness of \eqref{Pn} and properties of its solution, we introduce $d^\pm$. Given $(\bx, \bb) \in \cZ_n$, we set $d^+$ as the smallest distance between any two neighboring positive particles. Analogously, we define $d^-$ for the negative particles. Note that
\begin{equation} \label{dpm}
  d^\pm := \inf \big\{ x_j - x_i : i < j \text{ are neighbors and } b_i = b_j = \pm 1 \big\} \in (0,\infty].
\end{equation}

\begin{thm}[Properties of $(P_n)$] \label{t:Pn}
Let $n \geq 2$, $T > 0$ and $(\bx^\circ, \bb^\circ) \in \cZ_n$. 
Let $g$ and $f$ satisfy Assumption \ref{a:fg} for some $a \in (-1,\infty)$. Then $(P_n)$ has a solution $(\bx, \bb)$ with initial datum $(\bx^\circ, \bb^\circ)$ in the sense of Definition \ref{d:Pn}. This solution is unique modulo relabeling (see Remark~\ref{r:unique-modulo-renumbering}). 
Moreover, setting $S$ as the set of all collision times of the solution $(\bx, \bb)$, the following properties hold: 
\begin{enumerate}[label=(\roman*)]
  \item \emph{(Lower bound on distance between neighbors of equal sign)}. There exists a constant $c > 0$ such that \label{t:Pn:LB:d}
   \[ 
     d^\pm(t) \geq c \qquad  \text{for all } t \in [0,T];
     \]
   \item \emph{(Lower bound on distance between neighbors of opposite sign)}.  \label{t:Pn:LB:pm}
   Let $i < j$ be neighboring particles with opposite sign at time $t_0 \geq 0$. Then there exist constants $C, c > 0$ such that
\begin{align}
\label{opposite-sign-lower-bound} 
x_j(t) - x_i(t) \geq \big( c - C (t - t_0) \big)^{\tfrac1{2+a}}
\end{align}
for all $t_0 \leq t \leq t_0 + c/C$;

  \item \emph{(Upper bound prior to collision)}. \label{t:Pn:C12}
   For any $\tau \in S$ and any particles $i < j$ which collide at $\tau$ and are neighbors prior to $\tau$, there exists a constant $C > 0$ such that 
   $$ x_j(t) - x_i(t) \leq C (\tau - t)^{\tfrac1{2+a}} 
      \quad \text{for all } t \in [0, \tau]; $$ 
   
   \item \emph{(Lower bound prior to collision)}. \label{t:Pn:LB:col}
   Assume that there exist $a' \in (-1, a]$ and $c', \e_0 > 0$ such that $|x f'(x)| \geq c' / x^{1 + a'}$ for all $x \in (0, \e_0)$. Then, for each collision point $(\tau, y)$, there exist a constant $c > 0$ and indices $i < j$ such that $x_i(\tau) = x_j(\tau) = y$, $b_i(\tau-) \neq 0$, $b_j(\tau-) \neq 0$ and 
   \begin{align*}
     x_j(t) - y &> c (\tau - t)^{\tfrac1{2+a'}}, \\
     x_i(t) - y &< -c (\tau - t)^{\tfrac1{2+a'}}
   \end{align*}
   for all $t < \tau$ large enough; 
  
  \item \emph{(Stability with respect to $\bx^\circ$ and $\fext$)}. \label{t:Pn:stab}
Let $S = \{ \tau_1, \ldots, \tau_K \}$ with $ 0 := \tau_0 < \tau_1 < \ldots < \tau_K \leq \tau_{K+1} =: T$. Let $\fext^\sigma : \R \to \R$ be Lipschitz continuous with $\| \fext^\sigma - \fext \|_\infty \to 0$ as $\sigma \to 0$. Let $(\bx^{\sigma,\circ}, \bb^\circ) \in \cZ_n$ be such that $\bx^{\sigma,\circ} \to \bx^\circ$ as $\sigma \to 0$. Let $(\bx^\sigma, \bb^\sigma)$ be the solution of $(P_n)$ with initial data $(\bx^{\sigma,\circ}, \bb^\circ)$ and external force $\fext^\sigma$. Then, for each $k = 0,\ldots,K$, there exists a relabeling of $(\bx^{\sigma}, \bb^\sigma)$ (which may depend on $k, \sigma$) such that
\begin{itemize}
   \item $\bx^\sigma \to \bx$ in $C([\tau_k, \tau_{k+1}]; \R^n)$ as $\sigma \to 0$, and
   \item $\bb^\sigma = \bb$ on any $[t_0, t_1] \subset (\tau_k, \tau_{k+1})$ for all $\sigma$ small enough.
 \end{itemize}  
\end{enumerate} 
\end{thm} 

Before proving Theorem \ref{t:Pn} we give five comments. First, we stress that the constants $C, c$ appearing in Theorem \ref{t:Pn} may depend on $n, T, \bx^\circ, \bb^\circ, f, g$. Second, similar to \cite[Corollary 2.5]{VanMeursPeletierPozar22}, we observe that Theorem \ref{t:Pn}\ref{t:Pn:LB:d} implies Corollary \ref{c:Pn}. 

\begin{cor}[Multiple-particle collisions]\label{c:Pn}
Let $(\tau, y) \in S \times \R$ be a collision point, and let $I$  be the set of indices of all particles which collide at $(\tau, y)$, i.e.
  \begin{equation*} 
    I := \{ i : x_i(\tau) = y, \, b_i(\tau-) \neq 0 \}.
  \end{equation*}
Then, prior to collision, any two neighboring particles with indices in $I$ have opposite sign. In particular, 
\begin{equation*}
\Big| \sum_{i \in I} b_i(\tau-) \Big| \leq 1.
\end{equation*}
\end{cor}

Third, Theorem \ref{t:Pn}\ref{t:Pn:LB:pm}--\ref{t:Pn:LB:col} raise the question whether $t \mapsto \bx(t)$ is H\"older continuous with exponent $\frac1{2+a}$. If $f$ is a regular perturbation of the homogeneous force field $f_a (x) = \sign(x) |x|^{-1-a}$, then the answer is affirmative. This is proven in \cite{deJongVanMeurs23ArXiv}, which has been written simultaneously to the present paper, and relies on Theorem \ref{t:Pn} and on the techniques in the proof below. 

The fourth comment is the following remark.

\begin{rem}[Weakened assumptions on $f$] \label{r:Pn:ass:f}
With minor modifications to the proof below, the assumptions on $f$ made in Theorem \ref{t:Pn} can be weakened as follows. In terms of regularity, $f \in W_{\loc}^{2,1}(0,\infty)$ is sufficient. The monotonicity conditions in Assumption \ref{a:fg}.\ref{a:fg:mon} only need to hold locally (on $(0,\delta)$ for some $\delta > 0$), provided that $f'$ is bounded on $[1,\infty)$. 
\end{rem}

Fifth, we comment on the similarities and the key differences of the proof below with that in \cite{VanMeursPeletierPozar22}. One similarity is that the three challenges mentioned below \eqref{sol:n2:vMPP} are overcome by establishing Property \ref{t:Pn:LB:d}. Another similarity is the outline of the proof. The key difference, however, is that \cite{VanMeursPeletierPozar22} crucially relies on the (scaled and signed) moments of $\bx$ given by
\begin{equation*} 
  M_k (\bx) := \frac1k \sum_{i=1}^n x_i^k \qquad
  \text{for } k = 1, \ldots, n. 
\end{equation*}
Indeed, for the particular choices $V(x) = -\log|x|$, $U = 0$ and $\alpha_n = 1$ in \cite{VanMeursPeletierPozar22} it follows from a direct computation that
\begin{equation} \label{pf:va}
  \frac d{dt} M_1(\bx(t)) = 0,
  \qquad \frac d{dt} M_2(\bx(t)) = \frac1n \sum_{i=1}^n \sum_{j=1}^{i-1} b_i(t) b_j(t).
\end{equation}
Note that the right-hand sides are \textit{constant} in between any two consecutive collision times. In particular, the singularity of the interaction force cancels out when computing the time derivative of the moments $M_k$. The derivatives of higher order moments can be bounded in terms of lower order moments. This gives an a priori Lipschitz bound on all moments. In \cite{VanMeursPeletierPozar22} these bounds are exploited in the proofs of most of the properties listed in Theorem \ref{t:Pn} by translating back and forth between the list $\bx$ of particle positions and the list $(M_1(\bx), \ldots, M_n(\bx))$ of moment values.

However, if $f$ has a power-law singularity with exponent greater than $1$, then the singularity in the computation in \eqref{pf:va} does not cancel out, and we do not get automatically a bound on the derivative of the moments. The author did not see a modification of the moment bounds in  \cite{VanMeursPeletierPozar22} which would work for general $f$. Instead, in the proof below we reveal that the monotonicity of $f$ and $f'$ are sufficient for constructing alternative arguments at all places where \cite{VanMeursPeletierPozar22} relies on moment bounds.

Another important difference with \cite{VanMeursPeletierPozar22} is the uniform convergence of $\bx^\sigma$ in Property \ref{t:Pn:stab}. In \cite{VanMeursPeletierPozar22} the uniform convergence is stated only in terms of the moments, which has a limited practical use. The need for a $\sigma$-dependent relabeling comes from the uniqueness issue mentioned in Remark \ref{r:unique-modulo-renumbering}.

\begin{proof}[Proof of Theorem \ref{t:Pn}]
Several parts of the proof below are similar to the proof of \cite[Theorem 2.4]{VanMeursPeletierPozar22}, which is the analogue of Theorem \ref{t:Pn} with $f(x) = \frac1x$ and $g = 0$. We  briefly summarize these parts of the proof, and treat the new additions and modifications in full detail.
\medskip

\noindent \textbf{Uniqueness.} The same argument as in \cite{VanMeursPeletierPozar22} applies; standard ODE theory and the imposed regularity in Definition \ref{d:Pn} yield the uniqueness of $\bx$ up to and including the collision times. The uniqueness of $\bb$ (modulo relabeling) holds by construction (see Definition \ref{d:Pn}\ref{d:Pn:ann:rule}). 
\medskip

\noindent \textbf{Existence and Property \ref{t:Pn:LB:d}.} Standard ODE theory provides the existence of $\bx$ up to the first collision time $\tau := \tau_1$. To extend the solution beyond $\tau$, it is sufficient to prove that the limit $\bx(\tau-)$ exists and that Property \ref{t:Pn:LB:d} holds for all $t \in [0,\tau)$. The statement that these two conditions are sufficient was shown in  \cite{VanMeursPeletierPozar22}; we repeat the argument here. Property \ref{t:Pn:LB:d} implies Corollary \ref{c:Pn}. Then, at each point $y \in \{ x_i(\tau) : 1 \leq i \leq n \}$ Definition \ref{d:Pn}\ref{d:Pn:ann:rule} allows us to choose $\bb(\tau)$ such that there is at most one charged particle at $(\tau, y)$. For any possible choice of $\bb(\tau)$, it is easy to see that $d^\pm(\tau) \geq d^\pm(\tau-)$, and thus Property \ref{t:Pn:LB:d} also holds  at $t = \tau$. Furthermore, $(\bx(\tau), \bb(\tau)) \in \cZ_n$. Then, we can continue the ODE with initial condition $(\bx(\tau), \bb(\tau))$ until the second collision time $\tau_2$. By iteration over the collision times this construction can be continued until $T$. Therefore, it is left to show that $\bx(\tau-)$ exists and that Property \ref{t:Pn:LB:d} holds with $[0,T]$ replaced by $[0,\tau)$. For convenience, we assume that there are no neutral particles prior to $\tau$.
\smallskip

\emph{Property \ref{t:Pn:LB:d} on $[0, \tau)$}. 
For convenience, we focus on $d^+$ and first consider the case $\fext = 0$. 
Let $t \in (0, \tau)$ be a point of differentiability of $d^+$, and let $x_i(t)$ and $x_{i+1}(t)$ be particles for which the minimum in \eqref{dpm} is attained. Then, $(x_{i+1} - x_i)(t) = d^+(t)$, $b_{i+1}(t) = b_i(t) = 1$ and at time $t$, 
\begin{align} \label{pf:zq} 
  \frac d{dt} d^+
  = \frac{ dx_{i+1} }{dt} - \frac{ dx_i}{dt} 
  = \frac2n f(x_{i+1} - x_i) + \frac{1}{n} \sum_{ j \notin \{i, i + 1 \} } b_j \big( f(x_{i+1} - x_j) - f(x_i - x_j) \big).
\end{align} 

To continue the estimate, we apply the argument used in \cite{VanMeursPeletierPozar22} to bound the sum in \eqref{pf:zq} from below. For convenience, we consider the part of the sum where $1 \leq j \leq i-1$. The argument from \cite{VanMeursPeletierPozar22} removes a specific set of indices $\hat J \subset \{1,\ldots,i-1\}$ from the sum such that for the set of remaining indices $J := \{1, \ldots, i-1\} \setminus \hat J$ we have
\begin{equation} \label{pf:vp}
  b_j = 1 \text{ for all } j \in J,
  \qand
  \min_{ k, \ell \in J \cup \{i, i+1\}} |x_k - x_\ell| = d^+.
\end{equation}
For this argument to apply, it is sufficient to show that the terms which are removed yield a nonnegative contribution to the sum in \eqref{pf:zq}. 

In \cite{VanMeursPeletierPozar22} $J$ (and therefore also $\hat J$) is constructed in an iterative manner as follows. It starts with $J = \{1, \ldots, i-1\}$. First, it removes all pairs $(j, j+1)$ from $J$ for which $b_j = 1$ and $b_{j+1} = -1$. Such pairs correspond to the setting in Figure \ref{fig:dip} (right), and thus by Lemma \ref{l:f2} such pairs indeed yield a nonnegative contribution to the sum in \eqref{pf:zq}. Second, it removes all remaining negative particles. From Figure \ref{fig:dip} (left) and the monotonicity of $f$ it is easy to see that such particles also yield a nonnegative contribution to the sum in \eqref{pf:zq}. Hence,
\[
  \sum_{ j =1 }^{i-1} b_j \big( f(x_{i+1} - x_j) - f(x_i - x_j) \big)
  \geq \sum_{ j \in J } \big( f(x_{i+1} - x_j) - f(x_i - x_j) \big).
\]
Noting from \eqref{pf:vp} that $x_i - x_j \geq (i-j) d^+$ and $x_{i+1} - x_j = d^+ + x_i - x_j$, we apply Lemma \ref{l:f2} to further estimate
\begin{align*}
  \sum_{ j \in J } \big( f(x_{i+1} - x_j) - f(x_i - x_j) \big) 
  &\geq \sum_{ k =1 }^{|J|} \big( f((k+1) d^+) - f(k d^+) \big) \\
  &= f((|J|+1) d^+) - f(d^+).
\end{align*}
Using $|J| < n-1$ and applying a similar estimate to the other part of the sum in \eqref{pf:zq} corresponding to $i+2 \leq j \leq n$, we obtain
\begin{align*}
  \frac d{dt} d^+
  \geq \frac2n f(d^+) + \frac{2}{n} \big( f(n d^+) - f(d^+) \big)
  = \frac2n f(n d^+) 
  \geq 0. 
\end{align*}
This proves Property \ref{t:Pn:LB:d} on $[0,\tau)$ with $c = d^+(0)$ for when $\fext = 0$. 

Next we generalize to nonzero $\fext$. From \eqref{Pn:ito:forces} we observe that the contribution of $\fext$ can be treated in the right-hand side of \eqref{pf:zq} independently of the computation above. This yields
\[
  \frac d{dt} d^+
  \geq \frac2n f(n d^+) + \fext(x_i + d^+) - \fext(x_i)
  \geq \frac2n f(n d^+) - \| \fext' \|_\infty d^+.
\]
Since $f(0+) = \infty$, there exists $c > 0$ such that the right-hand side is positive whenever $d^+ \in (0,c)$. Hence, either $d^+(t) \geq c$ for all $t \in [0,\tau)$, or
$
  \frac d{dt} d^+
  \geq 0.
$
This proves Property \ref{t:Pn:LB:d} on $[0, \tau)$.
\smallskip

\emph{Existence of the limit $\bx(\tau -)$}. Let
\[ 
  d_i := x_{i+1} - x_i 
  \quad \text{for } i = 1,\ldots,n-1
\] 
be the distances between neighbors, and take 
\[
  \underline d_i := \liminf_{t \uparrow \tau} d_i(t)
  \quad \text{for } i = 1,\ldots,n-1. 
\]
For technical reasons, we set $\underline d_0 := \underline d_n := 1$. 
If $\underline d_i > 0$, then $x_{i+1}$ and $x_i$ cannot collide with each other at $\tau$. To group together particles which might collide at $\tau$, we split the index set $\{1,\ldots,n\}$ into the tentative collision clusters $I_l$ for $1 \leq l \leq L$: $i$ and $i+1$ belong to the same cluster $I_l$ if and only if $\underline d_i = 0$. Note that a cluster may be a singleton and that $I_1, \ldots, I_L$ are disjoint. Later in the proof it will  turn out that any two particles in the same cluster indeed collide with each other at $\tau$, and that at $\tau$ there are no collisions between any particles from different clusters.  

We prove the existence of $\bx(\tau -)$ by showing that for each tentative collision cluster $I$ the limits $\{ x_i(\tau -) \}_{i \in I}$ exist. With this aim, let $I = \{k, \ldots, \ell\}$ be any tentative collision clusters. Since $\underline d_i = 0$ for all $k \leq i \leq \ell-1$, Property \ref{t:Pn:LB:d} yields $b_i b_{i+1} = -1$ for all $k \leq i \leq \ell-1$. From $\underline d_{k-1}, \underline d_\ell > 0$ and the ordering of the particles it follows that the particles in $I$ remain separated from the particles outside of $I$ on $[0,\tau]$ by some positive distance. Then, since $f(x)$ is bounded away from $x=0$, we obtain for any $i \in I$ that
  \begin{equation} \label{pf:zc}
    \frac d{dt} x_i = \frac1n \sum_{ \substack{ j = k \\ j \neq i } }^{\ell} (-1)^{j-i} f(x_i - x_j) + F_i(\bx) + b_i g(x_i)
  \end{equation}
  for some $t \mapsto F_i(\bx(t))$ which is uniformly bounded on $[0, \tau)$.

Next we change unknowns from $x_k < \ldots < x_{\ell}$ to $d_k, \ldots, d_{\ell-1} > 0$ and the signed first moment
\begin{equation} \label{pf:yb}
  M := \sum_{i=k}^\ell x_i \in \R. 
\end{equation} 
Since this change of unknowns is given by a linear, invertible map, it is sufficient to show that the left limits of $M, d_k, \ldots, d_{\ell-1}$ exist at $\tau$. 

In preparation for this, let
\[
   \underline x := \inf_{[0,\tau)} x_k \leq \sup_{[0,\tau)} x_\ell =: \o x.
\]
We claim that $\underline x > -\infty$ and $\o x < \infty$. Indeed, by the monotonicity of $f$, it follows from \eqref{pf:zc} that
\begin{align*}
  \frac d{dt} x_\ell 
  \leq C + b_\ell g(x_\ell)
  \leq C' + \| g' \|_\infty x_\ell,
\end{align*}
and thus $\o x < \infty$. Similarly, we obtain $\underline x > -\infty$.

Next we prove that the limit $M(\tau-)$ exists. By the oddness of $f$, it follows from \eqref{pf:zc} that
  \begin{align} \label{pf:ya}
    \frac d{dt} M 
    &= \sum_{i=k}^\ell F_i(\bx) + \sum_{i=k}^\ell b_i (g(x_i) - g(\o x)) + g(\o x) \sum_{i=k}^\ell b_i \\\notag
    &\leq C +  \sum_{i=k}^\ell \| g' \|_\infty (\o x - x_i)
    = C' - \| g' \|_\infty M.
  \end{align}
Similarly, we obtain $\frac d{dt} M \geq - C + \| g' \|_\infty M$. Hence, $M$ is Lipschitz continuous on $[0, \tau)$, and thus the limit $M(\tau -)$ exists.

Next we prove the existence of the limits $d_k(\tau-), \ldots, d_{\ell-1}(\tau-)$. If $\ell = k$, then this list is empty and thus the statement is obvious. We let $\ell > k$, take an arbitrary $k \leq i \leq \ell - 1$ and proceed similarly as in \eqref{pf:zq}, but now with $b_j b_{j+1} = -1$ for all $k \leq j \leq \ell - 1$. We obtain from \eqref{pf:zc} that 
  \begin{multline} \label{pf:zd}
    \frac d{dt} d_i
    = \frac{ dx_{i+1} }{dt} - \frac{ dx_i}{dt} 
    = -\frac2n f(d_i) + \frac{1}{n} \sum_{ \substack{ j = k \\ j \neq i, i+1 } }^{\ell} (-1)^{i-j + 1} \big( f(x_{i+1} - x_j) + f(x_i - x_j) \big)  \\
    + F_{i+1} (\bx) - F_i(\bx)
    + b_{i+1} (g(x_{i+1}) + g(x_i)).
  \end{multline}

We claim that the sum in \eqref{pf:zd} is nonnegative. To see this, consider for convenience the part of the sum corresponding to $k \leq j \leq i-1$. Since $f$ is nonincreasing, each pair of indices $\{i-2m, i-2m+1\}$ yields a nonnegative contribution to the sum. This proves the claim for when the number of summands is even. If the number of summands is odd, then the term $j=k$ remains. However, this terms is nonnegative since $f \geq 0$ and $(-1)^{i-k+1} = 1$. This proves the claim.
 
Using this claim, we obtain from \eqref{pf:zd}
  \begin{align} \label{pf:zb}
    \frac d{dt} d_i
    \geq -\frac2n f(d_i) - C - 2 \max_{[\underline x, \o x]} |g|
    \geq -\frac2n f(d_i) - C_0
  \end{align}
  for some constant $C_0 > 0$.

Finally, we use \eqref{pf:zb} to show by contradiction that the limit $d_i(t -)$ exists. Suppose that this limit does not exist. Then 
  \[
    \limsup_{t \uparrow \tau} d_i(t) =: 3 c_0 > 0,
  \]
  and thus there exist $t_0 \in (0, \tau)$ such that
  \[
      \tau - t_0 \leq \delta := \frac{ c_0 }{\frac2n f(c_0) + C_0}
      \qand
      d_i(t_0) \geq 2 c_0.
  \]
  Then, for as long as $d_i \geq c_0$, i.e.\ on $(t_0, t_1)$ for a maximal $t_1 \in (t_0, \tau]$, we have by \eqref{pf:zb} and the monotonicity of $f$ that
  \[
      \frac d{dt} d_i 
      \geq -\frac2n f(c_0) - C_0
      = -\frac{c_0}\delta
      \geq - \frac{c_0}{\tau - t_0}.
  \]
  Thus, for any $t \in (t_0, t_1]$,
  \[
    d_i(t) 
    = d_i(t_0) + \int_{t_0}^t \frac d{ds} d_i(s) \, ds
    \geq 2 c_0 - (t - t_0) \frac{c_0}{\tau - t_0} 
    \geq c_0.
  \]
Hence, $t_1 = \tau$ and thus $d_i \geq c_0$ on $[t_0, \tau]$. This contradicts with $\underline d_i = 0$. This completes the proof for the existence of the limit $\bx(\tau-)$.  
\medskip

Next we prove Properties \ref{t:Pn:LB:pm}-\ref{t:Pn:LB:col}. From the iterative manner in which we have constructed the solution to \eqref{Pn}, it follows that it is sufficient to prove these properties only up to the first collision time $\tau := \tau_1$. In addition, we assume for convenience that there are no neutral particles prior to $\tau$. For Properties \ref{t:Pn:LB:pm} and \ref{t:Pn:C12} this implies that $j = i+1$ and thus $x_j - x_i = d_i$.
\smallskip

\noindent \textbf{Property \ref{t:Pn:LB:pm}.} We may assume that $d_i(t_0) \leq 1$.
Similar to the computation in \eqref{pf:zq},\eqref{pf:zd}, we get
\begin{align*}  
\frac{d}{dt} d_i \geq - \frac2n \bigg( f(d_i) + 2 \sum_{k=1}^{n} f(k \underline d) \bigg) + b_{i+1} \fext(x_{i+1}) - b_i \fext(x_i),
\end{align*}
where $\underline d := \min(d^+, d^-)$. Since $\bx$ is bounded on $[0,\tau]$, we have that $\fext (x_j)$ is bounded on $[0,\tau]$ for all $j$. By Property \ref{t:Pn:LB:d}, $\underline d$ is bounded from below. Then, we obtain from the monotonicity of $f$ and Assumption \ref{a:fg}\ref{a:fg:singUB} that
\begin{align*} 
\frac{d}{dt} d_i 
\geq - \frac2n f(d_i) - C
\geq - \frac{C_0}{d_i^{1+a}} 
\end{align*}
for some constants $C, C_0 > 0$.
By comparison with $\frac{d}{dt} \mathsf d = -C_0 \mathsf d^{-(1+a)}$ with initial datum $\mathsf d(t_0) = d_i(t_0)$, we deduce \eqref{opposite-sign-lower-bound}.  
\medskip

\noindent \textbf{Property \ref{t:Pn:C12}.} It is sufficient to prove that $d_i(t) \leq C (\tau - t)^{1/(2+a)}$ for all $t \in [\tau - \delta, \tau]$ for some $C, \delta > 0$. By the continuity of $d_i$ and $d_i(\tau) = 0$, we may therefore assume that $d_i(t)$ is small enough. Similar as in the proof of the existence of $\bx(\tau-)$ we obtain that $d_i$ satisfies \eqref{pf:zb} for some $C_0 > 0$. By Assumption \ref{a:fg}\ref{a:fg:singUB} this implies for $\delta$ small enough that
\[
  \left\{ \begin{aligned}
     \frac d{dt} d_i
    &\geq -\frac{ 2 C }{ n d_i^{1+a} } - C_0 \geq -\frac{ C_1 }{ d_i^{1+a} }
    &&\text{on } (\tau - \delta, \tau) \\
    d_i(\tau)
    &= 0
    &&
  \end{aligned} \right.
\]
for some constant $C_1 > 0$. By comparison with \comm{p.103}
\[
  \left\{ \begin{aligned}
     \frac d{dt} \mathsf d
    &= -\frac{ C_1 }{ \mathsf d^{1+a} } 
    &&\text{on } (\tau - \delta, \tau) \\
    \mathsf d(\tau)
    &= 0,
    &&
  \end{aligned} \right.
\]
we obtain for all $t \in (\tau - \delta, \tau)$ that
\[
  d_i(t) \leq \mathsf d(t) = C_2 (\tau - t)^{\tfrac1{2+a} }.
\]
for some constant $C_2 > 0$. This proves Property \ref{t:Pn:C12}.
\medskip

\noindent \textbf{Property \ref{t:Pn:LB:col}.} We may assume without loss of generality (recall Corollary \ref{c:Pn}) that:
\begin{itemize}
  \item $y = 0$,
  \item the indices of all particles colliding at $(\tau, 0)$ are given by $I = \{k, k+1, \ldots, \ell\}$ for some $1 \leq k < \ell \leq n$,
  \item $b_i(\tau-) = (-1)^{i-k+1}$ for all $k \leq i \leq \ell$, and
  \item for each $k \leq i \leq \ell$, the ODE for $x_i$ on $(0, \tau)$ is of the form \eqref{pf:zc} for some $t \mapsto F_i(\bx(t))$ which is uniformly bounded on $[0, \tau)$. 
\end{itemize}
We set 
\[
  G_i(\bx) := F_i(\bx) + b_i g(x_i).
\]
Since the precise expressions of $F_i$ and $G_i$ are irrelevant to the following arguments, we may further assume that $k=1$.

We first assume that $\ell$ is odd and treat the case in which $\ell$ is even afterwards. Let
\begin{equation} \label{pf:vm}
  D := x_\ell - x_1. 
\end{equation}
Note that $D$ is continuous on $[0,\tau]$ with $D(\tau) = 0$.
Using \eqref{pf:zc} we write 
\begin{align} \notag 
  \frac d{dt}  D
  &= \frac1n \sum_{ j=1 }^{\ell-1} (-1)^{j-\ell} f(x_\ell - x_j) 
    + \frac1n \sum_{ j=2 }^{\ell} (-1)^{j-1} f(x_j - x_1) 
    + G_\ell(\bx) - G_1(\bx)     \\ \notag
  &= \frac1n \sum_{j = 1}^{(\ell - 1)/2} \big( f(x_{\ell} - x_{\ell - 2j}) - f(x_{\ell} - x_{\ell - 2j + 1}) \big) \\\label{pf:vo}
  &\quad + \frac1n \sum_{j = 1}^{(\ell - 1)/2} \big( f(x_{2j+1} - x_1) - f(x_{2j} - x_1) \big)
  + G_\ell(\bx) - G_1(\bx).
\end{align}
Next we estimate \eqref{pf:vo} from above. Since $G_j$ is uniformly bounded, we apply $G_\ell(\bx) - G_1(\bx) \leq C$. Both summations can be treated similarly; we focus on the latter. By the monotonicity of $f'$ we have for all $x \geq y > z > 0$ that 
\begin{align*}
  f(y) - f(z)
  = \int_z^y f'(\xi) \, d\xi 
  \leq (y-z) f'(x).
\end{align*}
Applying this inequality with $x = D$, $y = x_{2j+1} - x_1$ and $z = x_{2j} - x_1$,
we obtain
\begin{align*}
   \frac1n \sum_{j = 1}^{(\ell - 1)/2} \big( f(x_{2j+1} - x_1) - f(x_{2j} - x_1) \big)
   \leq \frac1n f'(D) \sum_{j = 1}^{(\ell - 1)/2} (x_{2j+1} - x_{2j}).
\end{align*}  
Similarly, we estimate the first sum in the right-hand side of \eqref{pf:vo} as
\[
  \frac1n \sum_{j = 1}^{(\ell - 1)/2} \big( f(x_{\ell} - x_{\ell - 2j}) - f(x_{\ell} - x_{\ell - 2j + 1}) \big)
  \leq \frac1n f'(D) \sum_{j = 1}^{(\ell - 1)/2} (x_{\ell-2j+1} - x_{\ell-2j}).
\]
Substituting both estimates in \eqref{pf:vo} and recalling that $|x f'(x)| \geq c'/x^{1+a'}$ for $x > 0$ small enough, we obtain
\begin{align} \label{pf:uv}
  \frac d{dt} D 
  \leq \frac1n D f'(D) + C
  \leq - \frac{ c_0 }{ D^{1 + a'} } 
\end{align}
on $(\tau - \delta, \tau)$ for some $\delta \in (0, \tau]$ small enough with respect to the continuity of $D$ (recall $D(\tau) = 0$) and some constant $c_0 > 0$. By comparison with
\[
  \left\{ \begin{aligned}
     \frac d{dt} \mathsf D
    &= -\frac{ c_0 }{ \mathsf D^{1+a'} } 
    &&\text{on } (0, \tau) \\
    \mathsf D(\tau)
    &= 0,
    &&
  \end{aligned} \right.
\]
we obtain 
\[
  x_\ell(t) - x_1(t)
  = D(t) \geq \mathsf D(t) = c (\tau - t)^{\tfrac1{2+a'} }
  \qquad \text{for all } t \in (\tau - \delta, \tau)
\]
for some $c > 0$.
Hence, at least one of the following two statements holds:
\begin{itemize}
  \item $x_\ell(t)$ satisfies the upper bound in Property \ref{t:Pn:LB:col},  or
  \item $x_1(t)$ satisfies the lower bound in Property \ref{t:Pn:LB:col}.
\end{itemize}

In fact, both statements hold. To prove this, we show that one statement implies the other. Suppose that 
\begin{equation} \label{pf:zx}
  x_\ell(t) \geq c_0 (\tau - t)^{\tfrac1{2+a'} }
\end{equation}
for all $t \in (\tau - \delta, \tau)$ for some $\delta \in (0,\tau)$. Fix $t \in (\tau - \delta, \tau)$. Recalling the signed first moment $M$ of $x_1,\ldots,x_\ell$ defined in \eqref{pf:yb} with derivative computed in \eqref{pf:ya}, we obtain from $M(\tau) = 0$ that 
\[
  |M(t)| = \bigg| \int_t^\tau \frac{dM}{dt} (s) \, ds \bigg| \leq C_0 (\tau - t)
\]
for some constant $C_0$ independent of $t,\delta$.
Moreover, by the ordering of the particles, 
\[
  M(t) 
  = \sum_{i=1}^\ell x_i(t)
  \geq (\ell-1) x_1(t) + x_\ell(t)
  \geq (\ell-1) x_1(t) + c_0 (\tau - t)^{\tfrac1{2+a'} }.
\]
Rearranging the terms, we obtain
\[
  x_1(t) 
  \leq \frac{M(t)}{\ell-1} - \frac{c_0}{\ell-1} (\tau - t)^{\tfrac1{2+a'} }
  \leq \frac1{\ell-1} \Big( C_0 (\tau - t)^{\tfrac{1+a'}{2+a'} } - c_0 \Big) (\tau - t)^{\tfrac1{2+a'} }.
\]
Using that $\tau - t < \delta$ and taking $\delta$ small enough with respect to the constants $C_0, c_0 > 0$ and $a' \in (-1,a]$, we obtain the desired estimate
\begin{equation} \label{pf:xz}
  x_1(t) 
  \leq -\frac{c}{2(\ell-1)} (\tau - t)^{\tfrac1{2+a'} }
  \leq -\frac c{2n} (\tau - t)^{\tfrac1{2+a'} }.
\end{equation}
A similar argument shows that \eqref{pf:xz} implies \eqref{pf:zx}. This completes the proof of Property \ref{t:Pn:LB:col} for odd $\ell$.

The proof of Property \ref{t:Pn:LB:col} for even $\ell$ is similar; the only difference is that the computation leading to \eqref{pf:uv} is simpler. Indeed, in \eqref{pf:vo} we can still pair up particles which have a negative contribution to $\frac d{dt} D$. Here, we simply estimate these contributions from above by $0$. Since $\ell$ is even, each sum has one particle remaining, and this particle yields a negative contribution to $\frac d{dt} D$. Precisely, \comm{See p.119.bot for $f(x) \geq c/x^{1+a'}$ for $x$ SE}
\begin{align*}
  \frac d{dt} D
  \leq -\frac2n f(x_{\ell} - x_1) + G_\ell(\bx) - G_1(\bx)
  \leq -\frac2n f(D) + C
  \leq - \frac c{D^{1+a'}}.
\end{align*}
\medskip

\noindent \textbf{Property \ref{t:Pn:stab}.} The proof is loosely based on that in \cite{VanMeursPeletierPozar22}. Since the statement of Property \ref{t:Pn:stab} is different and since the setting is more general, we give the proof in full detail. 

The idea of the proof is as follows. Sufficiently before $\tau_1$, i.e.\ on $[0, \tau_1 - \delta]$ for some fixed small $\delta > 0$, the limiting trajectories (i.e.\ $t \mapsto (t, x_i(t))$) remain separated from each other by a positive distance. Then, the singularity of $f$ at $0$ plays no role in the right-hand side of the ODE \eqref{Pn:ito:forces}, and it follows from standard ODE theory that the perturbed trajectories (i.e.\ $t \mapsto (t, x_i^\sigma(t))$) converge to those of $x_i$ as $\sigma \to 0$. In particular, for $\sigma$ small enough, the perturbed trajectories also remain separated, and thus $\bb^\sigma = \bb$ on $[0, \tau_1 - \delta]$. If no problems occur around collision times, then this argument also applies to the solutions on the intervals $[\tau_k + \delta, \tau_{k+1} - \delta]$.

It then remains to show that indeed no problems occur around collision times. More precisely, we show that in a time neighborhood $\cN_k$ of $\tau_k$, the trajectories $\bx^\sigma$ are close to those of $\bx$, and that there is a one-to-one correspondence between the limiting particles which annihilate at $\tau_k$ and the perturbed particles which annihilate during $\cN_k$. This becomes complex when three or more limiting particles collide at a collision point $(\tau_k, y)$. Then, the corresponding perturbed particles collide typically at \textit{several different} collision points $(\tau_\ell^\sigma, y_\ell^\sigma)$ which are all close to $(\tau_k, y)$. We have no control on \textit{which} of the perturbed particles collide; that depends in detail on the choice of $\bx^{\sigma,\circ}$. This explains the need for a $\sigma$-dependent relabeling of the perturbed particles. Yet, we can show that both the limiting and the perturbed trajectories have to remain close to $y$ during $\cN_k$ (in Figure \ref{fig:stab} below this means that they remain inside the trapezoids), and this gives enough control to prove the convergence of the perturbed trajectories even beyond $\cN_k$. 
\smallskip

Next we prove Property \ref{t:Pn:stab}. For this it is sufficient to prove Property \ref{t:Pn:stab} only up to some fixed time in between $\tau_1$ and $\tau_2$, which we take for convenience to be
\[
  T_* := \frac{\tau_1 + \tau_2}2.
\]
Indeed, if Property \ref{t:Pn:stab} holds up to $T_*$, then for some $\sigma$-dependent perturbation $\mu \in S_n$ which encodes the relabeling, we have, setting $\bx_\mu := (x_{\mu(i)})_{i=1}^n \in \R^n$, that  $\bx_\mu^\sigma (T_*) \to \bx(T_*)$ as $\sigma \to 0$ and $\bb_\mu^\sigma(T_*) = \bb(T_*)$ for $\sigma$ small enough. These properties are equivalent to those at initial time, and thus Property \ref{t:Pn:stab} can be established up to the end time $T$ by iterating over the collision times. Therefore, in what follows it is sufficient to consider $\tau := \tau_1$.

To prove Property \ref{t:Pn:stab} up to $T_*$ it is sufficient to show that for all $\delta > 0$ with $\delta < \min \{\tau, T_* - \tau \}$ there exist $\mu = \mu_\sigma \in S_n$ and $\sigma_0 > 0$ such that for all $\sigma \in (0,\sigma_0)$ the following five (in)equalities hold:
\begin{subequations} \label{pf:vw} 
  \begin{align} \label{pf:vwa}
  \| \bx^\sigma - \bx \|_{C([0, \tau])} &\leq \delta, \\\label{pf:vwb}
  \| \bx_\mu^\sigma - \bx \|_{C([\tau, T_*])} &\leq \delta, \\\label{pf:vwc}
  \bb^\sigma(\tau - \delta) &= \bb(\tau-), \\\label{pf:vwd}
  \bb_\mu^\sigma(\tau + \delta) &= \bb(\tau), \\\label{pf:vwe}
  \bb_\mu^\sigma(T_*) &= \bb(\tau).
\end{align}
\end{subequations}
Note that \eqref{pf:vwc}--\eqref{pf:vwe} imply that all collisions between perturbed particles up to time $T_*$ happen on the time interval $(\tau - \delta, \tau + \delta]$.
\smallskip

In the remainder of the proof we prove \eqref{pf:vw}. Let $\delta > 0$ with $\delta < \min \{\tau, T_* - \tau \}$ be given. We split the time interval $[0,T_*]$ at $\tau - r$ and $\tau + h$ for parameters $0 < r \leq h \leq \delta$ which we specify during the proof. To avoid circular dependence on the parameters, we will take $h$ small enough independently of $\sigma_0, r$ and take $r$ small enough independently of $\sigma_0$. 

On the first of the three time intervals, i.e.\ on $[0, \tau - r]$, we show that for all $\delta_1 \in (0, \delta]$ and for all $\sigma$ small enough that \comm{we will apply \eqref{pf:zoa} for both choices $\delta_1 = h, d_r$}
\begin{subequations} 
  \begin{align} \label{pf:zoa}
  \| \bx^\sigma - \bx \|_{C([0, \tau - r])} &\leq \delta_1, \\\label{pf:zob}
  \bb^\sigma(\tau - r) &= \bb(\tau-).
\end{align}
\end{subequations}
On $[0, \tau - r]$ the limiting particles remain separated by the distance
\begin{align*}
  d_r := \min_{0 \leq t \leq \tau - r} \min_{1 \leq i \leq n-1} (x_{i+1} - x_i)(t) > 0,
\end{align*}
i.e.\ $(\bx, \bb)$ remains in a compact subset of $\cZ_n$.  Since the right-hand side of the ODE \eqref{Pn:ito:forces} is Lipschitz continuous on any compact subset of $\cZ_n$, it follows from a standard application of Gronwall's lemma (relying on $\|\fext^\sigma - \fext\|_\infty \to 0$) that \eqref{pf:zoa} holds.  Applying \eqref{pf:zoa} for some $\delta_1 < \frac13 d_r$, we have for $\sigma$ small enough that also the perturbed particles $x_i^\sigma$ remain separated from each other by some distance large than $\frac13 d_r$. Hence, there are no collisions before or at $t = \tau - r$, and thus \eqref{pf:zob} holds. Since $r \leq \delta$, this implies \eqref{pf:vwc}.
\smallskip

Next we turn to the second time interval $[\tau - r, \tau + h]$. We show that for all $i$ both $x_i$ and $x_i^\sigma$ remain inside the trapezoid centred at $x_i(\tau)$; see Figure \ref{fig:stab}. The constants $C_0 > 1$ and $\rho > 0$ in Figure \ref{fig:stab} are independent of $\sigma, r, h, \delta$. In particular, 
\[
  \rho := \min \{ |x_i(\tau) - x_j(\tau)| : x_i(\tau) \neq x_j(\tau) \}
\]
is the minimal distance between the centres of any two trapezoids. We choose $C_0$ later.
We assume that
\begin{equation} \label{pf:vn}
  h \leq \frac{2\rho}{3 C_0}
\end{equation}
such that the trapezoids are further apart than $\frac13 \rho$. Finally, we take $R > 0$ such that all trajectories $x_i$ are contained in $(0,T_*) \times B_R$.

\begin{figure}[ht]
\centering
\includegraphics[height=3.6cm]{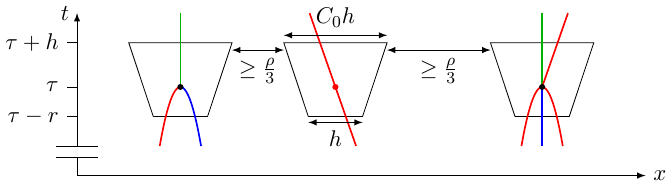}
\caption{Sketch of possible trajectories of the limiting particles $x_i$ around $\tau$. Each particle $x_i$ enters at $t = \tau - r$ the trapezoid around $x_i(\tau)$ and leaves the trapezoid at $t = \tau + h$. Green trajectories correspond to particles which become neutral after collision.}
\label{fig:stab}
\end{figure}

First, we establish an estimate on $\bx(t)$, which in particular implies that $x_i(t)$ remains inside its corresponding trapezoid. This estimate is 
\begin{align} \label{pf:vv}
  \max_{1 \leq i \leq n} |x_i(t) - x_i(\tau)| \leq \left\{ \begin{aligned}
    &\frac h4
    && \text{if } \tau - r \leq t < \tau \\
    &\frac{C_0}2 (t - \tau)
    && \text{if } \tau \leq t \leq \tau + h.
  \end{aligned} \right.
\end{align}
The estimate on $[\tau - r, \tau)$ follows from the continuity of $\bx$ on $[0,T]$, and holds when $r$ is small enough with respect to $h$. The estimate on $[\tau, \tau + h]$ is trivial for particles which are annihilated at $\tau$, because for those particles $x_i(t) = x_i(\tau)$. Since the surviving particles on $[\tau, T_*]$ remain separated by a positive distance $d_1 > 0$, we have that $\bx$ is Lipschitz continuous on $[\tau, T_*]$, and thus \eqref{pf:vv} holds for a sufficiently large $C_0$.

Second, we show that for $\sigma$ small enough also the perturbed particles $x_i^\sigma(t)$ remain inside the trapezoid centred at $x_i(\tau)$, and that for all perturbed particles which enter the same trapezoid, at most one can survive. In preparation for showing this, we divide the limiting particles $x_i$ into clusters $I_1, \ldots, I_L$ (similar to the construction in the proof for the existence of solutions) of particles which enter the same trapezoid, i.e.
\[
  x_i(\tau) = x_j(\tau)
  \iff
  \exists \, 1 \leq l \leq L : i,j \in I_l.
\]
For each $l$ we set
\[
  y_l := x_i(\tau) \quad \text{for some } i \in I_l.
\]
If $|I_l| \geq 2$, then all limiting particles with index in $I_l$ collide at the collision point $(\tau, y_l)$. For each $I_l$ we will prove (sequentially) that for $\sigma$ small enough
\begin{subequations} \label{pf:zp}
  \begin{align} \label{pf:zpa}
  \max_{\tau - r \leq t \leq \tau + h} \max_{i \in I_l} |x_i^\sigma(t) - y_l| 
  &\leq \frac{C_0}2 h, \\\label{pf:zpb}
  \sum_{i \in I_l} | b_i^\sigma(\tau + h) | &\leq 1.
\end{align}
\end{subequations}
The motivation for proving \eqref{pf:zpb} is as follows. Assume that \eqref{pf:zpa} holds. If $|I_l|$ is even, then \eqref{pf:zpb} implies by the local charge conservation in Definition \ref{d:Pn}\ref{d:Pn:ann:rule} that $b_i^\sigma(\tau + h) = 0$ for all $i \in I_l$, and therefore $b_i^\sigma(\tau + h) = b_i(\tau)$. In this case, no relabeling is necessary. If $|I_l|$ is odd, then from a similar reasoning we obtain that there is precisely one index $i \in I_l$ for which $b_i(\tau) \neq 0$ and precisely one index $j \in I_l$ for which $b_j^\sigma(\tau + h) \neq 0$. Since by \eqref{pf:zob} $b_k^\sigma(\tau - r) = b_k(\tau -)$ for all $k \in I_l$, we have that $b_j^\sigma(\tau + h) = b_i(\tau)$. Based on this we take $\mu \in S_n$ such that $\mu(i) = j$ and $\mu(k) \in I_l$ for all $k \in I_l$. By putting this together, we obtain 
\begin{equation} \label{pf:vu}
  \bb_\mu^\sigma(\tau + h) = \bb(\tau).
\end{equation}
We remark that this almost implies \eqref{pf:vwd} and \eqref{pf:vwe}; the only statement left to prove afterwards is that on the time interval $[\tau + h, T_*]$ no collisions happen.

Next we prove \eqref{pf:zp}. Applying \eqref{pf:zoa} with $\delta_1 = \frac14 h$ we have for $\sigma$ small enough that
\begin{subequations} 
  \begin{align} 
   \label{pf:zlb}
   \max_{1 \leq i \leq n} |x_i^\sigma - x_i|(\tau - r) 
   &\leq \frac h4, \\\label{pf:zlc}
   \| \fext^\sigma - \fext \|_\infty 
   &\leq 1.
\end{align}
\end{subequations} 
Together with \eqref{pf:vv} this implies that the perturbed particles enter their corresponding trapezoid at $\tau - r$, i.e.
\begin{equation*} 
  \max_{1 \leq i \leq n} |x_i^\sigma (\tau - r) - x_i(\tau)| 
   \leq \frac h2.
\end{equation*}
Let
\begin{equation} \label{pf:wg}
  T^\sigma 
  := \inf \Big\{ t > \tau - r \, \Big| \, \max_{1 \leq i \leq n} \big| x_i^\sigma(t) - x_i(\tau) \big| > \frac \rho3 \Big\}
\end{equation} 
be the first time at which a perturbed particle is further away from the center of its trapezoid than $\frac13 \rho$. By this definition, during $[\tau - r, T^\sigma]$ perturbed particles from different trapezoids are separated by at least $\frac13 \rho$. Later it will turn out that $T^\sigma \geq \tau + h$, which in particular implies that all perturbed particles leave their trapezoid at $t = \tau + h$. For now, by taking 
$
  h < \frac23  \rho
$
we have at least that $T^\sigma > \tau - r$. 

Next we fix $l$ and set $I_l = \{k, \ldots, \ell\}$. Since the case $k=1$ or $\ell = n$ can be treated with a simplification to the argument that follows, we assume for convenience that $k \geq 2$ and $\ell \leq n-1$. 
We split two cases: $\ell = k$ and $\ell \geq k + 1$. In the first case, the cluster $I_l$ contains only one perturbed particle, namely $x_k^\sigma$. 
Since $x_k^\sigma$ remains separated from the other perturbed particles on $[\tau - r, T^\sigma]$ by a distance of at least $\frac13 \rho$, we obtain from the ODE \eqref{Pn:ito:forces} and the monotonicity of $f$ that (recall \eqref{pf:zlc})
\begin{equation*}
  \Big| \frac d{dt} x_k^\sigma \Big|
  \leq \frac{n-1}n f \Big( \frac \rho3 \Big) + \sup_{B_{R + \rho}} | \fext^\sigma |
  \leq f \Big( \frac \rho3 \Big) + \sup_{B_{R + \rho}} | \fext | + 1 
  =: C_1
  \quad \text{on } (\tau - r, T^\sigma)
\end{equation*}
for $\sigma$ small enough with respect to $\rho$.
Thus, for all $t \in [\tau - r, \tau + h]$ with $t < T^\sigma$ we have (recall $r \leq h$)
\begin{align} \notag
  \big| x_k^\sigma(t) - y_l \big|
  &\leq \big| x_k^\sigma(t) - x_k^\sigma(\tau - r) \big| + \big| x_k^\sigma(\tau - r) - y_l \big| \\\label{pf:zg}
  &\leq \int_{\tau - r}^t \Big| \frac{d x_k^\sigma }{dt}  (s) \Big| \, ds + \frac h2
  \leq (t - (\tau - r)) C_1 + \frac h2
  \leq C_2 h.
\end{align}
Hence, by taking $C_0 \geq 2 C_2$ we have that $x_k^\sigma$ remains inside the trapezoid at $y_l$ at least until $T^\sigma$.

Next we treat the second case $\ell \geq k + 1$ where there are two or more perturbed particles in the cluster $I_l$. Let $S_l^\sigma := \{\tau_1^\sigma, \ldots, \tau_{K^\sigma}^\sigma \}$ be the collision times of the perturbed particles in $I_l$. For $t < \tau_1^\sigma$ we apply a similar computation as for $|I_l| = 1$ to obtain one-sided bounds on the outer two particles $\frac d{dt} x_k^\sigma$ and $\frac d{dt} x_\ell^\sigma$. For $\frac d{dt} x_\ell^\sigma$ we get, using again the monotonicity of $f$, that
\begin{align} \label{pf:ze} 
  \frac{ d x_\ell^\sigma}{dt} 
  &\leq \frac1n \sum_{j = k}^{\ell - 1} (-1)^{\ell - j} f(x_{\ell}^\sigma - x_{j}^\sigma)
  + \frac1n \sum_{j \notin I_l} f \Big( \frac \rho3 \Big) + \sup_{B_{R + \rho}} | \fext^\sigma | \\\label{pf:yz}
  &\leq 0 + f \Big( \frac \rho3 \Big) + \sup_{B_{R + \rho}} | \fext | + 1 
  = C_1
  \qquad \text{on } \big(\tau - r, \min \{T^\sigma, \tau_1^\sigma \} \big).
\end{align} 
 Then, similar to \eqref{pf:zg} we obtain 
\begin{equation} \label{pf:zf}
  x_i^\sigma(t) - y_l
  \leq x_\ell^\sigma(t) - y_l
  \leq C_2 h
  \quad \text{for all } \tau - r \leq t \leq \min \{T^\sigma, \tau_1^\sigma, \tau + h \}
\end{equation}
for all $i \in I_l$. 

The estimate in \eqref{pf:zf} also holds beyond collisions, i.e.\ \eqref{pf:zf} also holds when $\tau_1^\sigma$ is removed from the minimum. To show this, we briefly recall the corresponding argument from the proof of \cite[Theorem 2.4(vi)]{VanMeursPeletierPozar22}. The argument goes by iteration over the collision times $S_l^\sigma$. If at some collision time $\tau_j^\sigma$ particles other than $x_\ell^\sigma$ are annihilated, then the signs of the surviving particles remain alternating, and thus the first sum in \eqref{pf:ze} remains nonpositive. Hence, \eqref{pf:zf} holds beyond such collision. If $x_i^\sigma$ itself gets annihilated at $\tau_j^\sigma$, then the left-hand side in \eqref{pf:zf} is constant for all $t \geq \tau_j^\sigma$, and thus $x_i^\sigma(t) - y_l \leq C_2 h$ remains to hold. Finally, if $x_\ell^\sigma$ gets annihilated at $\tau_j^\sigma$, then we continue from the rightmost charged particle $x_m^\sigma$ in $I_l$, for which the estimates \eqref{pf:ze}, \eqref{pf:yz} and \eqref{pf:zf} hold beyond $\tau_j^\sigma$. 

An analogous treatment for $x_k^\sigma$ yields 
\begin{equation} \label{pf:yy}
  x_i^\sigma(t) - y_l 
  \geq - C_2 h
  \quad \text{for all } \tau - r \leq t \leq \min \{T^\sigma, \tau + h \}
\end{equation}
for all $i \in I_l$. Combining this with \eqref{pf:zf} we observe from \eqref{pf:vn} and \eqref{pf:wg} that the perturbed particles in $I_l$ do not move further away from $y_l$ than $\frac13 \rho$. Since $l$ was arbitrary, we obtain $T^\sigma \geq \tau + h$. Then, \eqref{pf:zf} and \eqref{pf:yy} imply \eqref{pf:zpa}.  
 
Next we prove \eqref{pf:zpb}. Again, we fix $l$ and write $I_l = \{k, \ldots, \ell\}$. If $\ell = k$, then \eqref{pf:zpb} is obvious. For $\ell > k$, we apply a similar computation as for $D$ in the proof of Property \ref{t:Pn:LB:col}. As in that proof, we focus on the difficult case where $\ell - k$ is even (i.e.\ $|I_l|$ is odd).

Similar to \eqref{pf:vm}, but now for the perturbed particles, we take
\[
  D^\sigma(t) := \max \big\{ |x_i^\sigma(t) - x_j^\sigma(t)| \: \big| \: i,j \in I_l \text{ and } b_i^\sigma(t) b_j^\sigma(t) \neq 0 \big\}
  \qquad \text{for } t \in [\tau - r, \tau + h].
\]
This expression is more involved than in \eqref{pf:vm}, in which case $D^\sigma = x_\ell^\sigma - x_k^\sigma$. While $D^\sigma = x_\ell^\sigma - x_k^\sigma$ holds initially, particle $x_\ell^\sigma$ or particle $x_k^\sigma$ may collide during $(\tau - r, \tau + h]$. 

We prove \eqref{pf:zpb} by contradiction. Suppose that the sum in \eqref{pf:zpb} is larger than $1$. Then, $D^\sigma(\tau + h) > 0$. To reach a contradiction, we first consider $D^\sigma$ on $[\tau - r, \tau_1^\sigma]$. Then, we can repeat the same computation as for $D$ in the proof of Property \ref{t:Pn:LB:col} (relying on the fact that perturbed particles from different trapezoids remain separated) to obtain
\begin{equation*} 
  \frac{ d D^\sigma }{dt} \leq \frac{1}{n} D^\sigma f'(D^\sigma) + C_3
  \qquad \text{on } (\tau - r, \tau_1^\sigma)
\end{equation*}
for some constant $C_3 > 0$ which does not depend on $\sigma, \sigma_0, r, h$. Since by \eqref{pf:zlb} we have 
$
  D^\sigma (\tau - r) 
  \leq h,
$
it follows from Assumption \ref{a:fg}\ref{a:fg:singLB} that $\frac d{dt} D^\sigma(t) \leq -1$ on $(\tau - r, \tau_1^\sigma)$ for $h$ small enough with respect to $f$ and $C_3$. In fact, we claim that $\frac d{dt} D^\sigma(t) \leq -1$ on $(\tau_i^\sigma, \tau_{i+1}^\sigma)$ for any $i$ for which $\tau_i^\sigma < \tau + h$ and that $D^\sigma (\tau_i^\sigma +) \leq D^\sigma (\tau_i^\sigma -)$ for any $1 \leq i \leq K^\sigma$. The proof for this claim is similar to the argument used below \eqref{pf:zf}; see \cite{VanMeursPeletierPozar22} for details. From this claim we conclude that $D^\sigma(t) \leq [h - (t - (\tau - r))]_+$. In particular, this implies $D^\sigma(\tau - r + h) = 0$, which contradicts with $D^\sigma(\tau + h) > 0$. Hence, \eqref{pf:zpb} follows.  
\smallskip

Finally, using that both the limiting and the perturbed particles stay inside their related trapezoids (see \eqref{pf:vv} and \eqref{pf:zp}), we prove the desired (in)equalities in \eqref{pf:vw}. We already saw that \eqref{pf:zob} implies \eqref{pf:vwc}. Now, a direct consequence of \eqref{pf:vv} and \eqref{pf:zp} is 
\begin{equation} \label{pf:vt}
  \| \bx^\sigma - \bx \|_{C([\tau - r, \tau + h])} \leq C_0 h
  \qand
  \| \bx_\mu^\sigma - \bx \|_{C([\tau - r, \tau + h])} \leq C_0 h,
\end{equation}
where we recall that the permutation $\mu$ only permutes the indices of particles which belong to the same trapezoid. Together with \eqref{pf:zoa} with $\delta_1 = h$, this implies \eqref{pf:vwa}  when taking $h < \delta / C_0$. Moreover, \eqref{pf:vwb} holds up to time $\tau + h$. 

To conclude the remaining part of \eqref{pf:vw}, we need to show that the perturbed particles remain close enough to the limiting particles on $[\tau + h, T_*]$. Since the annihilated particles remain stationary at their collision point inside their trapezoid, we have for all $i$ with $b_i(\tau) = 0$ that $|x_{\mu(i)}^\sigma (t) - x_i(t)| \leq C_0 h$ for all $t \geq \tau + h$. Hence, we may neglect the annihilated particles and focus only on the surviving ones. Doing so, the limiting particles remain separated by a positive distance $d_1 > 0$ (independent of $\sigma, r, h, \delta$) on $[\tau, T_*]$. Hence, if the perturbed particles remain close enough to the limiting particles on  $[\tau, T_*]$, then the remaining (in)equalities in \eqref{pf:vw} follow from \eqref{pf:vu}. To prove that the perturbed particles remain arbitrarily close to the limiting ones, consider \eqref{pf:vt} at $t = \tau + h$, i.e.
\begin{equation} \label{pf:vs}
  |\bx_\mu^\sigma - \bx|(\tau + h) \leq C_0 h.
\end{equation}
Then, applying Gronwall's lemma to the ODEs for $\bx$ and $\bx_\mu^\sigma$ starting at time $\tau + h$, we obtain
\[
  \| \bx_\mu^\sigma - \bx \|_{C([\tau + h, T_*])} \leq \delta_h + \delta_\sigma
\]
for some constants $\delta_h > 0$ (coming from the contribution of \eqref{pf:vs}) and $\delta_\sigma > 0$ (coming from the contribution of \eqref{pf:zlc}) which vanish as $h \to 0$ and $\sigma \to 0$ respectively. Hence, by choosing first $h$ small enough such that $\delta_h \leq \frac12 \min \{ \delta, d_1 \}$ and then $\sigma$ small enough such that $\delta_\sigma \leq \frac12 \min \{ \delta, d_1 \}$, we obtain that both \eqref{pf:vwb} holds and that the perturbed particles do not collide before or at $T_*$. From the latter and \eqref{pf:vu} we obtain \eqref{pf:vwd} and \eqref{pf:vwe}. This completes the proof of Property \ref{t:Pn:stab}.
\end{proof} 

\newpage

\section{Definitions of \eqref{Pn} and ($P^m$) in terms of viscosity solutions}
\label{s:HJ}
 
In this section we give a proper meaning to ($P^m$) for $m=1,2,3$, which we have formally stated in Section \ref{s:intro:aim}. We state this proper meaning in terms of viscosity solutions to integrated versions of these equations. These integrated versions resemble Hamilton-Jacobi equations; see \eqref{HJk} below. As preparation for proving the convergence of \eqref{Pn}, we also introduce viscosity solutions to an integrated version \eqref{HJe} of \eqref{Pn} with $\e = \frac1n$. We do not adopt the usual definition of viscosity solutions for \eqref{HJk} and \eqref{HJe}, because it can be too strong for obtaining existence of solutions.
\comm{HJ$_n$ and \eqref{HJ1} can be called HJ eqns (in a broad sense; classical HJ eqns need the right-hand side to depend only on $u_x(t,x)$), but \eqref{HJk} cannot because of $u_{xx}$ } 

The structure of this section is as follows. In Section \ref{s:HJ:not} we introduce the notation. In Section \ref{s:HJ:ass} we list the assumptions on $V$ and their consequences. In Section \ref{s:HJ:int} we formally state the integrated equations \eqref{HJk} and \eqref{HJe}. In Sections \ref{s:HJ:HJe},  \ref{s:HJ:HJ1} and \ref{s:HJ:HJ23} we define viscosity solution concepts with a comparison principle to these integrated equations, and consider the spatial derivative of these solutions as the solutions to \eqref{P1}, \eqref{P2} and \eqref{P3}. In particular, the comparison principles imply uniqueness of viscosity solutions. Existence of solutions will be a direct corollary of Theorem \ref{t}.

\subsection{Notation}
\label{s:HJ:not}

In addition to the notation introduced at the start of Section \ref{s:ODE}, we introduce the following:
\begin{itemize}
  \item For any $x \in \R$ and any $\rho > 0$, we set $B_\rho(x) := (x - \rho, x + \rho) \subset \R$ as the one-dimensional ball. We further set $B_\rho := B_{\rho}(0)$;
  \item For any $T > 0$, we set $Q_T := (0,T) \times \R$. Moreover, we set $Q := (0,\infty) \times \R$;
  \item We set $C_b(\o Q)$ as the space of continuous and bounded functions on $\o Q$, and
  $BUC(\o Q) \subset C_b(\o Q)$ as the space of bounded and uniformly continuous functions on $\o Q$;
  \item For a function of several variables such as $u = u(t,x)$ we write its partial derivatives as $u_t$ and $u_x$;
  \item $u_*$ and $u^*$ are respectively the lower and upper semi-continuous envelope with respect to all variables of the function $u$;
  \item For a sequence of upper semi-continuous functions $u_\e : Q \to \R$ we set 
  $$ \text{lim\,sup}^* \, u_\e (t,x) := \limsup_{(s,y,\e) \to (t,x,0)} u_\e(s,y)$$
  for each $(t,x) \in \o Q$. Similarly, we define $\liminf_* u_\e$ for lower semi-continuous functions.
\end{itemize} 

\subsection{Asumptions on $U$ and $V$}
\label{s:HJ:ass}

Throughout Sections \ref{s:HJ} and \ref{s:conv} we put the sufficient Assumption \ref{a:UV} on $U$ and $V$. It is chosen such that all results in Sections \ref{s:HJ} and \ref{s:conv} hold. See Section \ref{s:disc:weak:V} for weakened assumptions.

To state Assumption \ref{a:UV} we introduce
\begin{equation*} 
  \V_k(x) := \sup_{y \in (x,\infty)} \big| V^{(k)}(y) \big| \qquad \text{for any } k \in \N.
\end{equation*}

\begin{ass}[$U$ and $V$] \label{a:UV}
$U \in C^1(\R)$ is such that $U'$ is Lipschitz continuous. $V$ satisfies:
\begin{enumerate}[label=(\roman*)]
  \item $V : \R \setminus \{0\} \to \R$ is even;
  \item $V \in C^5((0,\infty))$;
  \item \label{a:UV:cv} $V'' \geq 0$ on $(0,\infty)$;
  \item \label{a:UV:to0} $V^{(k)}(x) \to 0$ as $x \to \infty$ for $k=0,\ldots,4$;
  \item \label{a:UV:int} $x \mapsto x^5 \V_5(x)$ is in $L^1 (0,\infty)$.
\end{enumerate}
\end{ass}

The main assumptions are the convexity in \ref{a:UV:cv} and the integrability in \ref{a:UV:int}. We consider \ref{a:UV:to0} as a supplement to \ref{a:UV:int} to determine the integration constants.
Note that \ref{a:UV:int} hardly limits the strength of the singularity and the decay of the tails of $V$ beyond integrability. Instead, it limits oscillations in the derivatives up to fifth order.

The following lemma lists several consequences of Assumption \ref{a:UV}.

\begin{lem} \label{l:UV}
If $V$ satisfies Assumption \ref{a:UV}, then
\begin{enumerate}[label=(\roman*)]
  \item \label{l:UV:mon} $(-1)^k V^{(k)} \geq 0$ on $(0,\infty)$ for $k = 0,1,2$;
  \item \label{l:UV:L1} $x \mapsto x^k \V_k(x)$ is in $L^1 (0,\infty)$ for $k = 0,\ldots,5$; 
  \item \label{l:UV:to0} $x^{k+1} \V_k (x) \to 0$ both as $x \to 0$ and as $x \to \infty$ for $k = 0,\ldots,4$;
  \item \label{l:UV:intx} $\displaystyle x \int_x^\infty y^{k-1} \V_k(y) \, dy \to 0$ as $x \to 0$ for $k = 1,\ldots,5$.
\end{enumerate} 
\end{lem}

\begin{proof} Lemma \ref{l:UV} can be proven by backward iteration over $k$. We demonstrate this by showing the first iteration step. We start with Property \ref{l:UV:mon}. The case $k=2$ is given. We obtain the case $k=1$ from Assumption \ref{a:UV}\ref{a:UV:cv},\ref{a:UV:to0} by
\[
  -V'(x) = \int_x^\infty V''(y) \, dy \geq 0.
\]

Next we prove Property\ref{l:UV:intx} for $k=5$. We write
\[
  x \int_x^\infty y^4 \V_5(y) \, dy
  = \int_0^\infty \psi_x(y) y^5 \V_5(y) \, dy,
\]
where $\psi_x : (0,\infty) \to \R$ is defined by $\psi_x(y) = 0$ for $y < x$ and $\psi_x(y) = \frac xy$ for $y > x$. Note that $\| \psi_x \|_\infty \leq 1$ and that $\psi_x(y) \to 0$ as $x \to 0$ pointwise in $y$. Then, by Assumption \ref{a:UV}\ref{a:UV:int} and the Dominated Convergence Theorem, Property \ref{l:UV:intx} for $k=5$ follows.

To prove Properties \ref{l:UV:L1} and \ref{l:UV:to0} for $k=4$, let $W(x) := \int_x^\infty \V_5(y) \, dy$ for $x > 0$. Clearly, $W \geq 0$ is non-increasing and $W(x) \to 0$ as $x \to \infty$. Using from Assumption \ref{a:UV}\ref{a:UV:to0} that $V^{(4)}(x) \to 0$ as $x \to \infty$, we obtain the lower bound 
\[
  W(x) 
  = \int_x^\infty \V_5(y) \, dy
  \geq \int_x^\infty |V^{(5)}(y)| \, dy
  \geq \bigg| \int_x^\infty V^{(5)}(y) \, dy \bigg|
  = |V^{(4)}(x)|
\]
for all $x > 0$. Then,
\begin{equation} \label{pf:us}
  \V_4(x)
  = \sup_{(x,\infty)} |V^{(4)}|
  \leq \sup_{(x,\infty)} W
  = W(x)
\end{equation}
for all $x > 0$.

Using \eqref{pf:us}, Property \ref{l:UV:L1} for $k=4$ follows from
\begin{align*}
  \int_0^\infty x^4 \V_4(x) \, dx
  \leq \int_0^\infty x^4 W(x) \, dx
  = \int_0^\infty x^4 \int_x^\infty \V_5(y) \, dy dx
  = \frac15 \int_0^\infty y^5 \V_5(y) \, dy
  < \infty.
\end{align*}
Also, Property \ref{l:UV:to0} for $k=4$ and $x \to \infty$ follows from
\begin{align*}
  x^5 \V_4(x)
  \leq x^5 W(x)
  \leq \int_x^\infty y^5 \V_5(y) \, dy
  \xto{x \to \infty} 0.
\end{align*}
Finally, take any $0 < x < \delta$ and split
 \begin{align*}
  x^5 \V_4(x)
  \leq x^5 \int_x^\infty \V_5(y) \, dy
  &= \int_x^\delta x^5 \V_5(y) \, dy + x^5 \int_\delta^\infty \V_5(y) \, dy \\
  &\leq \int_0^\delta y^5 \V_5(y) \, dy + \frac{x^5}{\delta^5} \int_0^\infty y^5 \V_5(y) \, dy.
\end{align*}
Then, property \ref{l:UV:to0} for $k=4$ and $x \to 0$ follows by first taking $x \to 0$ and then $\delta \to 0$.
\end{proof}

In preparation for what follows, we prove in Lemma \ref{l:f3} below that the infinite sum in \eqref{Psi} is well-defined, i.e.\ that 
\begin{equation} \label{f3}
  f_3(y) = \left\{ \begin{aligned}
    &\sum_{k=1}^\infty \beta^3 \frac{k^2}{y^2} V'' \Big( \beta \frac ky \Big)
    &&\text{if } y \neq 0  \\
    &0
    &&\text{if } y = 0.
  \end{aligned} \right.  
\end{equation}
is well-defined for each $y \in \R$.
We also establish several properties of $f_3$. 

\begin{lem}[Properties of $f_3$] \label{l:f3}
Let $\beta > 0$. The series in \eqref{f3} converges uniformly on any compact subset of $\R \setminus \{0\}$. Moreover, $f_3$ is locally Lipschitz continuous on $\R$, and $f_3 \geq 0$.
\end{lem}

\begin{proof}
Without loss of generality we assume $\beta = 1$. First, we prove that the series in \eqref{f3} converges uniformly on $[\frac1R, R]$ for any $R > 1$. For any integers $1 \leq \ell < n$ and any $x \in [\frac1R, R]$, we estimate
\begin{multline} \label{pf:wj}
  0
  \leq \sum_{k=\ell+1}^n k^2 V'' (kx) 
  \leq  \sum_{k=\ell+1}^n \frac1x \int_{(k-1)x}^{kx} (z+x)^2 \V_2(z) \, dz \\
  = \frac1x \int_{\ell x}^{nx} (z+x)^2 \V_2(z) \, dz
  \leq R \int_{\ell/R}^{Rn} ( z + R )^2 \V_2(z) \, dz,
\end{multline}
which vanishes as $\ell,n \to \infty$ uniformly in $x$ since $y \mapsto y^2 \V_2(y)$ is in $L^1(1,\infty)$. Since $V$ is even and $V'' \geq 0$, it follows that the series in \eqref{f3} converges uniformly on any compact subset of $\R \setminus \{0\}$. Then, it follows from $V \in C^2((0,\infty))$ and $V'' \geq 0$ that $f_3 \in C(\R \setminus \{0\})$ and that $f_3 \geq 0$. The continuity at $y=0$ follows from a similar computation as in \eqref{pf:wj}: \comm{Kosmala T843: $f_3 \in C$ if uf conv + summands $\in C$}
\begin{align} \notag
  f_3(y) 
  &\leq \frac1{y^2} V'' \Big( \frac 1y \Big) + y \int_{1/y}^\infty \Big( z + \frac1y \Big)^2 \V_2(z) \, dz \\ \label{pf:vk}
  &\leq y \frac1{y^3} \V_2 \Big( \frac 1y \Big) + 4 y \int_0^\infty z^2 \V_2(z) \, dz
  \leq Cy
  \qquad \text{for all } y \in (0,1),
\end{align}
where we have used that $x^3 \V_2(x) \to 0$ as $x \to \infty$.

Next we prove that $f_3$ is Lipschitz continuous on $[\frac1R, R]$ for any $R > 1$. First, we show that
\begin{equation} \label{pf:wi}
  f_3'(y) = -\sum_{k=1}^\infty \bigg( 2 \frac{k^2}{y^3} V'' \Big( \frac k y \Big) + \frac{k^3}{y^4} V''' \Big( \frac k y \Big) \bigg)
  \qquad \text{for all } y \neq 0,
\end{equation}
where the summand is the derivative of the summand in \eqref{f3}.
Since the series in \eqref{f3} converges, it is sufficient to show that the series in \eqref{pf:wi} converges in $C([\frac1R, R])$. We show this separately for both terms in \eqref{pf:wi}. The series corresponding to the first term is essentially the same as the series in \eqref{f3}, for which we have already established the uniform convergence. For the second term, we use that $V \in C^4 ((0,\infty))$ and that $x^3 \V_3(x)$ is in $L^1(1,\infty)$ to estimate for any $1 < \ell < n$ and any $y \in [\frac1R, R]$ \comm{Kosmala T8.4.17; $V \in C^3$ needed st each finite series is cts}
\[ 
  \sum_{k=\ell+1}^n k^3 \Big| V''' \Big( \frac ky \Big) \Big|
  \leq \int_{\ell/y}^{n/y} \Big( z + \frac1y \Big)^3 \V_3(z) \, dz
  \leq \int_{\ell/R}^\infty ( z + R )^3 \V_3(z) \, dz,
\]
which vanishes as $\ell,n \to \infty$ uniformly in $y$. 

Finally, we show that $f_3'(y)$ is bounded around $y=0$. For the first term in \eqref{pf:wi}, this follows from a similar estimate as in \eqref{pf:vk} (note the additional factor $\frac1y$). The second term can be estimated by a similar estimate too, this time by relying on $x^3 \V_3(x)$ being in $L^1(1,\infty)$.
\end{proof}

\subsection{The integrated PDEs formally}
\label{s:HJ:int}

Formally integrating the PDEs ($P^m$) in space and substituting $u_x = \kappa$ yields the PDEs
\begin{subequations} 
\begin{align}  \label{HJ1} \tag{$HJ^1$}
  u_t &= ( \cM[u] + U' ) |u_x| && \text{if } m = 1, \\ \label{HJk} \tag{$HJ^m$} 
  u_t &= f_m(u_x) u_{xx} + U' |u_x| && \text{if } m =  2,3,
\end{align}  
\end{subequations}
where $\cM[u] = V_\alpha' * u_x$, $f_2(\kappa) = \|V\|_{L^1(\R)} |\kappa|$ and $f_3$ is given by \eqref{f3}.
We will use the more convenient form of the operator $\cM$ given formally by
\begin{align*} 
  \cM[v](x)
  := \pv \int_\R [v(x+z) - v(x)] V_\alpha''(z) dz
\end{align*}
for any $x \in \R$ and any $v : \R \to \R$, which can formally be thought of as $V_\alpha'' * v$. We note that, even for $v \in C_c^\infty(\R)$, $\cM [v](x)$ may not be defined at each $x$ when the singularity of $V$ is strong enough (consider e.g.\ the case where $v$ has a local minimum at $x$).  

To `integrate' \eqref{Pn}, we require some preparation. Given $(\bx, \bb) \in \cZ_n$ (recall \eqref{Zn}), we set
\begin{equation} \label{un}
  u_n : \R \to \R, \qquad 
  u_n(x) := \frac1n \sum_{i=1}^n b_i H(x - x_i)
\end{equation}
as a piecewise constant function, where $H$ is the Heaviside function. Recall that $u_n' = \kappa_n$ is the signed empirical measure related to $(\bx, \bb)$.
Let $v \in C^1(\R)$ be slightly above $u_n$ such that
\[
  u_n^* \leq v \leq u_{n,*} + \frac1n,
\]
where we recall that an asterisk denotes the upper or lower semi-continuous envelope.

Now, given a solution $(\bx, \bb)(t)$ to \eqref{Pn}, let $u_n(t,x)$ and $v \in C^1(Q)$ be as above.
Writing out $\frac d{dt} v(t, x_i(t)) = 0$ (see \cite[Lemma 4.2]{VanMeursPeletierPozar22} for a detailed computation) yields
\begin{equation} \label{HJe} \tag{$HJ_\e$} 
  v_t = ( \cM_\e[v] + U' ) |v_x|, \qquad \e := \frac1n,
\end{equation}
where $\cM_\e$ is given formally for any $u : \R \to \R$ by
\begin{equation} \label{cMe}
  \cM_\e[u](x)
  := \pv \int_\R E_\e [u(x+z) - u(x)] V_{\alpha_\e}''(z) dz,
\end{equation}
$\alpha_\e = \alpha_n$ (see \eqref{alphaem}) and
\[
  E_\e : \R \to \R, \qquad E_\e [\gamma] := \e \Big( \Big\lfloor \frac\gamma\e \Big\rfloor + \frac12 \Big)
\]
is illustrated in Figure \ref{fig:Ee}. The expression of $\cM_\e$ has even more issues than that of $\cM$, and this requires care when defining viscosity solutions.

\begin{figure}[ht]
\centering
\begin{tikzpicture}[scale=1]
    \draw[->] (-3,0) -- (3,0) node[right]{$\gamma$};
    \draw[->] (0, -3) -- (0,3);
    
    \draw[dotted] (1,1.5) --++ (0,-1.5) node[below]{$\e$}; 
    \draw[dotted] (2,2.5) --++ (0,-2.5) node[below]{$2\e$};
    \draw[dotted] (1,1.5) --++ (-1,0) node[left]{$\frac32 \e$}; 
    \draw[dotted] (2,2.5) --++ (-2,0) node[left]{$\frac52 \e$};
    \draw (0,.5) node[left]{$\frac12 \e$};  
    
    \draw[dashed] (-3, -3) --++ (6,6);
    \foreach \k in {-3,-2,-1,0,1,2}{    
      \draw[very thick] (\k, \k +0.5) --++ (1,0);
      \draw[fill=black] (\k, \k + 0.5) circle[radius=2pt];
      \draw[fill=white] (\k + 1, \k + 0.5) circle[radius=2pt];
    }
    \draw (2,1.5) node[right]{$E_\e$};
\end{tikzpicture} \\
\caption{Plot of the staircase approximation $E_\e$ of the identity.}
\label{fig:Ee}
\end{figure}
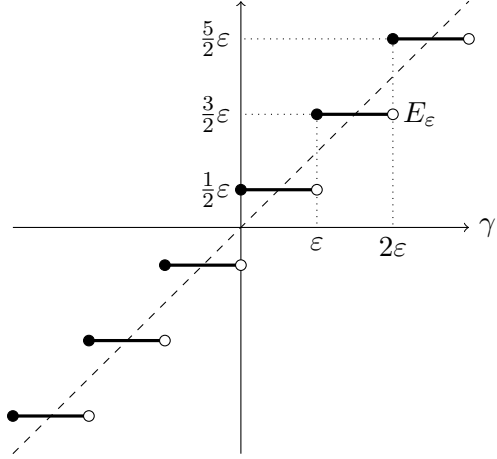 

When working with \eqref{HJe} it is more natural to switch from $n$ to the parameter $\e$, and consider any $\e > 0$. As a result, we replace $\alpha_n$ by 
\begin{align} \label{alphaem}
  \begin{cases}
    \alpha_\e \to \alpha
    &\text{if } m=1 \\
    1 \ll \alpha_\e \ll \e^{-1}
    &\text{if } m=2 \\
    \e \alpha_\e \to \beta
    &\text{if } m=3
  \end{cases} 
\end{align}  
as $\e \to 0$, where $\alpha, \beta > 0$ are the same as for $\alpha_n$ in \eqref{alphanm}.
 
We consider \eqref{HJe} as a more general equation than \eqref{Pn}. Indeed, for each solution $(\bx, \bb)(t)$ to \eqref{Pn} we show in Lemma \ref{l:Pn:to:HJe} that the corresponding function $u_n$ is a viscosity solution to \eqref{HJe}. Proposition \ref{p:Pn:to:HJe:un} provides the sense in which this viscosity solution is unique. In Section \ref{s:disc:HJe} we further discuss the connection between \eqref{HJe} and \eqref{Pn}.

We do not rigorously address the question to which extent viscosity solution of \eqref{HJe} correspond to solutions of \eqref{Pn}. Formally, the reconstruction is as follow. Given $v(t,x)$, consider the union of the level sets of $v$ at $\e \Z$, i.e.
\[
  \Big\{ (t,x) \in Q : \frac1\e v(t,x) \in \Z \Big\}.
\]
Assume that this set can be written as the union of graphs $t \mapsto (t,x_i(t))$ over $i$ (consider e.g.\ Figure \ref{fig:trajs} for an example). The signs $b_i$ are formally defined as $b_i = \sign( v_x(t,x) )$ at a point $(t,x)$ on a trajectory. Then, for any $t$ and any $i$ at which $x_i(t)$ is differentiable, a computation similar to the derivation above of \eqref{HJe} from \eqref{Pn} shows that $x_i$ at $t$ satisfies \eqref{Pn}. 

Next we compare \eqref{HJk} and \eqref{HJe} with those in \cite{VanMeursPeletierPozar22} where $\alpha_\e = \alpha = 1$, $U=0$ and $V(x) = -\log |x|$. In \cite{VanMeursPeletierPozar22}, ($HJ^2$) and ($HJ^3$) do not appear, and the expressions of \eqref{HJ1} and \eqref{HJe} are obtained by replacing $V_\alpha''(z)$ and $V_{\alpha_\e}''(z)$ by $1/z^2$. Hence, with respect to \cite{VanMeursPeletierPozar22}, the treatments of ($HJ^2$) and ($HJ^3$) are new, and for \eqref{HJ1} and \eqref{HJe} more care is needed regarding the singularity and tail of $V$. The appearance of $U$ does not cause significant difficulties.

\subsection{Viscosity solutions of \eqref{HJe} with a comparison principle} 
\label{s:HJ:HJe}

First we give a proper meaning to the right-hand side in \eqref{HJe}, which we call the Hamiltonian. Since we follow a similar construction as in \cite{VanMeursPeletierPozar22} (which is in turn based on that in \cite{ImbertMonneauRouy08} and \cite{ForcadelImbertMonneau09}), we will be brief in the motivation. 

In this section we keep $\e > 0$ fixed. Note that $V_{\alpha_\e}$ satisfies Assumption \ref{a:UV}, and thus the scaling by $\alpha_\e$ is irrelevant for the definition of viscosity solutions. Consider the Hamiltonian in \eqref{HJe}, suppose that $v \in C^1(Q)$ is bounded and take $(t,x) \in Q$. If $v_x(t,x) = 0$, then we define the Hamiltonian as $0$.
If $v_x(t,x) \neq 0$, then $\cM_\e[v](t,x)$ is properly defined. Indeed, the piecewise constant function $z \mapsto E_\e [ v(t,x+z) - v(t,x)]$ in the integrand in \eqref{cMe} jumps at $z = 0$ from $\pm \frac\e2$ to $\mp \frac\e2$. Since $V_{\alpha_\e}''$ is even, the integrand is odd locally around $0$, and thus \eqref{cMe} is well-defined. This motivates the following definition for the Hamiltonian for functions $v$ which need not be differentiable:

\begin{defn}[Hamiltonians at $\e>0$] 
\label{d:Hpe}
Fix $\rho>0$, $u \in L^\infty(Q)$, $t > 0$, $x\in \R$ and $\phi\in C^2(Q)$. If $\phi_x(t,x) \neq 0$, then we 
define 
\begin{align*}
\o H_{\rho,\e} [\phi,u](t, x) 
&:= \big( \o M_{\rho,\e}[\phi(t,\cdot),u(t, \cdot)](x) + U'(x) \big) \,|\phi_x(t, x)|,
 \\
\underline H_{\rho,\e} [\phi,u](t,x) 
&:= \big( \underline M_{\rho,\e}[\phi(t,\cdot),u(t, \cdot)](x) + U'(x) \big) \,|\phi_x(t, x)|,
\end{align*}
where 
\begin{align} \notag 
\o M_{\rho,\e}[\psi,v](x) 
&:= \pv \int_{B_\rho} E_\e^*\big[ \psi(x+z)-\psi(x)\big] V_{\alpha_\e}''(z) \, dz \\\notag 
&\qquad + \int_{B_\rho^c} E_\e^*\big[ v(x+z)-v(x)\big] V_{\alpha_\e}''(z) \, dz, \\\notag
\underline M_{\rho,\e}[\psi,v](x) 
&:= \pv \int_{B_\rho} E_{\e,*}\big[ \psi(x+z)-\psi(x)\big] V_{\alpha_\e}''(z) \, dz \\\label{uMrhoe}
&\qquad + \int_{B_\rho^c} E_{\e,*}\big[ v(x+z)-v(x)\big] V_{\alpha_\e}''(z) \, dz
\end{align}
for any $v \in L^\infty(\R)$ and any $\psi \in C^2(\R)$ with $\psi'(x) \neq 0$.
\end{defn}

The roles of $\rho$ and the test function $\phi$ are that $\phi$ acts as a regular substitute for the possibly discontinuous $u$ in the part of the integral in \eqref{cMe} over $B_\rho$. The two Hamiltonians $\o H_{\rho,\e}$ and $\underline H_{\rho,\e}$ only differ in taking either the upper or lower semi-continuous envelope of $E_\e$. This difference is introduced for technical reasons. We further remark that, while $\phi$ is defined on $Q$, the Hamiltonians only depend on $\phi |_{\{t\} \times B_\rho(x)}$. Hence, the time dependence is of no importance in Definition \ref{d:Hpe}, and is added for convenience later on in Definition \ref{d:HJe:VS} on viscosity solutions. 

Next we check that the expressions $\o M_{\rho,\e}$ and $\underline M_{\rho,\e}$ in Definition \ref{d:Hpe} are well-defined: 

\begin{lem} \label{l:Hpe}
Let $\e, \alpha_\e, \rho > 0$, $x \in \R$, $\phi, \psi \in C^2(\R)$ and $v \in L^\infty(\R)$. Then,
\begin{enumerate}[label=(\roman*)]
  \item \label{l:Hpe:Bp} if $\phi'(x) \neq 0$, then there exists $\rho_0 > 0$ such that for all $\tilde \rho \in (0, \rho_0]$ 
  \[
    \pv \int_{B_{\tilde \rho}} E_\e^*\big[ \phi(x+z)-\phi(x)\big] V_{\alpha_\e}''(z) \, dz = 0;
  \]
  \item \label{l:Hpe:Bpc} $\displaystyle \bigg| \int_{B_\rho^c} E_\e^*\big[ v(x+z)-v(x)\big] V_{\alpha_\e}''(z) \, dz \bigg| 
  \leq ( 4 \|v\|_\infty + \e ) |V_{\alpha_\e}'(\rho)|$;
  \item \label{l:Hpe:order} If $\phi \geq \psi$, $\phi(x) = \psi(x)$ and $0 \neq \phi'(x) \, (= \psi'(x))$. Then,
\begin{align*}
\pv \int_{B_1} E_\e^* [ \phi(x+z)-\phi(x)] V_{\alpha_\e}''(z) \, dz
\geq \pv \int_{B_1} E_\e^* [ \psi(x+z)-\psi(x)] V_{\alpha_\e}''(z) \, dz.
\end{align*}
\end{enumerate}
The above two properties also hold when $E_\e^*$ is replaced with $E_{\e,*}$. 
\end{lem}

\begin{proof}
The proof of \cite[Lemma 3.3]{VanMeursPeletierPozar22} applies to Statements \ref{l:Hpe:Bp} and \ref{l:Hpe:Bpc}.
Statement \ref{l:Hpe:order} follows from $V_{\alpha_\e}'' \geq 0$ and $E_\e$ being nondecreasing.
\end{proof}

The usual approach for defining viscosity solutions is to take $\o H_{\rho,\e}[\phi,u]$ or $\underline H_{\rho,\e}[\phi,u]$ as the Hamiltonian when $\phi_x(t,x) \neq 0$ and to take $0$ as the Hamiltonian when $\phi_x(t,x) = 0$. However, this solutions concept turns out to be too strong (see Section \ref{s:disc:tFct} for a further discussion). Indeed, collision events correspond to $\phi_x(t,x) = 0$, and thus proving that the Hamiltonian is $0$ for all test functions $\phi$ with $\phi_x(t,x) = 0$ requires a detailed description of all possible collision events. Fortunately, this can be largely avoided; the choice in the definition of viscosity solutions is lenient enough to restrict the class of test functions $\phi$ when $\phi_x(t,x) = 0$. However, the smaller we take the class of test functions, the larger the class of possible viscosity solutions, and thus the harder it becomes to establish a comparison principle (as that implies uniqueness of solutions). Therefore, we choose the restricted class of test functions based on our ability to prove both the comparison principles (see Theorems \ref{t:CPe}, \ref{t:CP:HJ1} and \ref{t:CP:HJk}) and the convergence (see Theorem \ref{t:conv}). The restricted class of test functions that we use when $\phi_x(t,x) = 0$ are the functions of the form $\phi(t,x) = \theta |x - \o x|^4 + g(t)$ for some constant $\theta \in \R$ and some function $g \in C^2((0,\infty))$. This class is the same as that used in \cite{VanMeursPeletierPozar22}. In Section \ref{s:disc:tFct} we motivate the choice of the power $4$ (rather than $2$, $6$, $8$ etc.). 

\begin{defn}[$\rho$-sub- and $\rho$-supersolutions for $\e>0$] \label{d:HJe:VS}
Let $\rho, \e > 0$.
\begin{itemize}
\item 
Let $u:Q \to \R$ be upper semi-continuous and bounded. The function $u$ is a $\rho$-subsolution of \eqref{HJe} in $Q$ if the following holds: whenever $\phi\in C^2(Q)$ is such that $u-\phi$ has a global maximum at $(\o t, \o x)$, we have
\[
\phi_t(\o t, \o x) \leq 
\begin{cases}
\o H_{\rho,\e}[\phi,u](\o t, \o x)& \text{if $\phi_x(\o t, \o x) \neq 0$,}\\
0 & \text{if $\phi( t,  x)$ is of the form $\theta |x - \o x|^4 + g(t)$,}\\
+\infty & \text{otherwise}.
\end{cases}
\]
\item Let $v:Q \to \R$ be lower semi-continuous and bounded. The function $v$ is a $\rho$-supersolution of \eqref{HJe} in $Q$ if the following holds: whenever $\psi \in C^2(Q)$ is such that $v-\psi$ has a global minimum at $(\o t, \o x)$, we have 
\[
\psi_t(\o t, \o x) \geq 
\begin{cases}
\underline H_{\rho, \e}[\psi,v](\o t, \o x)& \text{if $\phi_x(\o t, \o x) \neq 0$,}\\
0 & \text{if $\phi( t,  x)$ is of the form $\theta |x - \o x|^4 + g(t)$,}\\
-\infty & \text{otherwise}.
\end{cases}
\]
\end{itemize}
A function $u:Q \to \R$ is a $\rho$-solution of \eqref{HJe} in $Q$ if $u^*$ is a $\rho$-subsolution and $u_*$ is a $\rho$-supersolution.
\end{defn}

We remark that, as usual, we extend subsolutions $u$ and supersolutions $v$ to $t=0$ by
\[
  u(0,x) := u^*(0,x)
  \quad \text{and} \quad
  v(0,x) := v_*(0,x)
  \quad \text{for all } x \in \R.
\]
Without loss of generality, we may restrict the class of test functions for subsolutions in Definition \ref{d:HJe:VS} to those $\phi$ for which the maximum of $u-\phi$ is strict, that $(u - \phi)(\o t, \o x) = 0$ and that $\lim_{|t| + |x| \to \infty} (u - \phi)(t,x) = -\infty$. Indeed, if the maximum at $(\o t, \o x)$ is not strict, then we can approximate $\phi$ by 
\[
  \phi_\delta(t,x) := \phi(t,x) +  \delta |x - \o x|^4 + \delta |t - \o t|^2
\]
as $\delta \to 0$. Note that the maximum of $u-\phi_\delta$ is strict for any $\delta > 0$. Moreover, it follows from Lemma \ref{l:Hpe}\ref{l:Hpe:Bp} that the right-hand side in Definition \ref{d:HJe:VS} converges as $\delta \to 0$ to that of $\phi$.
Similarly, for supersolutions we may assume that $\psi$ satisfies the same properties (modulo obvious modifications).

The following lemma shows that Definition \ref{d:HJe:VS} does not depend on $\rho$. Therefore we can simply talk about subsolutions, supersolutions and viscosity solutions of \eqref{HJe}.

\begin{lem}[Independence of $\rho$ { \cite[Lem.\ 3.5]{VanMeursPeletierPozar22}}] \label{l:HJe:rho}
If $u$ is a $\rho$-sub- or $\rho$-supersolution of \eqref{HJe} for some $\rho > 0$, then it is a $\tilde\rho$-sub- or $\tilde\rho$-supersolution of \eqref{HJe}, respectively, for any $\tilde \rho > 0$. 
\end{lem}

The proof of \cite{VanMeursPeletierPozar22} applies directly to $U$ and $V$ under Assumption \ref{a:UV}.

\begin{thm}[Comparison principle for \eqref{HJe}]  
\label{t:CPe}
Let $U$ and $V$ satisfy Assumption \ref{a:UV}. Let $\e > 0$, and let $u$ be a subsolution and $v$ be a supersolution of \eqref{HJe}. Assume that for each $T>0$, 
\begin{equation} \label{CPe:far-field}
  \lim_{R\to\infty} \sup \big\{ u(t,x)-v(t,x):  |x|\geq R, \, t\in (0,T]\big\} \leq 0.  
\end{equation}
Then $u(0,\cdot) \leq v(0,\cdot)$ on $\R$ implies $u\leq v$ on $Q$.
\end{thm}

\begin{proof}
For the special case $U = 0$, $\alpha_\e = 1$ and $V(z) = -\log|z|$, Theorem \ref{t:CPe} is given by \cite[Theorem 3.6]{VanMeursPeletierPozar22}. That proof follows a standard approach, and can be extended with minor modifications to any $U$ and $V$ which satisfy Assumption \ref{a:UV}. Here, we wish to present in detail the least obvious modification, which is the proof of $J_\delta \to J$ in $L^1(\R)$ as $\delta \downarrow 0$, where 
\[
J(z) := \begin{cases}
V_{\alpha_\e}''(z)& \text{if }|z|> \rho\\
0 & \text{otherwise}
\end{cases}
\qquad \text{and}\qquad
J_\delta(z) := \begin{cases}
\displaystyle \frac1{1-\delta} J \Big( \frac z{1-\delta} \Big) & \text{if } z < 0\\
\displaystyle \frac1{1+\delta} J \Big( \frac z{1+\delta} \Big) & \text{if } z > 0
\end{cases}
\]
with $\rho > 0$ a given constant. To prove this, we show that for all $\hat \e > 0$ we have that $\| J_\delta - J \|_{L^1(0,\infty)} < \hat \e$ for all $\delta > 0$ small enough. This is sufficient, because the analogous statement on the interval $(-\infty,0)$ follows from a similar proof. Let $\tau_\delta$ be the bounded linear operator on $L^1(0,\infty)$ defined by $(\tau_\delta f)(x) = \frac1{1+\delta} f ( \frac z{1+\delta} )$. Note that $J_\delta = \tau_\delta J$ and that the operator norm of $\tau_\delta$ equals $1$. In addition to $\tau_\delta$, let $\varphi \in C([0,\infty))$ have bounded support such that $\| J - \varphi \|_{L^1(0,\infty)} < \frac13 {\hat \e}$. We observe that
\[
  \| J_\delta - \tau_\delta \varphi \|_{L^1(0,\infty)}
  = \| \tau_\delta ( J -  \varphi) \|_{L^1(0,\infty)}
  = \| J -  \varphi \|_{L^1(0,\infty)}
  \leq \frac{\hat \e}3. 
\]
Finally, since $\varphi$ is continuous with bounded support, it holds that $\| \tau_\delta \varphi - \varphi \|_{L^1(0,\infty)} < \frac\e3$ for all $\delta > 0$ small enough. Then, $\| J_\delta - J \|_{L^1(0,\infty)} < \e$ follows from the triangle inequality.
\end{proof} 

\subsection{Viscosity solutions of \eqref{HJ1} with a comparison principle} 
\label{s:HJ:HJ1} 

Viscosity solutions to the nonlocal equation \eqref{HJ1} have been introduced in \cite{ImbertMonneauRouy08,BilerKarchMonneau10} for the integrable Riesz potentials $V$ (recall \eqref{VR}).  This work goes back to at least \cite{Sayah91}. Here, however, we follow the approach in \cite{VanMeursPeletierPozar22} to adopt the weaker notion of viscosity solutions based on the restricted class of test functions as in Definition \ref{d:HJe:VS}. The reason for this is that it will simplify the proof of the convergence of \eqref{HJe} to \eqref{HJ1} in Theorem \ref{t:conv}. By reducing the class of test functions, we have to check that \eqref{HJ1} still satisfies a comparison principle. 
Even though we consider any $V$ satisfying Assumption \ref{a:UV}, this can be done with minor modifications to the arguments in \cite{VanMeursPeletierPozar22}. Hence, we will be brief in the following.

\begin{defn}[Limiting Hamiltonian] 
\label{d:Hp}
For $\alpha, \rho>0$, $u\in L^\infty(Q)$, $t > 0$, $x \in \R$ and $\phi\in C^2(Q)$, we define
\[
H_\rho[\phi,u](t,x) 
:= \big( \mathcal I_\rho[\phi(t,\cdot),u(t,\cdot)](x) + U'(x) \big) \, |\phi_x(t,x)|, 
\]
where 
\begin{equation} \label{Irho}
\mathcal I_\rho[\psi,v](x) 
:= \int_{B_\rho} \big( \psi(x+z)-\psi(x) - \psi'(x)z\big) V_\alpha''(z) \, dz
   + \int_{B_\rho^c} \big( v(x+z)-v(x)\big) V_\alpha''(z) \, dz
\end{equation} 
for any $\psi \in C^2(\R)$ and $v \in L^\infty(\R)$.
\end{defn}

This definition has some key differences with the Hamiltonians in Definition \ref{d:Hpe}. First, $\phi_x(t,x) = 0$ is allowed. Second, the step function $E_\e$ has disappeared from the integrand; we therefore do not need to separate cases between $E_{\e,*}$ and $E_\e^*$. Third, the principle value in \eqref{Irho} has disappeared, and instead  the additional term ``$- \psi'(x)z$" is added to the integrand. To motivate this, note from the evenness of $V_\alpha''(z)$ that:
\begin{equation} \label{odd0}
  f \in C(\o{B_\rho}) \text{ is odd} \implies
  \pv \int_{B_\rho} f(z) V_{\alpha}''(z) \, dz = 0.
\end{equation}
In particular, $\pv \int_{B_\rho} z V_\alpha''(z) \, dz = 0$, and thus the additional term  ``$- \psi'(x)z$" in \eqref{Irho} does not change the value of $\mathcal I_\rho[\psi,v](x)$. The reason for adding this term is that, by interpreting $\psi(x) + \psi'(x)z$ as the first order Taylor approximation of $z \mapsto \psi(x+z)$ at $z = 0$, it follows that $\psi(x+z)-\psi(x) - \psi'(x)z = C z^2 + o(z^2)$ as $z \to 0$. Then, from Lemma \ref{l:UV}\ref{l:UV:L1} it follows that the integral over $B_\rho$ in \eqref{Irho} exists without the need for using a principal value.

\begin{defn}[$\rho$-sub- and $\rho$-supersolutions of \eqref{HJ1}] \label{d:HJ1:VS}
Let $\alpha, \rho > 0$. 
\begin{itemize}
\item 
Let $u:Q \to \R$ be upper semi-continuous and bounded. The function $u$ is a $\rho$-subsolution of \eqref{HJ1} in $Q$ if the following holds: whenever $\phi\in C^2(Q)$ is such that $u-\phi$ has a global maximum at $(\o t, \o x)$, we have
\begin{align*} 
\phi_t(\o t, \o x) \leq 
\begin{cases}
H_\rho[\phi,u](\o t, \o x)& \text{if $\phi_x(\o t, \o x) \neq 0$,}\\
0 & \text{if $\phi( t,  x)$ is of the form $\theta |x - \o x|^4 + g(t)$,}\\
+\infty & \text{otherwise}.
\end{cases}
\end{align*}
\item Let $v:Q \to \R$ be lower semi-continuous and bounded. The function $v$ is a $\rho$-supersolution of \eqref{HJ1} in $Q$ if the following holds: whenever $\psi\in C^2(Q)$ is such that $v-\psi$ has a global minimum at $(\o t, \o x)$, we have
\[
\psi_t(\o t, \o x) \geq 
\begin{cases}
H_\rho[\psi,v](\o t, \o x)& \text{if $\phi_x(\o t, \o x) \neq 0$,}\\
0 & \text{if $\phi( t,  x)$ is of the form $\theta |x - \o x|^4 + g(t)$,}\\
-\infty & \text{otherwise}.
\end{cases}
\]
\end{itemize}
A function $u:Q \to \R$ is a $\rho$-solution of \eqref{HJ1} in $Q$ if $u^*$ is a $\rho$-subsolution and $u_*$ is a $\rho$-supersolution.
\end{defn} 

Definition \ref{d:HJ1:VS} shares several similarities with Definition \ref{d:HJe:VS}. Also in Definition \ref{d:HJ1:VS} we may assume without loss of generality that the maximum of $u-\phi$ is strict, that $(u - \phi)(\o t, \o x) = 0$ and that $\lim_{|t| + |x| \to \infty} (u - \phi)(t,x) = -\infty$.
Similarly, we may assume the same properties (with obvious modifications) of $\psi$ with respect to $v$. Moreover, the notion of viscosity solutions in Definition \ref{d:HJ1:VS} is independent of $\rho$ in a similar sense as in Lemma \ref{l:HJe:rho} (see Lemma \ref{l:HJ:rho}), and therefore we will not mention the dependence on $\rho$ henceforth.

\begin{lem}[Independence of $\rho$] \label{l:HJ:rho}
Let $\alpha > 0$. If $u$ is a $\rho$-sub- or $\rho$-supersolution of \eqref{HJ1} for some $\rho > 0$, then it is a $\tilde\rho$-sub- or $\tilde\rho$-supersolution of \eqref{HJ1} for any $\tilde \rho > 0$, respectively.
\end{lem}

The proof of Lemma \ref{l:HJ:rho} is a simplification of that of Lemma \ref{l:HJe:rho}; we omit it.

\begin{thm}[Comparison principle for \eqref{HJ1}] \label{t:CP:HJ1}
Let $U$ and $V$ satisfy Assumption \ref{a:UV}. Let $u^\circ \in BUC(\R)$, and let $u$ be a subsolution and $v$ be a supersolution of \eqref{HJ1}. If $u(0,\cdot) \leq u^\circ \leq v(0,\cdot)$ on $\R$, then  $u\leq v$ on $Q$.
\end{thm}
 
\begin{proof}
For the special case $U = 0$, $\alpha = 1$ and $V(z) = -\log|z|$, Theorem \ref{t:CP:HJ1} is given by \cite[Theorem 3.11]{VanMeursPeletierPozar22}. For general $U, V, \alpha$, the proof in \cite{VanMeursPeletierPozar22} can easily be extended. Indeed, to allow for $U \neq 0$, we rely on the well-known result \cite[Prop.\ 3.7]{CrandallIshiiLions92}.  Regarding $V_{\alpha}$, we need to redo the following computation. For the special case $V_\alpha(z) = -\log|z|$, we have \comm{For CIL92 Prop 3.7 appl: see p.85}
\begin{equation*} 
  \int_{B_\rho} z^2 V_\alpha''(z) \, dz = 2 \rho
  \qand
  \int_{\R \setminus B_\rho} V_\alpha''(z) \, dz = \frac2\rho.
\end{equation*}
For any $\alpha > 0$ and any $V$ under Assumption \ref{a:UV} we only have
\begin{equation} \label{pf:yt}
  \int_{B_\rho} z^2 V''(z) \, dz = o(1)
  \qand
  \int_{\R \setminus B_\rho} V''(z) \, dz = -2V'(\rho) = o(\rho^{-2})
  \quad \text{as } \rho \to 0,
\end{equation}
which converge to $0$ at a slower rate.
Yet, this slower rate is still sufficient for the proof of \cite[Thm.\ 3.11]{VanMeursPeletierPozar22} to carry through. \comm{[ see p.36 for details]}
\end{proof}

\subsection{Viscosity solutions of ($HJ^2$) and ($HJ^3$) with a comparison principle} 
\label{s:HJ:HJ23}

The equations \eqref{HJk} for $m=2,3$ are local. They fit to the class of PDEs considered in \cite{GigaGotoIshiiSato91} for which viscosity solutions are defined and for which a comparison principle is established. Therefore, it only remains to show that the comparison principle in \cite{GigaGotoIshiiSato91} remains to hold for the restricted class of test functions introduced in Definition \ref{d:HJe:VS}.

\begin{defn}[Sub- and supersolutions for \eqref{HJk}] \label{d:HJk:VS}
Let $m \in \{2,3\}$ and $\beta > 0$.
\begin{itemize} 
\item 
Let $u:Q \to \R$ be upper semi-continuous and bounded. The function $u$ is a subsolution of \eqref{HJk} in $Q$ if the following holds: whenever $\phi\in C^m(Q)$ is such that $u-\phi$ has a global maximum at $(\o t, \o x)$, we have
\begin{align}
\label{visc-subsol-cond-k}
\phi_t(\o t, \o x) 
\leq \begin{cases}
\big( f_m (\phi_x) \phi_{xx} + U' |\phi_x| \big) (\o t, \o x)& \text{if $\phi_x(\o t, \o x) \neq 0$,}\\
0 & \text{if $\phi( t,  x)$ is of the form $\theta |x - \o x|^4 + g(t)$,}\\
+\infty & \text{otherwise}.
\end{cases}
\end{align}
\item Let $v:Q \to \R$ be lower semi-continuous and bounded. The function $v$ is a supersolution of \eqref{HJk} in $Q$ if the following holds: whenever $\psi\in C^m(Q)$ is such that $u-\psi$ has a global minimum at $(\o t, \o x)$, we have
\begin{align}
\label{visc-supersol-cond-k} 
\psi_t(\o t, \o x) 
\geq \begin{cases}
\big( f_m (\psi_x) \psi_{xx} + U' |\psi_x| \big) (\o t, \o x)& \text{if $\phi_x(\o t, \o x) \neq 0$,}\\
0 & \text{if $\phi( t,  x)$ is of the form $\theta |x - \o x|^4 + g(t)$,}\\
+\infty & \text{otherwise}.
\end{cases}
\end{align}
\end{itemize}
A function $u:Q \to \R$ is a solution of \eqref{HJk} in $Q$ if $u^*$ is a subsolution and $u_*$ is a supersolution.
\end{defn}

The choice of $C^m$-regularity of the test functions $\phi$ and $\psi$ is of little importance. We have chosen the lowest regularity for which our proof of Theorem \ref{t:conv} works (see Lemma \ref{l:conv:RHS}). 

We state the comparison principle in terms of the functions $f_m$. From $f_2(\kappa) = \|V\|_{L^1(\R)} |\kappa|$ and the properties of $f_3$ stated in Lemma \ref{l:f3} we have that $f_m$ satisfy the requirements for any $V$ under Assumption \ref{a:UV}.

\begin{thm}[Comparison principles for \eqref{HJk}] \label{t:CP:HJk}
Let $m \in \{2,3\}$, $\beta > 0$ and $u^\circ \in BUC(\R)$. 
Let $u$ be a subsolution and $v$ be a supersolution of \eqref{HJk}. 
If 
\begin{itemize}
  \item $f_m$ is locally Lipschitz continuous on $\R$ with $f_m \geq f_m(0) = 0$,
  \item $u$ and $v$ are bounded on $Q$, and
  \item $u(0,\cdot) \leq u^\circ \leq v(0,\cdot)$ on $\R$,
\end{itemize}
then  $u\leq v$ on $Q$. \comm{p.80: these assumptions on $f_m$ are what is needed for GGIS's setting}
\end{thm} 

\begin{proof} 
The proof is a modification of the proof of \cite[Theorem 1.1]{GigaGotoIshiiSato91}. In fact, Theorem \ref{t:CP:HJk} fits to the general setting in \cite{GigaGotoIshiiSato91} except for the restriction of the class of test functions in Definition \ref{d:HJk:VS}. Therefore, we will be brief when following similar arguments from \cite[Theorem 1.1]{GigaGotoIshiiSato91}.

To force a contradiction, suppose that 
\begin{equation} \label{pf:wb}
  2 c_0 := \sup_{Q_T} (u-v) > 0 
\end{equation}
for some $T > 0$. We double the spatial variable and set 
\begin{align*}  
  w(t,x,y) &:= u(t,x) - v(t,y), \\
  \Psi(t,x,y) &:= \frac{(x-y)^4}{4\e} + \frac\delta4 (x^4 + y^4) + \frac\gamma{T-t}
\end{align*}
for some (small) parameters $\e, \delta, \gamma > 0$ (the use of fourth-order confinement is a minor difference with \cite{GigaGotoIshiiSato91} in which second-order confinement is used). For all $\delta, \gamma > 0$ small enough with respect to $c_0$, we have
\[
  \sup_{(0,T) \times \R^2} (w - \Psi) > c_0.
\]
Since $u$ and $v$ are uniformly bounded, we have for $\e$ small enough with respect to $c_0$ that there exist $\o t \in (0,T)$ and $\o x, \o y \in R$ such that \comm{p.99 for details}
\begin{gather*}
  (w - \Psi)(\o t,\o x,\o y) = \sup_{(0,T) \times \R^2} (w - \Psi), \\
  |\o x - \o y| \leq C \sqrt[4]\e,
  \qquad |\o x| + |\o y| \leq \frac{C}{\sqrt[4] \delta}.
\end{gather*}

We are going to use $\phi(t,x) := \Psi(t,x, \o y)$ as a test function for $u$ and $\psi(t,y) := - \Psi(t,\o x, y)$ as a test function for $v$ (while these test functions are not defined at and beyond $T$, only the local behaviour around $\o t, \o x$ (resp.\ $\o t, \o y$) is relevant; we omit the straightforward details on how to fix this). Note that 
\begin{align*}
  \Psi_x(t,x,y) &
  = \frac1\e (x-y)^3 + \delta x^3,
  & 
  \Psi_{xx}(t,x,y) 
  &= \frac3\e (x-y)^2 + 3\delta x^2, \\
  \Psi_y(t,x,y) &
  = \frac1\e (y-x)^3 + \delta y^3,
  & 
  \Psi_{yy}(t,x,y) 
  &= \frac3\e (y-x)^2 + 3\delta y^2.
\end{align*}
In view of Definition \ref{d:HJk:VS}, we set
\[
  \o \phi_x := \phi_x(\o t,\o x) = \Psi_x (\o t,\o x,\o y)
\]
and define similarly $\o \psi_y$, $\o \phi_t$ and $\o \psi_t$. Note that 
\begin{equation} \label{pf:wa}
  \o \phi_t = - \o \psi_t = \frac\gamma{(T-\o t)^2} > 0.
\end{equation}

To reach a contradiction, we keep $\e, \gamma$ fixed and take $\delta$ small enough with respect to $c_0, \e, \gamma$. We split several cases depending on whether $\o \phi_x$ and $\o \psi_y$ are $0$ or not.
If $\o \phi_x \neq 0 \neq \o \psi_y$, then Definition \ref{d:HJk:VS} provides the same estimate as the usual definition of viscosity solutions, and thus the proof of \cite[Theorem 1.1]{GigaGotoIshiiSato91} applies with obvious modifications (needed because of our fourth-order confinement instead of second-order confinement) to reach a contradiction with \eqref{pf:wb}. In the other cases, by the symmetry in $x$ and $y$, we may assume that $\o \phi_x = 0$, i.e.
\begin{equation} \label{pf:vl}
  \o y = \big(1 - \sqrt[3]{\e \delta} \big) \o x.
\end{equation}
If also $\o \psi_y = 0$, then a similar expression holds in which $\o x$ and $\o y$ are exchanged, and thus $\o x = \o y = 0$. Then, $\phi(t,x)$ is of the form $\theta |x - \o x|^4 + g (t)$, and thus by \eqref{visc-subsol-cond-k} we have $\o \phi_t \leq 0$. This contradicts with \eqref{pf:wa}.

It is left to consider the case $\o \phi_x = 0 \neq \o \psi_y$. Then, by \eqref{visc-supersol-cond-k},
\begin{equation*}
  -\frac\gamma{(T - \o t)^2}
  = \o \psi_t 
  \geq f_m (\o \psi_y) \o \psi_{yy} + U'(\o y) |\o \psi_y|
  = - f_m (- \o \Psi_y) \o \Psi_{yy} + U'(\o y) |\o \Psi_y|.
\end{equation*}
To reach a contradiction, we show that the right-hand side vanishes in absolute value as $\delta \to 0$. Recalling $|\o x|, |\o y| \leq C / \sqrt[4] \delta$ and \eqref{pf:vl} we obtain for $\delta$ small enough
\begin{align*}
  |\o \Psi_y|
  &\leq \frac1\e |\o y - \o x|^3 + \delta |\o y|^3
  = \frac1\e \Big| - \sqrt[3]{\e \delta} \o x \Big|^3 + \delta |\o y|^3
  \leq C \delta^{1/4} \\
  |\o \Psi_{yy}|
  &\leq \frac3\e (\o y - \o x)^2 + 3 \delta |\o y|^2
  = 3 \e^{-1/3} \delta^{2/3} |\o x|^2 + 3 \delta |\o y|^2
  \leq C \e^{-1/3} \delta^{1/6}.
\end{align*}
Then, recalling that $f_m(0) = 0$ and that $f_m$ is Lipschitz (see Lemma \ref{l:f3}), we obtain 
\[
  \big| - f_m (- \o \Psi_y) \o \Psi_{yy} + U'(\o y) |\o \Psi_y| \big|
  \leq C \frac{ \delta^{5/12} }{ \e^{1/3} } + C' \delta^{1/4},
\]
which indeed vanishes as $\delta \to 0$. 
\end{proof}

\newpage

\section{Convergence of \eqref{HJe} to \eqref{HJk}} 
\label{s:conv}

In this section we prove the convergence of \eqref{HJe} to \eqref{HJk} as $\e \to 0$ for $m=1,2,3$; see Theorem \ref{t:conv}. \eqref{HJe} and \eqref{HJk} are the integrated versions of respectively \eqref{Pn} and ($P^m$); see Section \ref{s:HJ} for details. The proof of Theorem \ref{t:conv} is the essence of the proof of the second main result in this paper, which is Theorem \ref{t}. 

\begin{thm}[Convergence as $\e\to0$] \label{t:conv}
Let $m \in \{1,2,3\}$ and $\alpha_\e, \alpha, \beta$ be as in \eqref{alphaem}. Let $U$ and $V$ satisfy Assumption \ref{a:UV}. Let $u_\e$ be a sequence of subsolutions of \eqref{HJe}, and assume that $\| u_\e \|_\infty$ is uniformly bounded. Set $\o u := \limsup^* u_\e$. Then $\o u$ is a subsolution of \eqref{HJk}.

Similarly, if $v_\e$ is a sequence of uniformly bounded  supersolutions of \eqref{HJe}, then $\underline v := \liminf_* v_\e$ is a supersolution of \eqref{HJk}.
\end{thm}

We proceed the proof by a formal outline. The proof uses the standard steps of convergence of viscosity solutions to reduce Theorem \ref{t:conv} to the continuous convergence of the Hamiltonian (see Definition \ref{d:Hpe}) evaluated at any test function $\phi$. We recall that a sequence $H_\e : Q \to \R$ converges continuously to some $H : Q \to \R$ if for all $(\o t, \o x) \in Q$ and all $(t_\e, x_\e) \to (\o t, \o x)$ we have that $H_\e(t_\e, x_\e) \to H(t,x)$. By the structure of Definition \ref{d:HJe:VS}  it is natural to split the cases $\phi_x (\o t, \o x) \neq 0$ and $\phi_x (\o t, \o x) = 0$.  

In the case $\phi_x (\o t, \o x) \neq 0$, we have argued that the formal expressions of the Hamiltonians in \eqref{HJk} and \eqref{HJe} are well-defined when $u$ is replaced by $\phi$. It is then essentially left to derive the limit of
\begin{equation} \label{pf:uz}
  \cM_\e[\phi](t_\e, x_\e) = \pv \int_\R E_\e [\phi_x (t_\e, x_\e + z) - \phi_x (t_\e, x_\e)] V_{\alpha_\e}''(z) dz 
\end{equation}
as $\e \to 0$. This limit passage is done rigorously in Lemma \ref{l:conv:RHS} below, which can be seen as a generalization of \cite[Lemma 3.14]{VanMeursPeletierPozar22} to general $V$ and to the new cases $m=2,3$. We give here a formal sketch of the proof where we ignore technical issues around $z = 0$ and when $|z|$ is large. The main difficulty is the presence of the discontinuous function $E_\e$. In cases $m=1,2$ we remove it by writing $E_\e[\gamma] = \gamma + (E_\e[\gamma] - \gamma)$ as the sum of the identity and a perturbation in the form of a sawtooth function which is uniformly bounded in absolute value by $\frac\e2$ (recall Figure \ref{fig:Ee} on page \pageref{fig:Ee}). The identity part leads to the desired limit; the difficulty is in showing that the contribution of the sawtooth perturbation vanishes as $\e \to 0$. 

For $m=3$, however, the contribution of the  sawtooth perturbation does not vanish in the limit $\e \to 0$. In fact, the $k$th summand in the sum over $k$ in the definition of $f_3$ in \eqref{f23} corresponds to the $k$th and $-k$th steps of $E_\e$. To see where this comes from, let $\psi$ be the inverse of $z \mapsto \phi_x (t_\e, x_\e + z) - \phi_x (t_\e, x_\e)$. Then,
\begin{align} \notag 
  \cM_\e[\phi](t_\e, x_\e) 
  &= \pv \int_\R E_\e [y] V_{\alpha_\e}''(\psi(y)) \psi'(y) dy \\\notag
  &= \sum_{\ell \in \Z} \int_{\e \ell}^{\e(\ell+1)} \e \Big(  \ell + \frac12 \Big) V_{\alpha_\e}''(\psi(y)) \psi'(y) dy \\ \label{pf:uy}
  &= \sum_{k = 0}^\infty \int_{\e k}^{\e(k+1)} \e \Big(  k + \frac12 \Big) \big( V_{\alpha_\e}''(\psi(y)) \psi'(y) - V_{\alpha_\e}''(\psi(-y)) \psi'(-y) \big) dy.
\end{align}
To continue this computation, we Taylor expand both $\psi$ around $0$ and $V_{\alpha_\e}''$ around $\psi(0) = 1/\phi_x ( t_\e,  x_\e)$. The validity of the Taylor expansion on $V$ requires sufficient regularity. Moreover, to control the remainder term we need nonlocal information on $V$; hence the appearance of $\V$ in Assumption \ref{a:UV}.

In the case $\phi_x (\o t, \o x) = 0$, we may assume by Definitions \ref{d:HJ1:VS} and \ref{d:HJk:VS} that $\phi(t,x) = \theta |x - \o x|^4 + g(t)$. However, there is no reason for $\phi_x(t_\e, x_\e)$ to be precisely equal to $0$, and thus we have to derive the limit of $\cM_\e[\phi](t_\e, x_\e)$ in \eqref{pf:uz} again. This limit passage is done rigorously in Lemma \ref{l:parabola}. On the one hand this is easier than the case $\phi_x (\o t, \o x) \neq 0$ since the expression of $\phi$ is explicit, the function $g(t)$ in this expression is of little importance, and the limit Hamiltonian is $0$ (i.e.\ we have to prove $\cM_\e[\phi](t_\e, x_\e) \to 0$ as $\e \to 0$). On the other hand, $z \mapsto \phi_x (t_\e, x_\e + z) - \phi_x (t_\e, x_\e)$ is only invertible in an $\e$-dependent neighborhood around $0$ (it reaches its extremal point at $z = \o x - x_\e$), and thus we cannot simply use the change of variables in \eqref{pf:uy} again. To overcome this, we split two cases depending on the value of $\o x - x_\e$. If it is large enough with respect to $\e$, we apply \eqref{pf:uy} on a subdomain of the integral. Outside of this domain, the rough estimate $|E_\e(\gamma)| \leq |\gamma| + \frac\e2$ turns out to be sufficient. When $\o x - x_\e$ is small enough with respect to $\e$, then $z \mapsto \phi_x (t_\e, x_\e + z) - \phi_x (t_\e, x_\e)$ is almost flat around $z = 0$, and thus, similar to Lemma \ref{l:Hpe}\ref{l:Hpe:Bp}, the part of the integral on $B_\rho$ vanishes for some $\rho > 0$ which is large with respect to $\e$. For the part of the integral on $B_\rho^c$ it turns out that the rough estimate $|E_\e(\gamma)| \leq |\gamma| + \frac\e2$ is again sufficient. 

In addition to the convergence of $\cM_\e[\phi](t_\e, x_\e)$, Lemma \ref{l:parabola} also provides a quantitative upper bound in terms of $x_\e - \o x$, which moreover holds when $x_\e - \o x$ or $\e$ is not small. This upper bound is not required for the proof of Theorem \ref{t:conv}; we need it in the proof of Theorem \ref{t} to construct barriers. The proof for this upper bound requires a further case splitting for when $x_\e - \o x$ or $\e$ are small enough or not. 

Finally, we mention a crucial difference between Lemma \ref{l:parabola} and the corresponding result \cite[Lemma 15]{VanMeursPeletierPozar22}. In the latter, a quadratic rather than fourth-order test function is used. Then, the splitting of the integral in the parts $B_\rho$ and $B_\rho^c$ is enough, and no splitting of cases or inverse functions are needed. While this argument works for potentials with at most logarithmic singularities, it does not apply to stronger singularities (such as those for Riesz potentials). Hence, we consider Lemma \ref{l:parabola} as a novel result.

\begin{proof}[Proof of Theorem \ref{t:conv}]
We follow roughly the proof of the corresponding convergence result in \cite[Theorem 3.13]{VanMeursPeletierPozar22}. However, we give the proof in full detail because of the general potentials $U$ and $V$ and because of the dependence of $\alpha_\e$ on $\e$. The main difference with the proof of \cite[Theorem 3.13]{VanMeursPeletierPozar22} are Lemmas \ref{l:conv:RHS} and \ref{l:parabola}. 

We only prove the subsolution case; the proof for supersolutions is analogous. We start with the case $m=1$ in which $\alpha_\e \to \alpha > 0$ as $\e \to 0$. To show that $\o u$ is a subsolution, we show that it satisfies Definition \ref{d:HJ1:VS}. By Lemma \ref{l:HJ:rho} we may do this for $\rho = 1$. Let $\phi$ be a test function for $\o u$ such that $\o u - \phi$ has a strict maximum at $(\o t,\o x)$. We split 2 cases: $\phi_x (\o t,\o x) \neq 0$ and $\phi_x (\o t,\o x) = 0$. If $\phi_x (\o t,\o x) \neq 0$, then we need to prove that $\phi$ satisfies 
\begin{multline} \label{pf:zz} 
  \phi_t(\o t, \o x)
  \leq \bigg( \int_{B_\rho} \big( \phi(\o t, \o x+z)-\phi(\o t, \o x) - \phi_x(\o t, \o x)z\big) V_\alpha''(z) \, dz \\
   + \int_{B_\rho^c} \big( \o u(\o t, \o x+z) - \o u(\o t, \o x)\big) V_\alpha''(z) \, dz + U'(\o x) \bigg) \, |\phi_x(\o t, \o x)|. 
\end{multline} 
If $\phi_x (\o t,\o x) = 0$, then we may assume that $\phi(t,x) = \theta |x - \o x|^4 + g(t)$ for some $\theta \in \R$ and some $g \in C^2([0,T))$, and prove for such a $\phi$ that 
\begin{equation} \label{pf:zv}
  \phi_t(\o t, \o x)
  \leq 0.
\end{equation}

For both cases, we obtain by the usual argument that, along a subsequence of $\e$, $u_\e-\phi$ has a global maximum at $(t_\e,x_\e)$, where $(t_\e,x_\e)$ satisfies $(t_\e,x_\e)\to (\o t,\o x)$ and $u_\e(t_\e,x_\e) \to \o u(\o t,\o x)$ as $\e\to0$. 
In what follows, $\e$ is taken along such a subsequence. 
Since $u_\e$ is a subsolution of \eqref{HJe} and $\phi$ is an admissible test function at the point $(t_\e,x_\e)$, we have that $(\phi, u_\e)$ satisfies the inequality in Definition \ref{d:HJe:VS} at $(t_\e,x_\e)$. From this inequality we will prove that \eqref{pf:zz} or \eqref{pf:zv} hold by passing to the limit $\e \to 0$.

To do this, we separate the two cases belonging to \eqref{pf:zz} and \eqref{pf:zv}. We first treat the case $\phi_x (\o t,\o x) \neq 0$. By the continuity of $\phi_x$ we also have $\phi_x (t_\e,x_\e) \neq 0$ for all $\e$ small enough. Then, by Definition \ref{d:HJe:VS},
\begin{multline} \label{pf:we} 
  \phi_t(t_\e, x_\e)
  \leq \bigg( \pv \int_{B_\rho} E_\e^* \big[ \phi(t_\e, x_\e+z)-\phi(t_\e, x_\e) \big] V_{\alpha_\e}''(z) \, dz \\
   + \int_{B_\rho^c} E_\e^* \big[ u_\e(t_\e, x_\e+z) - u_\e(t_\e, x_\e)\big] V_{\alpha_\e}''(z) \, dz + U'(x_\e) \bigg) \, |\phi_x(t_\e, x_\e)|. 
\end{multline}
We pass to the limit $\e \to 0$ separately for each of the five components in \eqref{pf:we}. Clearly, $\phi_t(t_\e, x_\e) \to \phi_t(\o t, \o x)$, $\phi_x(t_\e, x_\e) \to \phi_x(\o t, \o x)$ and $U'(x_\e) \to U'(\o x)$ as $\e \to 0$. This leaves the two components corresponding to the two integrals. Noting that $V_{\alpha_\e}''(x) \to V_{\alpha}''(x)$ as $\e \to 0$ pointwise for all $x \neq 0$, it follows from the proof of \cite[Theorem 3.13]{VanMeursPeletierPozar22} that the limsup of the integral over $B_\rho^c$ in \eqref{pf:we} is bounded from above by the integral over $B_\rho^c$ in \eqref{pf:zz}. Finally, the key step is the limit passage of the integral over $B_\rho$. This is done in Lemma \ref{l:conv:RHS}. 

Next we treat the second case in which $\phi ( t, x) = \theta |x - \o x|^4 + g(t)$. Since $\o u$ is bounded and $\phi$ is an admissible test function, we have $\theta \geq 0$. By the usual approximation argument we may further assume that $\theta > 0$ and that $g(t) \to \infty$ as $t \to \infty$. If $x_\e = \o x$ along a (further) subsequence of $\e$, then $\phi_x(t_\e, x_\e) = 0$, and thus Definition \ref{d:HJe:VS} states that $\phi_t(t_\e, x_\e) \leq 0$. Passing to the limit $\e \to 0$, \eqref{pf:zv} follows. Hence, we may assume that $x_\e \neq \o x$ for all $\e$ small enough. Then, $\phi_x(t_\e, x_\e) \neq 0$ for all $\e$ small enough, and thus \eqref{pf:we} holds. Except for the integral over $B_\rho$, we can pass to the limit in $\e \to 0$ in each term in \eqref{pf:we} by repeating the arguments used for the case $\phi_x (\o t,\o x) \neq 0$. For the integral over $B_\rho$, we substitute $\phi (t, x) = \theta |x - \o x|^4 + g(t)$ and apply Lemma \ref{l:parabola} to obtain
\begin{multline*}
  | \phi_x(t_\e, x_\e) | \bigg( \pv \int_{B_\rho} E_\e^* \big[ \phi(t_\e, x_\e+z)-\phi(t_\e, x_\e) \big] V_{\alpha_\e}''(z) \, dz \bigg) \\
  = 4\theta |x_\e - \o x|^3 \bigg( \pv \int_{B_\rho} E_\e^* \big[ \theta \big( (z + (x_\e - \o x))^4 - (x_\e - \o x)^4 \big) \big] V_{\alpha_\e}''(z) \, dz \bigg)
  \leq C |x_\e - \o x|,
\end{multline*}
which vanishes as $\e \to 0$. This proves Theorem \ref{t:conv} for the case $m=1$.
\medskip 

Next we treat simultaneously the cases $m=2,3$ in which either $1 \ll \alpha_\e \ll \frac1\e$ or $\e \alpha_\e \to \beta > 0$ as $\e \to 0$. The argument is similar to the case $m=1$. Let $\phi$, $\o t$, $\o x$, $t_\e$ and $x_\e$ be as in case $m=1$ (if $m = 3$, we require in addition that $\phi \in C^3(Q)$). The analogue of \eqref{pf:zz} and \eqref{pf:zv} is
\begin{equation} \label{pf:wd}
  \phi_t(\o t, \o x)
  \leq \big( f_m (\phi_x) \phi_{xx} + U' |\phi_x| \big) (\o t, \o x). 
\end{equation}
Indeed, if $\phi_x(\o t, \o x) = 0$, then the right-hand side in \eqref{pf:wd} equals $0$, which is consistent with \eqref{pf:zv}. Again, we split the case $\phi_x(\o t, \o x) \neq 0$ from $\phi_x(\o t, \o x) = 0$. In the former, we obtain that for all $\e$ small enough \eqref{pf:we} holds. Passing to the limit in the terms in \eqref{pf:we} is obvious, except for the two terms given by the integrals. For the integral over $B_\rho^c$ we use the uniform bound on $u_\e$ and $\rho = 1$ to obtain
\begin{multline*}
  \int_{B_\rho^c} E_\e^* \big[ u_\e(t_\e, x_\e+z) - u_\e(t_\e, x_\e)\big] V_{\alpha_\e}''(z) \, dz 
  \leq ( 4 \| u_\e \|_\infty + \e ) \int_1^\infty \alpha_\e^3 V''(\alpha_\e z) \, dz \\
  \leq C \int_{\alpha_\e}^\infty \alpha_\e^2 V''(y) \, dy 
  = C \alpha_\e^2 |V'(\alpha_\e)|,
\end{multline*}
which, due to $y^2 |V'(y)| \to 0$ as $y \to \infty$, vanishes as $\e \to 0$. The limit of the integral over $B_\rho$ is given in Lemma \ref{l:conv:RHS}.

The second case $\phi_x(\o t, \o x) = 0$ follows from a similar argument as used in the case $m=1$. The key estimate is again given by Lemma \ref{l:parabola}. This completes the proof of Theorem \ref{t:conv}.
\end{proof}

\begin{lem}[Convergence of the right-hand side of \eqref{HJe}] \label{l:conv:RHS}
Let $m \in \{1,2,3\}$ with corresponding $\alpha_\e, \alpha, \beta > 0$ as in \eqref{alphaem}. Let $(\o t, \o x) \in Q$, $\phi \in C^{2 \vee m}(Q)$ with $\phi_x(\o t, \o x) \neq 0$ and $(t_\e, x_\e) \to (\o t, \o x)$ as $\e \to 0$. Then,
\begin{multline*}
  \lim_{\e \to 0} \pv \int_{-1}^1 E_\e^* [\phi(t_\e, x_\e + z) - \phi(t_\e, x_\e)] V_{\alpha_\e}''(z) \, dz \\
  = \left\{ \begin{aligned}
    &\int_{-1}^1 \big( \phi(\o t, \o x+z)-\phi(\o t, \o x) - \phi_x(\o t, \o x)z\big) V_\alpha''(z) \, dz
    &&\text{if } m =1  \\
    &\Big( \frac{f_m(\phi_x)}{|\phi_x|} \phi_{xx} \Big) (\o t, \o x)
    &&\text{if } m =2,3.
  \end{aligned} \right.
\end{multline*}
\end{lem}  
 
\begin{proof}
Since the dependence of $\phi$ on the time variable is of little importance, we simplify notation by setting \comm{Checked 13 March: On m=3 we dont need anything at $x=0$ :).}
\[
  \phi_\e(x) := \phi( t_\e, x) \qand \phi(x) := \phi(\o t, x).
\] 
To simplify the integrands, we further introduce
\[
  \varphi_\e (z) := \phi_\e(x_\e + z) - \phi_\e(x_\e)
  \qand
  \varphi (z) := \phi(\o x + z) - \phi(\o x).
\]
In the remainder, we split the cases $m=1,2,3$.
\smallskip

\noindent \textbf{Case $m=3$.} In the case $m=3$ we have that  
$$
  \beta_\e := \e \alpha_\e \to \beta > 0
  \qquad \text{as } \e \to 0.
$$ 
We need to prove that
 \begin{equation} \label{pf:zu}
  \lim_{\e \to 0} \pv \int_{-1}^1 E_\e^* [\varphi_\e(z)] V_{\alpha_\e}''(z) \, dz
  = \frac{\beta^3}{|\varphi'(0)|^3} \Psi \Big( \frac\beta{\varphi'(0)} \Big) \varphi''(0).
\end{equation}

We start with several preparations. First, we may assume without loss of generality that $\varphi'(0) = \phi'( \o x) > 0$. To see this, suppose that $\varphi'(0) < 0$. Then, using the oddness of $E_\e^*$ we rewrite
\[
  \pv \int_{-1}^1 E_\e^* [\varphi_\e(z)] V_{\alpha_\e}''(z) \, dz
  = - \pv \int_{-1}^1 E_\e^* [-\varphi_\e(z)] V_{\alpha_\e}''(z) \, dz.
\]
Since $(-\varphi)'(0) > 0$, \eqref{pf:zu} applies to the right-hand side. This yields
\begin{align*}
  -\lim_{\e \to 0} \pv \int_{-1}^1 E_\e^* [-\varphi_\e(z)] V_{\alpha_\e}''(z) \, dz
  &= -\frac{\beta^3}{|-\varphi'(0)|^3} \Psi \Big( \frac\beta{-\varphi'(0)} \Big) (-\varphi)''(0) \\
  &= \frac{\beta^3}{|\varphi'(0)|^3} \Psi \Big( \frac\beta{\varphi'(0)} \Big) \varphi''(0).
\end{align*}

Second, since $\phi_\e \to \phi$ in $C_{\loc}^3(\R)$ as $\e \to 0$ we have $\varphi_\e \to \varphi$ in $C^3([-1,1])$ as $\e \to 0$. Since $\varphi(0) = 0$ and $\varphi'(0) > 0$, there exists a neighborhood $0 \in \cN_0 \subset (-1,1)$ on which: (i) $\varphi$ is invertible, and (ii) there exists a $\delta \in (0,1)$ such that  $\| \varphi^{-1} \|_{C(\o{B_\delta})} < 1$ and $\| \varphi^{-1} \|_{C^3(\o{ B_{\delta}})} < \infty$. Then, for all $\e > 0$ small enough, $\varphi_\e$ is also invertible on $\cN_0$ and \comm{the following holds by IFT. The bd on higher order derivatives follows from the computations below for $\psi'$ and $\psi''$, and easily extend to higher order der's}
\begin{equation} \label{pf:yr}
  \| \varphi_\e^{-1} \|_{C(\o{B_\delta})} < 1
  \qand 
  \| \varphi_\e^{-1} \|_{C^3(\o{ B_\delta })} \leq C.
\end{equation}
For ease of notation, we set 
\[
  \psi = \psi_\e := \varphi_\e^{-1} \in C^3 (\o{B_\delta}).
\]
In addition to $\psi(0) = 0$, we note that 
\begin{align*}
  \psi' &= ( \varphi_\e^{-1})' 
         = \frac1{ \varphi_\e' \circ \varphi_\e^{-1} } \\
  \psi'' &= \Big( \frac1{ \varphi_\e' \circ \varphi_\e^{-1} } \Big)'
          = - \frac{\varphi_\e'' \circ \varphi_\e^{-1}}{ (\varphi_\e' \circ \varphi_\e^{-1})^3 }.
\end{align*}
In particular,
\begin{align*}
  p = p_\e
  &:= \psi'(0)
  = \frac1{\varphi_\e'(0)}
  \xto{\e \to 0} \frac1{\varphi'(0)}
  > 0 \\
  q = q_\e
  &:= \psi''(0)
  = - \frac{\varphi_\e''(0)}{\varphi_\e'(0)^3}
  \xto{\e \to 0} - \frac{\varphi''(0)}{\varphi'(0)^3}
  \in \R.
\end{align*}
The Taylor expansions of $\psi$ and $\psi'$ on $\o{B_\delta}$ at $y=0$ are
\begin{subequations} 
\begin{align} \label{pf:yp1}
  \psi(y)
  &= \psi(0) + \psi'(0) y + \frac12 \psi''(0) y^2 + R_3(y)
  = p y + \frac q2 y^2 + R_3(y)
  &&\text{for all } |y| \leq \delta, \\\label{pf:yp2}
  \psi'(y)
  &= \psi'(0) + \psi''(0) y + R_2(y)
  = p + q y + R_2(y)
  &&\text{for all } |y| \leq \delta,
\end{align}
\end{subequations}
where, by \eqref{pf:yr}, the remainder terms $R_\ell(y)$ satisfy $|R_\ell(y)| \leq C |y|^\ell$ for $\ell = 2,3$ and for some $C > 0$ independent of $\e$ and $y$. This completes the preparations.

Next we prove \eqref{pf:zu}. In the left-hand side in \eqref{pf:zu} we split the integration domain $(-1,1)$ into four parts. With this aim, let $K = K_\e \in \N$ be such that $1 \ll K \ll 1/\e$ as $\e \to 0$; we put further restrictions on $K$ later on. We take $\e$ smaller if necessary such that $\e K < \delta$. By \eqref{pf:yr} we have that $|\psi(\pm \e K)| < 1$ for all $\e$ small enough. The four parts of $(-1,1)$ are the three intervals $\psi(B_\e)$, $(-1, \psi(-\e K))$ and $( \psi(\e K), 1)$, and the remaining set $\psi(B_{\e K}) \setminus \psi(B_{\e})$, which is the union of two intervals. The contributions of $(-1, \psi(-\e K))$ and $( \psi(\e K), 1)$ to the integral in \eqref{pf:zu} can be treated similarly; we focus on the latter. Using that 
\begin{itemize}
  \item $E_\e^*[\gamma] \leq \gamma + \e$,
  \item $\| \varphi_\e \|_{C(\o{B_1(\o x)})} \leq 2 \| \phi \|_{C(\o{B_1(\o t)} \times \o{B_2(\o x)})} \leq C$,
  \item $\psi(\e K) = p \e K + O(\e K)^2 \geq \frac p2 \e K$ for all $\e$ small enough, and
  \item $V'' \geq 0$ and $V' \leq 0$, 
\end{itemize}
we obtain that
\begin{multline} \label{pf:yj}
  \int_{\psi(\e K)}^1 E_\e^* [\varphi_\e(z)] V_{\alpha_\e}''(z) \, dz  
  \leq ( C + \e ) \big( V_{\alpha_\e}'(1) - V_{\alpha_\e}'\big( \psi(\e K) \big) \big) \\
  \leq C' \frac1{\e^2 K^2} \Big(\frac p2 \e \alpha_\e K \Big)^2 \Big| V'\Big(\frac p2 \e \alpha_\e K \Big) \Big|
\end{multline}  
for all $\e$ small enough. Since $x^2 V'(x) \to 0$ as $x \to \infty$, there exists a non-increasing function $g$ such that $x^2 |V'(x)| < g(x) \to 0$ as $x \to \infty$. Recalling that $\beta_\e = \e \alpha_\e \to \beta$, we continue the estimate above by
\begin{align*}
  C \frac1{\e^2 K^2} \Big(\frac p2 \e \alpha_\e K \Big)^2 \Big| V'\Big(\frac p2 \e \alpha_\e K \Big) \Big|
  \leq \frac{C'}{\e^2 K^2} g\Big(\frac{\beta p}3 K \Big).
\end{align*}
Hence, there exists $K = K_\e \ll \frac1\e$ such that the right-hand side vanishes in the limit $\e \to 0$. In conclusion, by taking this $K$ the contributions of the intervals $(-1, \psi(-\e K))$ and $( \psi(\e K), 1)$ to the integral in the left-hand side of \eqref{pf:zu} vanish in the limit $\e \to 0$.

Next we treat the contribution of $\psi(B_\e) = (\psi(-\e), \psi(\e))$ of the integration domain in the left-hand side of \eqref{pf:zu}. We assume for convenience that $\psi(\e) \geq -\psi(-\e)$; the other case can be treated analogously. Then,
\[
  \pv \int_{\psi(B_\e)} E_\e^* [\varphi_\e (z)] V_{\alpha_\e}'' (z) \, dz
  = \pv \int_{\psi(-\e)}^{-\psi(-\e)} E_\e^* [\varphi_\e (z)] V_{\alpha_\e}'' (z) \, dz + \int_{-\psi(-\e)}^{\psi(\e)} E_\e^* [\varphi_\e (z)] V_{\alpha_\e}'' (z) \, dz.
\]
Noting that $E_\e^* [\varphi_\e (z)] = \frac\e2 \sign (z)$ for all $z \in \psi(B_\e)$, the first integral in the right-hand side vanishes. We rewrite the second integral as
\begin{align} \notag
   \int_{-\psi(-\e)}^{\psi(\e)} E_\e^* [\varphi_\e (z)] V_{\alpha_\e}'' (z) \, dz
   &= \frac\e2 \big( V_{\alpha_\e}' (\psi(\e)) - V_{\alpha_\e}' (-\psi(-\e)) \big) \\\label{pf:yo}
   &=  \frac{ \e \alpha_\e^2}2 \big( V' (\alpha_\e \psi(\e)) + V' (\alpha_\e \psi(-\e)) \big).
 \end{align} 
As preparation for rewriting the right-hand side, we first Taylor expand $\psi(\e)$ as in \eqref{pf:yp1} and then $V'$ around $\beta_\e p$. This yields, for $\e$ small enough with respect to $p$, $q$, $\delta$ and $R_3$, that, using $V \in C^4((0,\infty))$,
\begin{align}   \notag
   V' (\alpha_\e \psi(\e))
   &= V' \Big( \beta_\e p + \e \frac{\beta_\e q}2 +  \frac{R_3(\e)}\e \beta_\e \Big) \\\notag
   &= V' (\beta_\e p)  
      + \e \beta_\e \Big( \frac{q}2 + \frac{R_3(\e)}{\e^2} \Big) V'' (\beta_\e p) \\ \label{pf:ux}
   &\qquad   + \frac{\e^2 \beta_\e^2}2 \Big( \frac{q}2 + \frac{R_3(\e)}{\e^2} \Big)^2 V''' \Big( \beta_\e p + \theta_\e \e \beta_\e \Big( \frac{q}2 + \frac{R_3(\e)}{\e^2} \Big) \Big)
\end{align} 
for some $\theta_\e \in [0,1]$. Similarly, we expand 
\begin{align*}   
   V' (\alpha_\e \psi(-\e))
   &= -V' (\beta_\e p)  
      + \e \beta_\e \Big( \frac{q}2 + \frac{R_3(-\e)}{\e^2} \Big) V'' (\beta_\e p) \\
   &\qquad   + \frac{\e^2 \beta_\e^2}2 \Big( \frac{q}2 + \frac{R_3(-\e)}{\e^2} \Big)^2 V''' \Big( -\beta_\e p + \theta_\e \e \beta_\e \Big( \frac{q}2 + \frac{R_3(-\e)}{\e^2} \Big) \Big).
\end{align*}  
Substituting these expansions in the right-hand side of \eqref{pf:yo} and using that $|V'''|$ is bounded in a neighborhood of $\lim_{\e \to 0} \beta_\e p = \beta / \varphi'(0) > 0$, we obtain
 \begin{align} \label{pf:yi}   
    \frac{ \e \alpha_\e^2}2 \big( V' (\alpha_\e \psi(\e)) + V' (\alpha_\e \psi(-\e)) \big)
   = \frac{\beta_\e^3}2 q V'' (\beta_\e p) + R(\e)
 \end{align}  
for $\e$ small enough, where the remainder $R$ satisfies $|R(\e)|\leq C \e$ for some $C$ independent of $\e$.
 In conclusion,
\begin{equation*} 
   \limsup_{\e \to 0} \pv \int_{\psi(B_\e)} E_\e^* [\varphi_\e (z)] V_{\alpha_\e}'' (z) \, dz
   = -\frac{\beta^3}2 \frac{\varphi''(0)}{\varphi'(0)^3} V'' \Big( \frac\beta{\varphi'(0)} \Big).
 \end{equation*} 

Finally, we treat the contribution of the remaining part $\psi(B_{\e K}) \setminus \psi(B_\e)$ of the integration domain in the left-hand side of \eqref{pf:zu}. Employing the variable transformation $y = \varphi_\e(z)$ we get
\begin{align} \label{hiero}
   \int_{\psi(B_{\e K}) \setminus \psi(B_\e)} E_\e^* [\varphi_\e (z)] V_{\alpha_\e}'' (z) \, dz
  = \int_{B_{\e K} \setminus B_\e} E_\e^* [y] V_{\alpha_\e}'' ( \psi (y) ) \psi' (y) \, dy.
\end{align}
Setting $f(y) := V_{\alpha_\e}'' ( \psi (y) ) \psi' (y)$, the display above equals
\begin{align} \notag
  \int_{B_{\e K} \setminus B_\e} E_\e^* [y] f(y) \, dy
  &= \sum_{k = 1}^{K-1} \bigg( \int_{\e k}^{\e(k+1)} \e \Big( k + \frac12 \Big) f(y) \, dy + \int_{-\e (k+1)}^{-\e k} - \e \Big( k + \frac12 \Big) f(y) \, dy \bigg) \\\label{pf:zr}
  &=   \sum_{k = 1}^{K-1} \Big( k + \frac12 \Big) \int_k^{k+1} \e^2  \big( f(\e x) - f(- \e x) \big) \, dx.
\end{align}

To continue the estimate, we expand $f(y)$ for $|y| \leq \e K < \delta$. Similar to \eqref{pf:ux} we get for $\e$ small enough that
\begin{align*}
  V_{\alpha_\e}'' ( \psi (y) )
  &= V_{\alpha_\e}'' \Big( p y + \frac q2 y^2 + R_3(y) \Big) \\
  &= V_{\alpha_\e}'' ( p y ) 
    + \Big( \frac q2 y^2 + R_3(y) \Big) V_{\alpha_\e}''' ( p y ) \\
  & \qquad  + \frac12 \Big( \frac q2 y^2 + R_3(y) \Big)^2 V_{\alpha_\e}^{(4)} \Big( py + \theta_y \Big( \frac q2 y^2 + R_3(y) \Big) \Big)
\end{align*}
for some $\theta_y \in [0,1]$. 
Multiplying this by the expansion of $\psi' (y)$ given by \eqref{pf:yp2} we get
\begin{multline} \label{pf:wc}
  f(y)
  = p V_{\alpha_\e}'' ( p y) 
    + q y V_{\alpha_\e}'' ( p y ) 
    + \frac{p q}2  y^2 V_{\alpha_\e}''' ( p y ) \\
    + \tilde R_2(y) V_{\alpha_\e}'' ( p y ) 
    + \tilde R_3(y) V_{\alpha_\e}''' ( p y ) 
    + \tilde R_4(y) V_{\alpha_\e}^{(4)}  \Big( py + \theta_y \Big( \frac q2 y^2 + R_3(y) \Big) \Big),
\end{multline}
where the remainder terms $\tilde R_\ell(y)$ satisfy $|\tilde R_\ell(y)| \leq C |y|^\ell$ for $\ell = 2,3,4$ and for some constant $C > 0$ independent of $y$ and $\e$. To bound the term $\tilde R_2(y) V_{\alpha_\e}'' ( p y )$ in \eqref{pf:wc} in absolute value, we use $y \geq \e$, \ $\alpha_\e p y \geq \beta_\e p \geq \beta (2 \varphi'(0))^{-1} > 0$ and $x^3 V''(x) \to 0$ as $x \to \infty$ to obtain 
\[
  \big| \tilde R_2(y) V_{\alpha_\e}'' ( p y ) \big|
  = \alpha_\e^3 \big| \tilde R_2(y) V'' ( \alpha_\e p y ) \big|
  \leq \frac C{p^3 y}
  \leq \frac{C'}y.
\]
Similarly, we bound in absolute value the terms in \eqref{pf:wc} corresponding to $\tilde R_3$ and $\tilde R_4$: by using that  $\max_{x \geq c} x^{k+1} \V_k(x) < \infty$ for $k=3,4$ and that
\[
  py + \theta_y \Big( \frac q2 y^2 + R_3(y) \Big)
  \geq \frac{py}2
  \qquad \text{for all } y \in (\e, \delta)
\] 
for $\e$ small enough, we obtain
\[
 \Big| \tilde R_3(y) V_{\alpha_\e}''' ( p y ) 
    + \tilde R_4(y) V_{\alpha_\e}^{(4)}  \Big( py + \theta_y \Big( \frac q2 y^2 + R_3(y) \Big) \Big) \Big|
    \leq \frac Cy.
\] 

Assuming for the moment that $\tilde R_\ell(y) = 0$ for each $\ell = 2,3,4$ and all $y \in (\e, \delta)$, we obtain for the integrand in \eqref{pf:zr} that
\begin{align*} 
  \e^2  \big( f(\e x) - f(- \e x) \big) 
  &= \e^2  \big[ 2 q \e x V_{\alpha_\e}'' ( p \e x ) + p q ( \e x)^2 V_{\alpha_\e}''' ( p \e x) \big] \\
  &= (\e \alpha_\e)^3 q \big[ 2 x V'' ( \e \alpha_\e p x ) + \e \alpha_\e p x^2 V''' ( \e \alpha_\e p x ) \big] \\
  &= \beta_\e^3 q \frac d{dx} \big( x^2 V'' ( \beta_\e p x ) \big).
\end{align*}
Then, without assuming that $\tilde R_\ell(y) = 0$, we obtain that
\begin{align} \label{pf:ys}
  \Big| \e^2  \big( f(\e x) - f(- \e x) \big)  - \beta_\e^3 q \frac d{dx} \big( x^2 V'' ( \beta_\e p x ) \big) \Big| 
  \leq C \frac \e x.
\end{align}
From \eqref{pf:ys} it follows that the right-hand side in \eqref{pf:zr} is approximately equal to
\begin{align} \notag
  &\beta_\e^3 q \sum_{k = 1}^{K-1} \Big( k + \frac12 \Big)  \int_k^{k+1} \frac d{dx} \big( x^2 V'' ( \beta_\e p x ) \big) \, dx \\\notag
  &= \beta_\e^3 q \sum_{k=1}^{K-1}  \Big( k + \frac12 \Big) \Big( (k+1)^2 V'' ( \beta_\e p (k+1) )  -  k^2 V'' ( \beta_\e p k )  \Big) \\\label{pf:ym}
  &= \beta_\e^3 q \Big( K - \frac12 \Big) K^2 V'' ( \beta_\e p K )
     -\beta_\e^3 q \sum_{k=1}^{K-1}  k^2 V'' ( \beta_\e p k ) 
     - \frac12 \beta_\e^3 q V'' ( \beta_\e p ),
\end{align}
where the error made by this approximation is bounded by
\begin{align*} 
  \sum_{k = 1}^{K-1} \Big( k + \frac12 \Big) \int_{k}^{k+1} C \frac\e x \, dx
  \leq C \e \sum_{k = 1}^{K-1} \Big( 1 + \frac1{2k} \Big)
  \leq C' \e K.
\end{align*} 

Now we are ready to pass to the limit $\e \to 0$ in \eqref{hiero}. The computations above show that \eqref{hiero} is equal to \eqref{pf:ym} modulo an error term which is bounded in size by $C' \e K$. Since we chose $K$ such that $\e K \ll 1$, the error term vanishes in the limit $\e \to 0$. Since $V \in C^4((0,\infty))$, the third term in \eqref{pf:ym} converges to  
\[
  \frac12 \beta^3 \frac{\varphi''(0)}{\varphi'(0)^3} V'' \Big( \frac\beta{\varphi'(0)} \Big).
\]
By $x^3 V''(x) \to 0$ as $x \to \infty$, the first term in \eqref{pf:ym} vanishes as $\e \to 0$. Finally, by Lemma \ref{l:f3} the second term in \eqref{pf:ym} converges to 
\[
  \beta^3 \frac{\varphi''(0)}{\varphi'(0)^3} \Psi \Big( \frac\beta{\varphi'(0)} \Big).
\]

In conclusion, 
\[
  \limsup_{\e \to 0} \int_{\psi(B_{\e K}) \setminus \psi(B_\e)} E_\e^* [\varphi_\e (z)] V_{\alpha_\e}'' (z) \, dz \\
  = \beta^3 \frac{\varphi''(0)}{\varphi'(0)^3} \Psi \Big( \frac\beta{\varphi'(0)} \Big) 
    + \frac12 \beta^3 \frac{\varphi''(0)}{\varphi'(0)^3} V'' \Big( \frac\beta{\varphi'(0)} \Big).
\]
Together with the limits computed above on the other parts of the integration domain, we obtain \eqref{pf:zu}. This completes the proof of Lemma \ref{l:conv:RHS} for $m=3$.
\medskip

\noindent \textbf{Case $m=2$.} In the case $m=3$ we have that
$1 \ll \alpha_\e \ll \frac1\e$ as $\e \to 0$.
We need to prove that
 \begin{equation} \label{pf:zw}
  \lim_{\e \to 0} \pv \int_{-1}^1 E_\e^* [\varphi_\e(z)] V_{\alpha_\e}''(z) \, dz
  = \|V\|_{L^1(\R)} \varphi''(0).
\end{equation}
We follow a similar computation as for $m=3$ with simplifications whenever possible. Again, we may assume that $\varphi'(0) > 0$. Let $\delta, \psi, p, q$ be given as in the preparatory step in the case $m=3$. Since $\phi \in C^2(Q)$ instead of $C^3(Q)$, we have that $\varphi_\e$ and $\psi$ are of class $C^2$. Then, \eqref{pf:yr} changes into
\begin{equation} \label{pf:vj} 
  \| \psi \|_{C(\o{B_\delta})} < 1
  \qand 
  \| \psi \|_{C^2(\o{ B_\delta })} \leq C
\end{equation}
and, instead of the expansions in \eqref{pf:yt}, we have to work with the shorter expansions given by
\begin{align*} \label{pf:yp1}
  \psi(y)
  = p y + R_2(y)
  \qand
  \psi'(y) = p + R_1(y)
  \qquad \text{for all } |y| \leq \delta,
\end{align*} 
where, due to \eqref{pf:vj}, $R_\ell(y)$ satisfy $|R_\ell(y)| \leq C |y|^\ell$ with $C > 0$ independent of $\e$.

A difference with the case $m=3$ is that here the leading-order term can easily be split off from the integral in the left-hand side of \eqref{pf:zw}. With this aim, we split
\begin{equation} \label{pf:yh}
  \pv \int_{-1}^1 E_\e^* [\varphi_\e(z)] V_{\alpha_\e}''(z) dz
  = \pv \int_{-1}^1 F_\e [\varphi_\e(z)] V_{\alpha_\e}''(z) dz + \pv \int_{-1}^1 \varphi_\e(z) V_{\alpha_\e}''(z) dz,
\end{equation} 
where 
\[
  F_\e[x] := E_\e^*[x] - x
\] 
is a sawtooth function (recall Figure \ref{fig:Ee} on page \pageref{fig:Ee}). Note that 
\[
  \| F_\e \|_\infty \leq \frac\e2. 
\]   
For the second integral in the right-hand side of \eqref{pf:yh} we apply integration by parts twice. Since $V$ is singular, this requires justification. For this we refer to \cite[Lemma 3.7]{VanMeursTanaka23} and its proof. Then, integrating by parts twice, we obtain
\[
  \pv \int_{-1}^1 \varphi_\e(z) V_{\alpha_\e}''(z) \, dz
  = \int_{-1}^1 \varphi_\e''(z) V_{\alpha_\e}(z) \, dz
    + \big[ \varphi_\e V_{\alpha_\e}' -  \varphi_\e' V_{\alpha_\e} \big]_{z=-1}^1.
\]
Since $\alpha_\e \to \infty$ as $\e \to 0$, the boundary terms vanish as $\e \to 0$.  
Since $\varphi_\e'' \to \varphi'' $ in $ C([-1,1])$ and $V_{\alpha_\e} \weakto \|V\|_{L^1(\R)} \delta_0$ in the weak topology on measures (i.e.\ tested against $C_b(\R)$) as $\e \to 0$, we obtain
\[
  \lim_{\e \to 0} \int_{-1}^1 \varphi_\e''(z) V_{\alpha_\e}(z) \, dz
  = \int_{-1}^1 \varphi''(z) \|V\|_{L^1(\R)} \, d \delta_0(z)
  = \|V\|_{L^1(\R)} \varphi''(0).
\]

It is left to show that the first integral in the right-hand side of \eqref{pf:yh} vanishes as $\e \to 0$. To show this, we follow a similar computation as for the case $m=3$. We split the integration domain $(-1,1)$ into $(-1,\psi(-\delta))$, $\psi(B_\delta) \setminus \psi (B_\e)$, $\psi (B_\e)$ and $(\psi(\delta), 1)$. Similarly to \eqref{pf:yj}, we estimate, using $\max_{x \geq 1} x^2 |V'(x)| < \infty$, 
\[
  \bigg| \int_{\psi(\delta)}^1 F_\e [\varphi_\e(z)] V_{\alpha_\e}''(z) \, dz \bigg| 
  \leq \frac\e2 \alpha_\e^2 \big| V'(\alpha_\e \psi(\delta)) \big|  \\
  \leq C_\delta \e,
\] 
which shows that the contribution of $(\psi(\delta), 1)$ (and, similarly, that of $(-1,\psi(-\delta))$) to \eqref{pf:zw} vanishes in the limit $\e \to 0$. For the contribution of $\psi (B_\e)$, we again have that the integral over $(\psi(-\e), -\psi(-\e))$ vanishes. We estimate the remaining part (using $V'' \geq 0$) as
\begin{align*} 
   \bigg| \int_{-\psi(-\e)}^{\psi(\e)} F_\e [\varphi_\e (z)] V_{\alpha_\e}'' (z) \, dz \bigg|
   &\leq \frac\e2 \int_{-\psi(-\e)}^{\psi(\e)} V_{\alpha_\e}'' (z) \, dz.
 \end{align*} 
The right-hand side is equal to that of \eqref{pf:yo}. To bound it, we simplify the computation leading to \eqref{pf:yi} as
\begin{equation*} 
\begin{aligned}
  V'(\alpha_\e \psi(\e))
  &= V' (\e \alpha_\e p + \alpha_\e R_2(\e))
  = V' (\e \alpha_\e p) + \alpha_\e R_2(\e) V'' (\e \alpha_\e p + \theta_\e \alpha_\e R_2(\e)), \\
  V'(\alpha_\e \psi(-\e))
  &= -V' (\e \alpha_\e p) + \alpha_\e R_2(-\e) V'' (\e \alpha_\e p - \theta_\e \alpha_\e R_2(-\e))
\end{aligned}
\end{equation*}
for some $\theta_\e \in [0,1]$. Then, by \eqref{pf:yo} with $\e$ small enough we get, setting $\beta_\e := \e \alpha_\e$ and noting that $\beta_\e \to 0$ as $\e \to 0$,
\begin{equation*} 
  \bigg| \int_{-\psi(-\e)}^{\psi(\e)} F_\e [\varphi_\e (z)] V_{\alpha_\e}'' (z) \, dz \bigg|
  \leq \frac{\e}2 \Big( C \alpha_\e^3 \e^2 \V_2 \Big( \frac{ \e \alpha_\e p }2 \Big) \Big)
  \leq C' \beta_\e^3 \V_2 ( c \beta_\e^2 ),
\end{equation*}
which vanishes as $\e \to 0$.

Finally we treat the contribution of the remaining part $\psi(B_\delta) \setminus \psi(B_\e)$. As in \eqref{hiero}, we change the integration variable:
\begin{align} \label{pf:yg}
   \int_{\psi(B_\delta) \setminus \psi(B_\e)} F_\e [\varphi_\e (z)] V_{\alpha_\e}'' (z) \, dz
  = \int_{B_\delta \setminus B_\e} F_\e [y] V_{\alpha_\e}'' ( \psi (y) ) \psi' (y) \, dy.
\end{align}
Here, it is sufficient to show that the leading-order term with respect to $y$ of the integrand is an odd function and to show that the remaining terms yield a vanishing contribution to the integral. With this aim, we use $V \in C^3((0,\infty))$ to expand
\begin{align*}
  V_{\alpha_\e}'' ( \psi (y) ) \psi' (y)
  = p V_{\alpha_\e}'' ( p y) 
    + R_1(y) V_{\alpha_\e}'' ( p y ) 
    + R_2(y) (p + R_1(y)) V_{\alpha_\e}''' ( p y + \theta_y R_2(y) )
\end{align*}
for some $\theta_y \in [0,1]$. Since the leading-order term $p V_{\alpha_\e}'' ( p y)$ is even in $y$ and $F_\e$ is odd, their contribution to \eqref{pf:yg} vanishes, and thus
\begin{align} \notag
   \int_{\psi(B_\delta) \setminus \psi(B_\e)} F_\e [\varphi_\e (z)] V_{\alpha_\e}'' (z) \, dz
  &= \int_{B_\delta \setminus B_\e} F_\e [y] R_1(y) V_{\alpha_\e}'' ( p y ) \, dy \\\label{pf:vi}
  &\quad + \int_{B_\delta \setminus B_\e} F_\e [y] R_2(y) (p + R_1(y)) V_{\alpha_\e}''' ( p y + \theta_y R_2(y) ) \, dy.
\end{align} 

Next we show that both integrals in the right-hand side of \eqref{pf:vi} vanish as $\e \to 0$. The following estimates are different from those in the case $m=3$. We rely on $|F_\e[y]| < \frac\e2$ and $|R_j (y)| \leq C_j |y|^j$ for $j = 1,2$, and recall $\beta_\e = \e \alpha_\e \to 0$. Then, we estimate the first integral as
\begin{align} \notag
  \bigg| \int_{B_\delta \setminus B_\e} F_\e [y] R_1(y) V_{\alpha_\e}'' ( p y ) \, dy \bigg|
  &\leq C \e \int_\e^\delta y V_{\alpha_\e}'' ( p y ) \, dy \\\notag
  &= \frac C{p^2} \e \alpha_\e \int_{p \e \alpha_\e}^{\delta p \alpha_\e} z V'' ( z ) \, dz \\ \label{pf:vh}
  &\leq \frac C{p^3} (p \beta_\e) \int_{p \beta_\e}^\infty z V'' ( z ) \, dz.
\end{align}
By Lemma \ref{l:UV}\ref{l:UV:intx} this vanishes as $\e \to 0$.  

For the second integral in the right-hand side of \eqref{pf:vi} we take $\delta$ small enough with respect to both $C_2$ (which is the constant from $R_2$) and $1/\varphi'(0) = \lim_{\e \to 0} p$ such that for the argument of $V_{\alpha_\e}'''$ we have  
\[
  p y + \theta_y R_2(y)
  \geq \frac{p y}2 
  \qquad \text{for all } y \in (0, \delta)
\]
for all $\e$ small enough. Then, we obtain
\begin{align*}
  \bigg| \int_{B_\delta \setminus B_\e} F_\e [y] R_2(y) (p + R_1(y)) V_{\alpha_\e}''' ( p y + \theta_y R_2(y) ) \, dy \bigg|
  &\leq C p \e \int_\e^\delta y^2 \alpha_\e^4 \V_3 \Big( \alpha_\e \frac{p y}2 \Big) \, dy \\
  &\leq \frac{16 C}{p^3} \frac{p \beta_\e}{2} \int_{p \beta_\e/2}^{\delta p \alpha_\e/2} z^2  \V_3(z) \, dz.
\end{align*} 
By Lemma \ref{l:UV}\ref{l:UV:intx} we obtain, similarly as in \eqref{pf:vh}, that the integral vanishes as $\beta_\e \to 0$. This completes the proof of \eqref{pf:zw}.  
\medskip

\noindent \textbf{Case $m=1$.} In the case $m=1$ we have that
$\alpha_\e \to \alpha > 0$ as $\e \to 0$.
We may assume $\varphi'(0) > 0$ and need to prove that 
 \begin{equation*} 
  \lim_{\e \to 0} \pv \int_{-1}^1 E_\e^* [\varphi_\e(z)] V_{\alpha_\e}''(z) dz
  = \pv \int_{-1}^1 \varphi(z) V_\alpha''(z) dz.
\end{equation*}
We take $\delta$, $\psi$ and $F_\e$ as in the case $m=2$. It is sufficient to show that 
\begin{align} \label{pf:ye}
  \lim_{\e \to 0} \pv \int_{-1}^1 F_\e [\varphi_\e(z)] V_{\alpha_\e}''(z) \, dz &= 0, \\ \label{pf:yd}
  \lim_{\e \to 0} \bigg( \pv \int_{-1}^1 \varphi_\e(z) V_{\alpha_\e}''(z) \, dz - \pv \int_{-1}^1 \varphi(z) V_{\alpha}''(z) \, dz \bigg) &= 0.
\end{align}

The proof of \eqref{pf:ye} is similar to that in the case $m=2$; we omit it. To prove \eqref{pf:yd}, we assume for convenience that $\gamma_\e := \alpha_\e / \alpha \geq 1$. A linear change of variables reveals that
\[
  \pv \int_{-1}^1 \varphi_\e(z) V_{\alpha_\e}''(z) \, dz
  = \pv \int_{-\gamma_\e}^{\gamma_\e} \gamma_\e^2 \varphi_\e \Big( \frac x{\gamma_\e} \Big) V_{\alpha}''(x) \, dx.
\]
Since $\gamma_\e \to 1$ as $\e \to 0$ and since the integrand in the right-hand side is bounded on $B_{\gamma_\e} \setminus B_1$ uniformly in $\e$, the contribution over $B_{\gamma_\e} \setminus B_1$ vanishes as $\e \to 0$. Hence, it is left to show that
\[
  \lim_{\e \to 0} \pv \int_{-1}^1 (\tilde \varphi_\e - \varphi )(z) V_{\alpha}''(z) \, dz = 0,
\]
where $\tilde \varphi_\e(z) := \gamma_\e^2 \varphi_\e (  z / {\gamma_\e} )$. Expanding around $z=0$, we obtain
\[
  (\tilde \varphi_\e - \varphi_\e)(z)
  = 0 + z (\tilde \varphi_\e - \varphi_\e)'(0) + \frac{z^2}2 (\tilde \varphi_\e - \varphi_\e)''(\theta_z z)
\]
for some $\theta_z \in [0,1]$. Hence, 
\begin{align*}
  \bigg| \pv \int_{-1}^1 (\tilde \varphi_\e - \varphi )(z) V_{\alpha}''(z) \, dz \bigg|
  &\leq \frac12 \| (\tilde \varphi_\e - \varphi)'' \|_{C([-1,1])} \bigg( \pv \int_{-1}^1 z^2 V_{\alpha}''(z) \, dz \bigg) \\
  &\leq C \| \tilde \varphi_\e - \varphi \|_{C^2([-1,1])},
\end{align*}
which vanishes as $\e \to 0$.
\end{proof}

\begin{lem}
\label{l:parabola}
Let $m \in \{1,2,3\}$ with corresponding $\alpha_\e, \alpha, \beta > 0$ as in \eqref{alphaem}. For any $K, L > 1$ there exists a constant $C > 0$ which only depends on $V,K,L$ such that  
\begin{align}
\label{l:parabola:eq0}
0 \leq \gamma^2 \left( \pv \int_{B_1} E_\e^* \big[ K \big( (z + \gamma)^4 - \gamma^4 \big) \big] V_{\alpha_\e}''(z) \, dz\right) 
&\leq C
\end{align}
for all $\e \in (0,1)$ and all $|\gamma| \leq L$. Moreover, \eqref{l:parabola:eq0} also holds with $E_\e^*$ replaced by $E_{\e,*}$.
\end{lem}

\begin{proof} 
We start with two observations.
First, we write the argument of $E_\e^*$ as 
\begin{equation} \label{pf:xs}
  \varphi(z) 
  := K \big( (z + \gamma)^4 - \gamma^4 \big)
  = K \big( z^4 + 4 \gamma z^3 + 6 \gamma^2 z^2 + 4 \gamma^3 z \big).
\end{equation}
Note the similarity with the setting in the proof of Lemma \ref{l:conv:RHS}, where $\varphi_\e$ corresponds to the current $\varphi$ in \eqref{pf:xs}. Indeed, we will make use of several estimates in the proof of Lemma \ref{l:conv:RHS}. However, a crucial difference is that $\varphi'(0) = 4K\gamma^3$, which can be arbitrarily close to $0$, and thus we can only use 
\begin{equation} \label{pf:uw}
  \psi(y) := \varphi^{-1} (y)
  = \gamma \bigg( \sqrt[4]{1 + \frac y{K \gamma^4}} - 1 \bigg)
\end{equation}
in a $\gamma$-dependent neighborhood around $0$.

Second, we may assume that $\gamma > 0$ without loss of generality. Indeed, the case $\gamma < 0$ is readily transformed to the case $\gamma > 0$ by the change of variables  $z \mapsto -z$. In addition, it does not matter whether we use $E_\e = E_\e^*$ or $E_{\e,*}$ in \eqref{l:parabola:eq0}, because any level set of $\varphi$ consists of at most two points.

The lower bound in \eqref{l:parabola:eq0} is easy to see. Indeed, since $\varphi$ is convex, it is bounded from below by its tangent at $0$. Then, the lower bound follows by applying Lemma \ref{l:Hpe}\ref{l:Hpe:order} and \eqref{odd0}. 

In the remainder we prove the upper bound in \eqref{l:parabola:eq0}. 
As preparation, we observe from $x^2 V''(x) \in L^1(\R)$ that  
\begin{align} \label{pf:xo}
  \int_0^1 z^2 V_{\alpha_\e}''(z) \, dz 
  &= \int_0^{\alpha_\e} y^2 V''(y) \, dy
  \leq  C
\end{align}  
and from $V' \leq 0$ and {$\sup_{x > 0} x^2 |V'(x)| < \infty$} that
\begin{equation} \label{pf:wl}
  \int_\delta^1 V_{\alpha_\e}''(z) \, dz
  \leq | V_{\alpha_\e}'(\delta) |
  = \alpha_\e^2 | V'(\alpha_\e \delta) |
  \leq \frac C{\delta^{2}}
\end{equation}
for all $\delta, \e \in (0,1)$.

We prove \eqref{l:parabola:eq0} by separating four regions $\Omega_i$ in the parameter space $(\gamma, \e)$; see Figure \ref{fig:Oms}. To motivate the curve which separates $\Omega_1$ from $\Omega_2 \cup \Omega_3$, let $(a, b) \ni 0$ be the largest interval such that 
\begin{equation} \label{pf:xv}
\begin{aligned}
-\e &<  \varphi(z) < 0, & a &< z < 0,\\
0 &< \varphi(z) < \e, & 0 &< z < b.
\end{aligned}  
\end{equation}
Note that $\varphi(b) = \e$. Inverting this, we find
\begin{align} \label{pf:xq}
b 
= \psi(\e) = \gamma \bigg( \sqrt[4]{1 + \frac\e{K \gamma^4}} - 1 \bigg).
\end{align}
Note further that $\varphi(a)$ can be equal to either $0$ or $-\e$, depending on the parameters; see Figure \ref{fig:phi}. Since $\min_\R \varphi = -K\gamma^4$, we obtain that $\varphi(a) = 0$ if and only if
\begin{equation} \label{pf:xt} 
  \e > K \gamma^4.
\end{equation}
This motivates the separation of region $\Omega_1$.
\smallskip

\begin{figure}[h]
\centering
\begin{tikzpicture}[scale=3, >= latex]    
\def \gam {0.841}
        
\draw[->] (0,0) -- (0,1.2) node[left]{$\e$};
\draw[->] (0,0) -- (2.2,0) node[right] {$\gamma$};
\draw (0,1) node[left]{$1$} --++ (2,0) --++ (0,-1) node[below]{$L$};

\draw[domain=0:1, smooth, thick, blue] plot (\x,{\x^4});
\draw[blue] (1,1) node[above]{$\e = K \gamma^4$};
\draw (\gam,0) node[below]{$\gamma_0$} -- (\gam,.5);
\draw[dotted] (0,.3)  node[left]{$\e_0$} -- (\gam, 0.3);
\draw (\gam, 0.3) -- (2,.3);

\draw (.4,.5) node{$\Omega_1$};
\draw (\gam,0) node[anchor = south east]{$\Omega_3$};
\draw (1.5,.65) node{$\Omega_2$};
\draw (1.5,.15) node{$\Omega_4$};
\end{tikzpicture} \\
\caption{Division of the parameter space into $\Omega_i$ ($i = 1,\ldots,4$).}
\label{fig:Oms}
\end{figure}
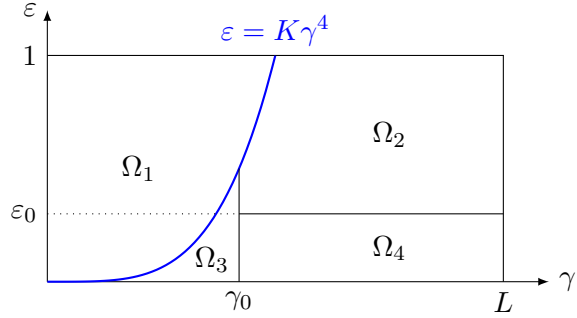 

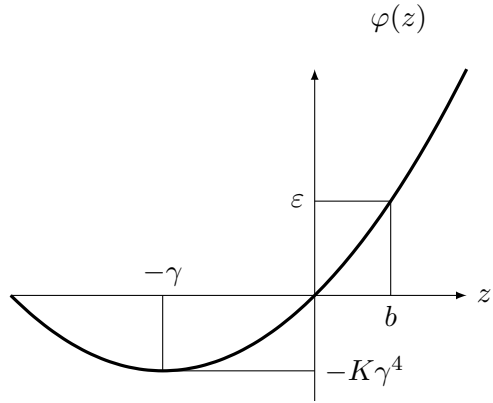
\begin{figure}[h]
\centering
\begin{tikzpicture}[scale=2, >= latex]    
\def \w {2}
        
\draw[->] (0,-.7) -- (0,1.5);
\draw[->] (-2,0) -- (1,0) node[right] {$z$};
\draw[domain=-2:1, smooth, very thick] plot (\x,{(.5 * \x * (\x+2))}); 
\draw (0.3, \w) node[anchor = north west]{$\varphi(z)$};
\draw (-1,0) node[above] {$-\gamma$} --++ (0,-.5) --++ (1,0) node[right] {$-K\gamma^4$};
\draw (.5,0) node[below] {$b$} --++ (0,.625) --++ (-.5,0) node[left] {$\e$};
\end{tikzpicture} \\
\caption{Sketch of $\varphi$.}
\label{fig:phi}
\end{figure} 

Next we prove \eqref{l:parabola:eq0} in the region $\Omega_1$, i.e.\ $\e, \gamma$ are such that \eqref{pf:xt} holds. Note that $\gamma \leq \sqrt[4]{\e / K} < 1$ and that
\[
  a = -2\gamma
  \qand
  b \geq \gamma (\sqrt[4]2 - 1) =: c_0 \gamma.
\]
Hence, $B_{c_0 \gamma} \subset (a,b)$. Then, considering \eqref{l:parabola:eq0}, we observe from \eqref{pf:xv} that
\begin{equation*} 
  \pv \int_{B_{c_0 \gamma}} E_\e[\varphi(z)] V_{\alpha_\e}''(z) \, dz = 0.  
\end{equation*}
It is therefore left to estimate the contribution of $B_1 \setminus B_{c_0 \gamma}$ to the integral in \eqref{l:parabola:eq0}. With this aim we use respectively $E_\e(x) \leq x + \frac\e2$, \eqref{odd0} and \eqref{pf:xo},\eqref{pf:wl} to obtain 
\begin{align} \notag
&\gamma^2 \int_{B_1 \setminus B_{c_0 \gamma}} E_\e[\varphi(z)] V_{\alpha_\e}''(z) \, dz \\\notag
&\leq \gamma^2 \int_{B_1 \setminus B_{c_0 \gamma}}\left(\varphi(z) + \frac\e 2\right) V_{\alpha_\e}''(z) \, dz \\ \notag
&= 2 \gamma^2 K \int_{c_0 \gamma}^1 z^4 V_{\alpha_\e}''(z) \, dz + 12 \gamma^4 K \int_{c_0 \gamma}^1 z^2 V_{\alpha_\e}''(z) \, dz + \gamma^2 \e \int_{c_0 \gamma}^1 V_{\alpha_\e}''(z) \, dz \\ \label{pf:zk}
&\leq C \big( \gamma^2 + \gamma^4 + c_0^{-2} \e   \big) 
\leq C'.
\end{align}
This proves \eqref{l:parabola:eq0} in the region $\Omega_1$.
\smallskip

Next we prove \eqref{l:parabola:eq0} in the region $(\e_0,1) \times (\gamma_0, L) \supset \Omega_2$, where $\gamma_0 \in (0,\frac12]$ is a universal constant which we choose later in \eqref{pf:wo} and 
\begin{equation} \label{pf:wp}
  \e_0 := K \gamma_0^4 \Big(1 - \Big(1- \frac1L \Big)^4 \Big)
\end{equation}
for a reason which becomes apparent later when treating $\Omega_4$. The proof of \eqref{l:parabola:eq0} above applies with minor modifications. From \eqref{pf:xq} we note that 
\[
  b 
  \geq \gamma_0 \bigg( \sqrt[4]{1 + \frac{\e_0}{K L^4}} - 1 \bigg) 
  =: b_0 \in (0, \gamma_0).
\] 
Since either $|a| > b$ or $|a| = 2 \gamma \geq \gamma_0$, we have $|a| \geq b_0$. Hence,
\[
  B_{b_0} \subset B_1 \cap (a,b),
\]
and thus the contribution of $B_{b_0}$ to the integral in \eqref{l:parabola:eq0} vanishes. We estimate the contribution of $B_1 \setminus B_{b_0}$ similarly as in \eqref{pf:zk}. This yields
\begin{align*} 
\gamma^2 \int_{B_1 \setminus B_{b_0}} E_\e[\varphi(z)] V_{\alpha_\e}''(z) \, dz 
\leq C \big( \gamma^2 + \gamma^4 + \gamma^2 \e b_0^{-2} \big)
\leq C'. 
\end{align*}
This proves \eqref{l:parabola:eq0} in region $\Omega_2$.
\medskip 

Next we prove \eqref{l:parabola:eq0} in the region $\Omega_3$, i.e.
\begin{equation} \label{pf:xp}
  \e \leq K \gamma^4 
  \qand 
  \gamma \in (0, \gamma_0).
\end{equation}
Note from \eqref{pf:xs} (see also Figure \ref{fig:phi}) that $\varphi(a) = -\e$, $b < |a| \leq \gamma$ and that $\varphi |_{(-\gamma, \infty)}$ is increasing, convex and invertible.
As in the proof of \eqref{pf:zu}, we divide the integration domain $B_1$ in \eqref{l:parabola:eq0} into (note from \eqref{pf:xt} that $(a,b) = \psi(B_\e)$)
\begin{equation} \label{pf:xr}
  (a,-b), \quad  
  (-b,b), \quad
  \psi(B_\delta) \setminus \psi(B_\e), \quad
  B_1 \setminus \psi(B_\delta),
\end{equation}
where 
\begin{equation*} 
  \delta := \max \{\e, K \gamma^6\}.
\end{equation*}
By construction, the four parts in \eqref{pf:xr} are disjoint. Moreover, all four parts are subsets of $B_1$. Indeed, since $\gamma \leq \gamma_0 \leq \frac12$ and $\e \leq K \gamma^4$, we have $\delta \leq K \gamma^4$, from which we obtain (recall  \eqref{pf:uw}) $\psi(-\delta) \geq \psi(-K \gamma^4) = -\gamma \geq -\frac12$. Since $\psi(\delta) < |\psi(-\delta)|$, we obtain $\psi(B_\delta) \subset B_1$, and thus all four parts in \eqref{pf:xr} are contained in $B_1$.

Next we estimate the integral in \eqref{l:parabola:eq0} separately over each of the four parts in \eqref{pf:xr}.
The part over $(-b,b)$ vanishes and the part over $(a,-b)$ is nonpositive. To estimate the part over $B_1 \setminus \psi(B_\delta)$, we observe (recall \eqref{pf:xp}) from
$$E_\e[\varphi(z)] 
\leq \varphi(z) + \frac\e2 
\leq K \Big( (z + \gamma)^4 - \frac12 \gamma^4 \Big)
\leq C (z^4 + \gamma^4 ),$$ 
$|\psi(-\delta)| > \psi(\delta)$ and \eqref{pf:xo},\eqref{pf:wl} that 
\begin{align} \notag  
\gamma^2 \int_{B_1 \setminus \psi(B_\delta)} E_\e[\varphi(z)] V_{\alpha_\e}''(z) \, dz 
&\leq 2 C \gamma^2  \bigg( \int_0^1 z^4 V_{\alpha_\e}''(z) \, dz  + \gamma^4 \int_{\psi(\delta)}^1 V_{\alpha_\e}''(z) \, dz \bigg) \\\label{pf:yl}
&\leq C' \gamma^2 \big( 1  + \gamma^4 \psi(\delta)^{-2} \big). 
\end{align} 
It is left to estimate $\psi(\delta)$ from below. Using that $\delta \geq K \gamma^6$ and $\gamma \leq \frac12$, we obtain (recall \eqref{pf:uw})
\begin{equation} \label{pf:vf}
  \psi(\delta)
  \geq \gamma \big( \sqrt[4]{1 + \gamma^2} - 1 \big)
  \geq \big( \sqrt[4]{2} - 1 \big) \gamma^3,
\end{equation}
where the second inequality follows from the following inequality:
\begin{equation} \label{pf:wk}
  \sqrt[4]{1 + x} - 1 \geq (\sqrt[4]{2} - 1) x
  \quad \text{for all } x \in [0,1].
\end{equation}
To see that \eqref{pf:wk} holds, note that $\sqrt[4]{1 + x} - 1$ is concave and that the line given by the graph of $(\sqrt[4]{2} - 1) x$ intersects the graph of $\sqrt[4]{1 + x} - 1$ at $x=0$ and at $x=1$. From \eqref{pf:vf} we observe that the right-hand side in \eqref{pf:yl} is bounded uniformly in $\gamma$ and $\delta$, and therefore the contribution of $B_1 \setminus \psi(B_\delta)$ to the integral satisfies \eqref{l:parabola:eq0}.

Finally, we estimate the part over $\psi(B_\delta) \setminus \psi(B_\e)$. If $\e \geq K \gamma^6$, then $B_\delta = B_\e$ and the contribution to the integral in \eqref{l:parabola:eq0} vanishes. Hence, we may assume that
\[
  \e < K \gamma^6 = \delta.
\]
Since $\gamma \leq \gamma_0 \leq \frac12$, we have that $B_\delta$ remains away from the endpoint $-K\gamma^4$ of the domain of $\psi$. 

We follow a similar estimate as for \eqref{pf:yg} with modification. As preparation for this, we compute
\begin{align*} 
  \psi'(y) &= \frac1{4K \gamma^3} \Big( 1 + \frac y{K \gamma^4} \Big)^{-\tfrac34} \\
  \psi''(y) &= \frac{-3}{16 K^2 \gamma^7} \Big( 1 + \frac y{K \gamma^4} \Big)^{-\tfrac74}
\end{align*}
for all $|y| \leq \delta$ (recall $\delta \leq \frac14 K \gamma^4$), and Taylor expand around $y = 0$
\begin{align*} 
  \psi(y)
  &= \frac y{4K \gamma^3} + R_2^\gamma(y) \\
  \psi'(y)
  &= \frac1{4K \gamma^3} + R_1^\gamma(y),
\end{align*}
where the remainder terms $R_j^\gamma(y)$ satisfy  
\begin{equation*} 
  |R_j^\gamma(y)| 
  \leq \frac{ C_j |y|^j }{ K^2 \gamma^7 }
  \quad \text{for } j=1,2 \text{ and all } |y| \leq \frac14 K \gamma^4
\end{equation*}
for some universal constants $C_j > 1$.
Then, similar to \eqref{pf:vi} with 
$$
  p_\gamma := \frac1{4 K \gamma^3}, 
$$  
we obtain
\begin{multline} \label{pf:zi}
\gamma^2 \int_{\psi(B_\delta) \setminus \psi(B_\e)} E_\e[\varphi(z)] V_{\alpha_\e}''(z) \, dz 
= \int_{B_{\delta} \setminus B_\e} \gamma^2 E_\e [y] V_{\alpha_\e}'' ( p_\gamma y ) R_1^\gamma(y) \, dy \\
 + \int_{B_{\delta} \setminus B_\e} \gamma^2 E_\e [y] R_2^\gamma(y) V_{\alpha_\e}''' \big( p_\gamma y + \theta_y R_2^\gamma(y) \big) \big( p_\gamma + R_1^\gamma(y) \big) \, dy
\end{multline}
for some $\theta_y \in [0,1]$.

Next we estimate both integrals in the right-hand side of \eqref{pf:zi} from above. With respect to the estimate on \eqref{pf:yg} we apply two modifications. First, $E_\e [y]$ appears in the integrand instead of $F_\e [y]$, and we estimate it as $|E_\e [y]| \leq |y| + \e \leq 2|y|$. Second, we have to keep track of the dependence on $\gamma$. We estimate the first integral as
\begin{align*} 
  \int_{B_{\delta} \setminus B_\e} \gamma^2 E_\e [y] V_{\alpha_\e}'' ( p_\gamma y ) R_1^\gamma(y) \, dy
  &\leq \frac{4 C_1}{K^2 \gamma^5} \int_\e^\delta y^2 V_{\alpha_\e}'' ( p_\gamma y ) \, dy \\
  &= \frac{64 C_1 \gamma}{p_\gamma} \int_{\alpha_\e p_\gamma \e}^{\alpha_\e p_\gamma \delta} z^2 V'' ( z ) \, dz \\ 
  &\leq C \gamma^4 \int_0^\infty z^2 V'' ( z ) \, dz 
  \leq C' \gamma^4.
\end{align*}
For the second integral, we first note that for any
\begin{equation} \label{pf:wo}
  \gamma \leq \gamma_0 := \frac12 \wedge \frac1{\sqrt{8 C_2}}
\end{equation}
and any $0 < y \leq \delta = K \gamma^6$ we have
\begin{subequations} \label{pf:xj}
\begin{align} \label{pf:xja}
  \big| p_\gamma + R_1^\gamma(y) \big|
  &\leq \frac1{4K \gamma^3} \Big( 1 + \frac{C_1 y}{K \gamma^4} \Big)
  \leq C p_\gamma,
  \\ \label{pf:xjb}
  p_\gamma y + \theta_y R_2^\gamma(y) 
  &\geq \frac{y}{K \gamma^3} \Big( \frac14 -  \frac{C_2 y}{K \gamma^4} \Big)
  \geq p_\gamma y ( 1 -  4 C_2 \gamma^2 )
  \geq \frac{p_\gamma y}2.
\end{align}
\end{subequations}
Then, the second integral is bounded from above by
\begin{align*} 
  &\int_{B_{\delta} \setminus B_\e} \gamma^2 E_\e [y] R_2^\gamma(y) V_{\alpha_\e}''' \big( p_\gamma y + \theta_y R_2^\gamma(y) \big) \big( p_\gamma + R_1^\gamma(y) \big) \, dy \\
  &\leq C \int_\e^\delta \frac{ y^3 }{K^2 \gamma^5}  \alpha_\e^4 \V_3 \Big( \alpha_\e \frac{p_\gamma y}2 \Big) p_\gamma \, dy \\
  &= 2^{10} C K \gamma^4 \int_{ \frac12 \alpha_\e p_\gamma \e}^{\frac12 \alpha_\e p_\gamma \delta} z^3 \V_3 (z) \, dz \\
  &\leq C' \int_0^\infty z^3 \V_3 (z) \, dz
  \leq C''.
\end{align*}
This completes the proof of \eqref{l:parabola:eq0} in the region $\Omega_3$. 
\smallskip

Finally, we prove \eqref{l:parabola:eq0} in the region $\Omega_4 = (0, \e_0) \times (\gamma_0, L)$. The proof in $\Omega_3$ applies with the only modification (and two ramifications thereof) that we choose a different $\delta$. In what follows, we only focus on these ramifications.

Part of the need for a different $\delta$ is to guarantee that the four parts in \eqref{pf:xr} are contained in $B_1$. We have $\psi(B_\e) \subset B_1$ by the choice of $\e_0$ in \eqref{pf:wp}, since
\begin{align*}
  \psi(\e)
  \leq |\psi(-\e)|
  = \gamma \bigg( 1 - \sqrt[4]{1 - \frac{\e}{K \gamma^4}} \bigg)
  \leq L \bigg( 1 - \sqrt[4]{1 - \frac{\e_0}{K \gamma_0^4}} \bigg) = 1.
\end{align*}
Similarly, if $\delta \in [\e, \e_0 \wedge K \gamma^4]$, then $\varphi^{-1}$ is defined on $B_\delta$, $\psi(B_\delta) \subset B_1$ and thus the four parts in \eqref{pf:xr} are contained in $B_1$. For reasons that become apparent later, we take
\[
  \delta := \e \vee \frac{K \gamma^4}{8 C_2} \wedge \e_0.
\]

The first ramification is that \eqref{pf:vf} changes into
\[
  \psi(\delta) 
  = \gamma \bigg( \sqrt[4]{1 + \frac{\delta}{K \gamma^4}} - 1 \bigg)
  \geq \gamma_0 \bigg( \sqrt[4]{1 + \min \Big\{ \frac{\e_0}{KL^4}, \frac1{8 C_2}  \Big\}} - 1 \bigg)
  \geq c.
\]
With this \eqref{pf:yl} can again be bounded from above uniformly in $\gamma$ and $\delta$.
The second ramification concerns the integral over $\psi(B_\delta) \setminus \psi(B_\e)$. If $\delta = \e$, then the integration domain is empty, and thus we may assume 
\begin{align*}
  \e < \delta \leq \frac{K \gamma^4}{8 C_2}.
\end{align*}
Then, the estimates in \eqref{pf:xj} change. However, it still holds that the left-hand sides are bounded by the right-hand sides (in \eqref{pf:xjb} we rely on $y \leq \delta \leq K \gamma^4 / (8 C_2)$), possibly for a different constant $C$ in \eqref{pf:xja}.
This completes the proof of \eqref{l:parabola:eq0}.
\end{proof}

\newpage

\section{Convergence of \eqref{Pn} to ($P^m$)}
\label{s:t}  

In this section we state and prove Theorem \ref{t}, which is the second of the two main results in this paper. It specifies the sense in which solutions of \eqref{Pn} converge to a solution of ($P^m$) (see Section \ref{s:intro:aim} for the definitions of \eqref{Pn} and ($P^m$)). We do this in terms of the viscosity solution framework introduced in Section \ref{s:HJ}. Whereas the PDEs ($P^m$) are only stated formally and the only rigorous meaning we have given to them are the viscosity solutions to \eqref{HJk} (see Definitions \ref{d:HJ1:VS} and \ref{d:HJk:VS}), the ODE system \eqref{Pn} and the corresponding Hamilton-Jacobi equation \eqref{HJe} (see Definition \ref{d:HJe:VS}) are both defined. While in the proof of Theorem \ref{t} it suffices to show that for any solution $(\bx,\bb)$ of \eqref{Pn} we can build a corresponding solution $v$ of \eqref{HJe}, it also holds that we can characterize any solution $(\bx,\bb)$ of \eqref{Pn} as the unique solution of \eqref{HJe} for a carefully selected initial condition. We state and prove this in Proposition \ref{p:Pn:to:HJe:un}.

We start by stating Theorem \ref{t}. Given $(\bx, \bb) \in \cZ_n$ (recall \eqref{Zn}), we recall from \eqref{un} the related piecewise constant function given by 
\begin{equation} \label{un:repd} 
  u_n(x) = \frac1n \sum_{i=1}^n b_i H(x - x_i).
\end{equation}
 
\begin{thm}[Convergence of \eqref{Pn} to \eqref{HJk}] \label{t}
Let $m \in \{1,2,3\}$ with corresponding $\alpha_n$, $ \alpha$, $\beta > 0$ as in \eqref{alphanm}. Let $U$ and $V$ satisfy Assumptions \ref{a:UV:fg} and \ref{a:UV}.
For each $n \geq 2$, let $(\bx^\circ, \bb^\circ)_n \in \cZ_n$ be such that the related sequence $u_n^\circ$ of functions defined by \eqref{un:repd} satisfies $u_n^\circ \to u^\circ$ uniformly on $\R$ as $n \to \infty$ for some $u^\circ \in BUC(\R)$. Then, the solution $(\bx, \bb)_n$ to $(P_n)$ (see Definition \ref{d:Pn}) with initial datum $(\bx^\circ, \bb^\circ)_n$, interpreted as $u_n$, converges locally uniformly as $n \to \infty$ to the viscosity solution $u$ of \eqref{HJk} (see Definitions \ref{d:HJ1:VS} and \ref{d:HJk:VS}) with initial datum $u^\circ$. Furthermore, $u \in C_b(\o Q)$ with $\|u\|_\infty = \|u^\circ\|_\infty$.
\end{thm}

In Section \ref{s:disc:weak:V} we discuss how the assumptions on $V$ can be weakened for different choices of $m$.

Next we sketch the proof of Theorem \ref{t}. It follows the usual approach. We set
\begin{equation} \label{pf:t}
  \underline u := {\liminf}_* u_n \leq {\limsup}^* u_n =: \o u
  \quad \text{on } Q,
\end{equation}
and use the convergence result in Theorem \ref{t:conv} to show that $\underline u$, $\o u$ are respectively super- and subsolutions of \eqref{HJk}. For the opposite inequality, i.e.\ $\underline u \geq \o u$, we rely on the comparison principles of \eqref{HJe} ($\e := \frac1n$) and \eqref{HJk} (Theorems \ref{t:CPe}, \ref{t:CP:HJ1}, \ref{t:CP:HJk}). For this approach to work, we require two new statements:
\begin{itemize}
  \item $u_n$ is a solution of \eqref{HJe}. We prove this in Lemma \ref{l:Pn:to:HJe}. 
  \item there exist sub- and supersolutions (barriers) of \eqref{HJe} which are at initial time close to $u^\circ$ and independent of $\e$. We construct these barriers in Lemma \ref{l:barrier}. The construction relies on Definition \ref{def:4th:well}. 
\end{itemize}

\begin{defn}[Fourth-order wells] \label{def:4th:well}
We say that a function $\phi \in C(\R)$ has fourth-order wells with respect to a constant $K > 0$ if for all $\o x \in \R$ there exists $x_0 \in \R$ such that
\begin{equation} \label{4th:well}
  \phi(\o x + x) - \phi(\o x) \leq K [ (x - x_0)^4 - x_0^4 ]
  \qquad \text{for all } x \in \R.
\end{equation}
\end{defn} 

One can interpret the condition \eqref{4th:well} as follows: the graph of $\phi$ can be touched at any point $\o x$ from above by a horizontally and vertically shifted version of $x \mapsto Kx^4$ for some $K > 0$ large enough. From this viewpoint it is easy to see that if $\phi$ satisfies \eqref{4th:well} and $\phi$ is differentiable at $\o x$, then the derivatives of both sides in \eqref{4th:well} have to be equal at $x = 0$. This fixes $x_0$ by
\begin{equation} \label{x03}
  x_0^3 = \frac{- \phi' (\o x)}{4K}.
\end{equation}
We further note that if $\phi \in C(\R)$ has fourth-order wells, then $\phi$ is locally Lipschitz continuous and $\phi(x) - \lambda x^2$ is locally concave for $\lambda$ large enough. Finally, the class of functions with fourth-order wells does not contain or is contained in $C^k(\R)$ for any $k \geq 1$. Indeed, $x \mapsto - |x|$ has fourth-order wells, but the negative of the usual mollifier in $C_c^\infty(\R)$ does not satisfy \eqref{4th:well} at $\o x = 0$.

Definition \ref{def:4th:well} does not appear in \cite{VanMeursPeletierPozar22}. In \cite{VanMeursPeletierPozar22} it is enough to work with $K$-semiconcave functions (which would correspond to functions of \textit{second}-order wells). However, this notion only works when the singularity of $V$ is at most logarithmic. 

\begin{proof}[Proof of Theorem \ref{t}]
Consider \eqref{pf:t}. Lemma \ref{l:Pn:to:HJe} below shows that $u_n$ is a viscosity solution of \eqref{HJe} with $\e = \frac1n$. Since $\| u_n \|_\infty \leq 1$, we obtain from Theorem \ref{t:conv} that  $\underline u$, $\o u$ are respectively super- and subsolutions of \eqref{HJk}. Note from $\| u_n \|_\infty \leq 1$ that
\begin{equation} \label{pf:vc}
  \| \underline u \|_\infty \leq 1
  \qand
  \| \o u \|_\infty \leq 1.
\end{equation}

Next we use the comparison principles of \eqref{HJe} and \eqref{HJk} (Theorems \ref{t:CPe}, \ref{t:CP:HJ1} and \ref{t:CP:HJk}) to prove the opposite inequality $\o u \leq \underline u$ on $Q$. With this aim, we first construct a regular upper bound $u_\delta^\circ$ of $u^\circ$.
Since $u^\circ \in BUC(\R)$, there exists a sequence $(u^\circ_\delta)_\delta \subset C^1(\R)$ such that for all $\delta$ small enough
\begin{itemize}
  \item $\| u^\circ_\delta \|_{C^1(\R)} < \infty$,
  \item $u^\circ_\delta \geq u^\circ + \delta$ on $\R$, 
  \item $(u^\circ_\delta)' = 0$ on $B_{1/\delta}^c$,
  \item $u^\circ_\delta$ has fourth-order wells with respect to some $K_\delta > 0$ (see Definition \ref{def:4th:well}),
  \item $u^\circ_\delta(x) \to u^\circ(x)$ as $\delta \to 0$ for all $x \in \R$.
\end{itemize}
A possible construction of $u^\circ_\delta$ is as follows. Take 
\[
  \hat u_\delta^\circ(x) := \left\{ \begin{aligned}
    &\sup_{(-\infty, -\delta^{-1}+1)} u^\circ
    &&\text{if } x < -\delta^{-1} + 1 \\   
    &u^\circ(x)
    &&\text{if } |x| \leq \delta^{-1} - 1  \\
    &\sup_{(\delta^{-1}-1, \infty)} u^\circ
    &&\text{if } x > \delta^{-1} - 1.
  \end{aligned} \right.
\]
Note that $\hat u_\delta^\circ \geq u^\circ$ and that $\hat u_\delta^\circ$ may be discontinuous at $x = \pm (\delta^{-1} -1)$. Then, mollify $\hat u_\delta^\circ(x) + 2 \delta$ with sufficiently small, $\delta$-dependent radius such that the mollified function, say $\tilde u_\delta^\circ$, satisfies the first three conditions. Obviously, $\tilde u_\delta^\circ$ also satisfies the fifth condition. Finally, take \comm{p.89 for details on props $u_\delta^\circ$}
\[
  u_\delta^\circ (x)
  := \inf \{ \phi (x) \mid \phi \geq \tilde u_\delta^\circ \text{ and } \phi(y) = K_\delta (y - y_0)^4 + \phi_0 \text{ for some } y_0, \phi_0 \in \R \}
  \geq \tilde u_\delta^\circ (x),
\]
which by construction has fourth-order wells with constant $K_\delta > 0$. Note that $u_\delta^\circ$ is in $C^1(\R)$ and that for $K_\delta$ large enough we have $u_\delta^\circ \leq \tilde u_\delta^\circ + \delta$. Hence, $u_\delta^\circ$ satisfies all required conditions. 

By Lemma~\ref{l:barrier} below, there exists $\sigma_\delta > 0$ independent of $\e$ such that $u_\delta (t,x) := u^\circ_\delta(x) + \sigma_\delta t$ is a supersolution for \eqref{HJe} on $Q_T$ for any $T > 0$. We set $T=1$ in the remainder. 

Next we show that the comparison principle for \eqref{HJe} (Theorem \ref{t:CPe}) applies on $Q_T$ with $u_n^*$ as the subsolution, $u_\delta$ as the supersolution, $n$ large enough with respect to $\delta$ and $\delta > 0$ arbitrary, i.e.\ we need to show \eqref{CPe:far-field} and that $u_n^*(0, \cdot) \leq u_\delta^\circ$. We start with the latter. Since $\bx$ is continuous (in particular at $t = 0$), we have that $u_n^*(0, \cdot) = u_n^{\circ, *}$. By construction (recall \eqref{un:repd}), $u_n^{\circ, *} \leq u_n^{\circ} + \frac1n$ and $u_\delta^\circ \geq u^\circ + \delta$. Then, from the uniform convergence of $u_n^{\circ}$ to $u^{\circ}$, we conclude that $u_n^*(0, \cdot) \leq u_\delta^\circ$ for $n$ large enough with respect to $\delta$. It is left to show \eqref{CPe:far-field} with $T=1$. Since $\bx \in C([0,1], \R^n)$, we have that $R := \max_{1 \leq i \leq n} \max_{0 \leq t \leq 1} |x_i(t)| < \infty$. Note that $u_n^*(t, x) = u_n^\circ(x)$ for all $t \in [0,1]$ and all $|x| > R$. Then, for any $t \in [0,1]$ and any $|x| > R$,
\[
  u_n^*(t, x) - u_\delta(t,x) 
  = u_n^\circ(x) - u_\delta^\circ(x) - \sigma_\delta t
  \leq u_n^\circ(x) - u^\circ(x) - \delta,
\]
which is negative for $n$ large enough with respect to $\delta$. This shows that \eqref{CPe:far-field} is satisfied. 

In conclusion, Theorem \ref{t:CPe} applies for all $\delta > 0$ small enough and all $n$ large enough with respect to $\delta$. Hence, $u_n^*(t,x) \leq u^\circ_\delta(x) + \sigma_\delta t$ for all $(t,x) \in Q_T$ and all $n$ large enough, and therefore $\o u(0, \cdot) \leq u^\circ_\delta$. Using a similar argument for approximation from below, in the limit $\delta \to 0$ we obtain
\begin{equation*} 
  u^\circ \leq \underline u (0, \cdot) \leq \o u (0, \cdot) \leq u^\circ.
\end{equation*}
Hence, $\underline u (0, \cdot) = \o u (0, \cdot) \in BUC(\R)$, and thus the comparison principles for \eqref{HJk} (Theorems \ref{t:CP:HJ1} and \ref{t:CP:HJk} (recall \eqref{pf:vc})) yield $\o u \leq \underline u$ on $Q$. Hence, the inequality in \eqref{pf:t} has to be an equality, and thus $u := \underline u \ (= \o u)$ is the viscosity solution of \eqref{HJk} and $u_n \to u$ locally uniformly as $n \to \infty$.

Finally, $\|u\|_\infty = \|u^\circ\|_\infty$ is obvious from the comparison principle because constants are solutions of \eqref{HJk}. 
\end{proof}

\begin{lem}
\label{l:Pn:to:HJe}
Consider the setting of Theorem \ref{t}. Then for any $n \geq 2$, $u_n$ is a viscosity solution to \eqref{HJe} with $\e = \frac1n$.
\end{lem}

We refer to \cite[Lemma 4.2]{VanMeursPeletierPozar22} for the proof of Lemma \ref{l:Pn:to:HJe}. It needs minor modifications and relies on Theorem \ref{t:Pn}\ref{t:Pn:C12},\ref{t:Pn:LB:col}.

\begin{lem}
\label{l:barrier}
Let $m \in \{1,2,3\}$ with corresponding $\alpha_\e, \alpha, \beta > 0$ as in \eqref{alphaem}. Let $v_0 \in C^1 (\R)$ be such that 
\begin{itemize}
  \item $v_0' = 0$ on $B_R^c$ for some $R > 0$, and
  \item $v_0$ has fourth-order wells with respect to some $K > 1$ (see Definition \ref{def:4th:well}). 
\end{itemize}
 Then there exists $\sigma > 0$ which may depend on $R, K, V, U, \|v_0\|_{C^1(\R)}$ and $\inf_{0 < \e <1 } \alpha_\e$ (and $\sup_{0 < \e <1 } \alpha_\e$ if $m=1$) such that $v(t,x) := v_0(x) + \sigma t$ is a supersolution of \eqref{HJe} for all $\e \in (0,1)$. 
\end{lem}

\begin{proof}  
Set $\alpha_* := \inf_{0 < \e <1 } \alpha_\e > 0$ and $\alpha^* := \sup_{0 < \e <1 } \alpha_\e$, which is finite if $m=1$. For a certain $\sigma > 0$ which we specify later, we check that $v$ satisfies Definition \ref{d:HJe:VS} with $\rho = 1$. Suppose $v - \psi$ has a global minimum at $(\o t, \o x) \in Q$. Since $v \in C^1(Q)$, we have
\begin{gather} \label{pf:sig:phit}
  0 
  < \sigma 
  = v_t (\o t, \o x)
  = \psi_t (\o t, \o x), \\\notag
  \gamma 
  := v_0' (\o x)
  = v_x (\o t, \o x)
  = \psi_x (\o t, \o x).
\end{gather}
In particular, if $\psi_x(\o t, \o x) = 0$, then \eqref{pf:sig:phit} is consistent with Definition \ref{d:HJe:VS}. Hence, we may assume that $\gamma \neq 0$. Then, we write
\[
  \underline H_{\rho, \e}[\psi, v](\o t, \o x)
  = \underline H_{1, \e}[\psi, v_0](\o t, \o x)
  = |\gamma| (\underline M_{1, \e}[\psi, v_0](\o t, \o x) + U'(\o x)).
\]

We bound separately the components of the right-hand side from above. First,
\begin{align*}
  |\gamma| U'(\o x)
  \leq |v_0' (\o x)| (U'(0) + L_{U'} |\o x|)
  \leq \| v_0' \|_\infty (|U'(0)| + L_{U'} R),
\end{align*}
where $L_{U'} \geq 0$ is the Lipschitz constant of $U'$. Second, we bound the second term of $\gamma \underline M_{1, \e}[\psi, v_0](\o t, \o x)$ in \eqref{uMrhoe} as
\begin{align*}
  |\gamma| \int_{B_1^c} E_{\e, *}[v_0(\o x + z) - v_0(\o x)]  \, V_{\alpha_\e}''(z) \, dz
  &\leq \| v_0' \|_\infty \int_{B_1^c} \Big( 2 \| v_0 \|_\infty + \frac\e2 \Big) \, V_{\alpha_\e}''(z) \, dz \\
  &\leq \| v_0' \|_\infty ( 4 \| v_0 \|_\infty + 1 ) \, |V_{\alpha_\e}'(1)|,
\end{align*}
which we claim to be uniformly bounded in $\e$. Indeed, for $m=2,3$ we have $V_{\alpha_\e}'(1) = \alpha_\e^2 V'(\alpha_\e) \to 0$ as $\e \to 0$, and for $m=1$ we have by the monotonicity of $V'$ that $|V_{\alpha_\e}'(1)| \leq (\alpha^*)^2 |V'(\alpha_*)| < \infty$. Third, we bound the first term of $\underline M_{1, \e}[\psi, v_0](\o t, \o x)$ in \eqref{uMrhoe}. Let $\phi(z) := \psi(\o t, \o x + z) - \psi(\o t, \o x)$. Since $v_0$ has fourth-order wells with respect to $K$, we have
\begin{align*}  
\phi(z) 
\leq v(\o t, \o x + z) - v(\o t, \o x) 
= v_0(\o x + z) - v_0(\o x) 
\leq K [ (z - z_0)^4 - z_0^4 ]
\end{align*}
for all $z \in \R$, where (recall \eqref{x03})
\[ 
  z_0^3 = \frac{- \gamma}{4K}.
\]  
Hence,
\begin{align*}
|\gamma| \bigg( \pv \int_{B_1} E_{\e, *}[\phi(z)]  \, V_{\alpha_\e}''(z) \, dz \bigg)
\leq 4K |z_0^3| \bigg( \pv \int_{B_1} E_{\e, *}[K ( (z - z_0)^4 - z_0^4 )] \, V_{\alpha_\e}''(z) \, dz \bigg),
\end{align*}
which, by Lemma~\ref{l:parabola}, is uniformly bounded in $\e \in (0,1)$ and in $z_0$ for $|z_0| \leq \| v_0' \|_\infty^{1/3} (4K)^{-1/3}$. 

Collecting the bounds above, we obtain
\[
  \underline H_{\rho, \e}[\psi, v](\o t, \o x)
  \leq \sigma = \psi_t (\o t, \o x)
\]
for some constant $\sigma > 0$ which only depends on $R, K, V, U, \|v_0\|_{C^1(\R)}$ and $\alpha_*$ (and $\alpha^*$ if $m=1$). This shows that $v$ is a supersolution of \eqref{HJe}.
\end{proof}

With Lemmas \ref{l:Pn:to:HJe} and \ref{l:barrier} the proof of Theorem \ref{t} is complete, which achieves the main aim of this section. The secondary aim is to show that any solution $(\bx, \bb)$ of \eqref{Pn} can be characterized by \eqref{HJe}. We achieve this aim by proving Proposition \ref{p:Pn:to:HJe:un} below. Indeed, Proposition \ref{p:Pn:to:HJe:un} states that $u_n$ can be characterized uniquely from \eqref{HJe}. Then, the union of the trajectories of $x_i$ (up to the time point when $b_i$ jumps to $0$) is precisely the jump set of $u_n$ in $Q$.

\begin{prop}
\label{p:Pn:to:HJe:un}
Consider the setting of Theorem \ref{t}. Then the viscosity solution $v$ to \eqref{HJe} with $\e = \frac1n$ which satisfies $v^*(0, \cdot) = (u_n^\circ)^*$ and $v_*(0, \cdot) = (u_n^\circ)_*$ is unique.
\end{prop}

\begin{proof}
The proof of \cite[Proposition 4.5]{VanMeursPeletierPozar22} applies with obvious modifications. The key ingredient of that proof is stability of \eqref{Pn} with respect to the initial data, which we establish in this paper in Theorem \ref{t:Pn}\ref{t:Pn:stab}.
\end{proof}

\newpage

\section{Discussion} 
\label{s:disc}
 
In this section we reflect on the setting, assumptions and statements in the main Theorems \ref{t:Pn} and \ref{t}. We also demonstrate the extent to which these theorems cover the examples of $V$ from the applications mentioned in Section \ref{s:intro:lit}.

\subsection{Weaker assumption on $V$} 
\label{s:disc:weak:V}
Our main results in Theorems \ref{t:Pn}, \ref{t:conv} and \ref{t} hold under weaker assumptions on $V$ than those stated in Assumptions \ref{a:UV:fg} and \ref{a:UV}. In this section we discuss how these assumptions can be weakened. Under these weaker assumptions all proofs in this paper remain valid with minor modifications.

We start with Theorem \ref{t:Pn} on the well-posedness and properties of \eqref{Pn}, currently stated under Assumption \ref{a:UV:fg}. Assumption \ref{a:UV:fg} can be relaxed as follows:
\begin{itemize}
  \item for the regularity, $V \in W_{\loc}^{3,1}(0,\infty)$ is sufficient,
  \item the monotonicity assumption only needs to hold locally around $x=0$, provided that $V'' \in L^\infty(1,\infty)$.
\end{itemize}
Local monotonicity is enough, because the difficulty with \eqref{Pn} is to control $x_i$ prior to collisions. To obtain this control it is sufficient to bound the interactions with particles further away by a constant. However, we need to keep $V'' \in L^\infty(1,\infty)$ to prevent particles from diverging to $\infty$ in finite time. 
In addition, Assumption \ref{a:UV:fg} imposes an upper bound on the singularity. This is needed for several estimates in Theorem \ref{t:Pn}, but not for Corollary \ref{c:Pn} or for the existence and uniqueness of solutions.

Next we treat Theorem \ref{t:conv} on the convergence of the Hamilton-Jacobi equations, currently stated under Assumption \ref{a:UV}. This theorem separates the three scaling regimes of $\alpha_\e$ labelled by $m = 1,2,3$. As mentioned in the introduction, in the case $m=1$ the tail of $V$ is expected to be of no importance, while in the case $m=3$ the singularity of $V$ only matters prior to collision. Hence, we weaken Assumption \ref{a:UV} depending on $m$. For $m=1$, Assumption \ref{a:UV} can be relaxed as follows:
\begin{itemize}
  \item for the regularity, $V \in C^3((0,\infty))$ is sufficient;
  \item no condition on $V^{(k)}(x)$ as $x \to \infty$ has to be imposed;
  \item instead of \ref{a:UV:int}, it is enough to have $x \mapsto x^3 \sup_{x < y < 1} |V'''(y)|$ in $L^1(0,1)$.
\end{itemize}
Then, as a consequence, in Lemma \ref{l:parabola} the constant $C$ may also depend on $\sup_{0 < \e < 1} \alpha_\e$.
For $m=2$, Assumption \ref{a:UV} can be relaxed as follows:
\begin{itemize}
  \item for the regularity, $V \in C^3((0,\infty))$ is sufficient;
  \item $V^{(k)}(x) \to 0$ as $x \to \infty$ only for $k = 0,1,2$;
  \item instead of \ref{a:UV:int}, it is enough to have $x \mapsto x^3 \V_3(x)$ in $L^1(0,\infty)$ and $x^4 V'''(x) \to 0$ as $x \to 0$.
\end{itemize} 
For $m=3$, Assumption \ref{a:UV} can be relaxed as follows:
\begin{itemize}
  \item for the regularity, $V \in C^4((0,\infty))$ is sufficient;
  \item $V^{(k)}(x) \to 0$ as $x \to \infty$ only for $k = 0,\ldots,3$;
  \item instead of \ref{a:UV:int}, it is enough to have 
  \begin{itemize}
    \item[$\bullet$] $x \mapsto x^3 \V_3(x)$ in $L^1(1,\infty)$, 
    \item[$\bullet$] $x \mapsto x^2 V''(x) \in L^1(0,1)$, and 
    \item[$\bullet$] $\sup_{x > 1} x^{k+1} |V^{(k)}(x)| < \infty$ for $k = 3,4$.
  \end{itemize}  
\end{itemize} 

Finally, Theorem \ref{t} simply holds under the combination of the weakened assumptions on Theorems \ref{t:Pn} and \ref{t:conv}.

\subsection{The restricted class of test functions} 
\label{s:disc:tFct} 

First, we argue that a restriction of the class of test functions in the definition of the viscosity solutions of \eqref{HJe} (see Definition \ref{d:HJe:VS}) is necessary for the existence of the solution $u_n$ of \eqref{HJe} with $\e = \frac1n$ as constructed in Theorem \ref{t}. Suppose that we employ the usual definition of viscosity solutions in which the space of test functions is not restricted. For simplicity, consider $n=2$, $U=0$ and the Riesz potential $V(x) = |x|^{-a}$ with $a > 0$. Consider the sketch of the solution in Figure \ref{fig:trajs} to the left of the $t$-axis. In this figure, $u_n$ equals $-\frac1n$ in the region bounded by the trajectories and $0$ outside. Take $(\o t, \o x)$ as the collision point and consider the test function $\phi(t,x) = C|x - \o x|^2 + c (t - \o t)$ for some constants $C, c > 0$. Note that $\phi (\o t, \o x) = 0 = u_n^*(\o t, \o x) $ and that $\o \phi_x (\o t, \o x) = 0$. It follows from the solution in Example \ref{ex:Pn=2:expl} that there exist $C, \frac1c > 0$ such that $\phi \geq u_n^*$. For such  $C, \frac1c > 0$ the function $\phi$ is an admissible test function in the usual definition. However, $\o \phi_t (\o t, \o x) > 0$, which contradicts that $u_n^*$ is a subsolution. 

This conclusion is not restricted to the case in which $n=2$, $U=0$ and $V$ is $a$-homogeneous; by Theorem \ref{t:Pn}\ref{t:Pn:LB:col} the same argument applies to two-particle collisions for any $n$, any $U$, and any $V$ with $x V''(x) \geq c x^{1+a}$ for all $x \in (0, \e_0)$ for some $a,c,\e_0 > 0$. Thus, we conclude that it is necessary to remove a certain class of quadratic test functions from the definition of viscosity solutions in order for $u_n$ to be a solution. 

Next we motivate that taking a very small class of test functions, namely those of the form $\phi(t,x) = \theta |x - \o x|^4 + g(t)$, is not restrictive for the class of potentials $V$ described in Section \ref{s:disc:weak:V}. Having fewer test functions makes it easier to prove the convergence of solutions (see Theorem \ref{t:conv}) and harder to prove a comparison principle (see Theorems \ref{t:CPe}, \ref{t:CP:HJ1} and \ref{t:CP:HJk}). Yet, our comparison principles hold under even weaker assumptions on $V$ than those in Section \ref{s:disc:weak:V}.

Finally, we address the choice for the power $4$ in the term $\theta |x - \o x|^4$. We simply took the smallest even integer greater than $2$; larger even numbers would work as well. The choice of this power becomes important only when considering potentials $V$ which are not integrable around $0$. For instance, consider the example above in which $u_n$ is not a viscosity solution in the usual definition. If we take $a \geq 2$ (see Example \ref{ex:Pn=2:expl}), then we can choose $C, c$ in $\phi(t,x) = C|x - \o x|^4 + c (t - \o t)$ to demonstrates that $u_n^*$ is not a subsolution in the sense of Definition \ref{d:HJe:VS}. However, the case $a \geq 2$ corresponds to potentials $V$ which are not integrable at $0$, and those potentials are excluded in Theorem \ref{t} for independent reasons. If one is nonetheless interested in considering \eqref{HJe} for potentials $V$ which are not integrable at $0$, then one should use a higher power than $4$, or use a function which is even flatter around $x = \o x$.

\subsection{Comparison between \eqref{Pn} and \eqref{HJe}}
\label{s:disc:HJe} 

In Section \ref{s:t} we have shown that \eqref{HJe} can characterize all solutions of \eqref{Pn}. 
Interestingly, the comparison principle in Theorem \ref{t:CPe} for \eqref{HJe} holds under a different set of assumptions than \eqref{Pn} (see Section \ref{s:disc:weak:V}). These assumptions are that $V$ is even, that $V|_{(0,\infty)} \in W_{\loc}^{2,1}(0,\infty)$ is convex and that $V'' \in L^1(1,\infty) \cap L^\infty(1,\infty)$. In comparison to the assumptions on \eqref{Pn}, we need less regularity, no sign restriction on $V'''$ and no lower bound on the singularity. However, we do need global convexity and integrability of the tails of $V''$.

Other than a different class of potentials $V$, we remark that \eqref{HJe} has the potential to describe a system of infinitely many particles (e.g.\ $u^\circ(x) = \sin x$). To prove that \eqref{HJe} admits a unique solution for such initial data with the techniques in this paper, two obstacles need to be overcome. First, the comparison principle requires the condition \eqref{CPe:far-field} on the sub- and supersolutions, which is not obviously met if the set of initial particle positions $\{ x \in \R \mid u^\circ(x) \in \e \Z \}$ is unbounded. Second, this paper provides no existence of solutions of \eqref{HJe} except for those built from solutions of \eqref{Pn}.

\subsection{Examples of $V$} In Section \ref{s:intro:lit} we mentioned several examples of $V$. Here we check whether they satisfy Assumptions \ref{a:UV:fg} and \ref{a:UV}, i.e.\ the assumptions under which all main results in this paper hold. 

First, we consider $V$ as given by \eqref{Vwall} which describes the setting of dislocation walls. In Appendix \ref{a:Vwall} we show that it satisfies all assumptions. 

Second, we consider the Riesz potential $V$ given by \eqref{VR} with parameter $a \in (-1,\infty)$. For $a \in (-1,1) \cup \{2\}$ we mentioned applications to dislocation dynamics. Assumption  \ref{a:UV:fg} holds for all $a \in (-1,\infty)$, and thus Theorem \ref{t:Pn} on \eqref{Pn} holds. However, $V$ does not satisfy Assumption \ref{a:UV} because it is not integrable on $\R$. This is not surprising; the scaling in \eqref{Val} assumes integrability. Yet, for $m=1$, this scaling is not important. Moreover, for $m=1$ the potential $V$ satisfies the weakened assumptions mentioned in Section \ref{s:disc:weak:V} for $|a| < 1$. Hence, for $m=1$, Theorem \ref{t} applies to $V$ for any $|a| < 1$. This only leaves out the case $a = 2$, which corresponds to dislocation dipoles.

\section*{Acknowledgements}

The author was supported by JSPS KAKENHI Grant Number JP20K14358. The author expresses his gratitude to Norbert Po\v{z}\'ar for fruitful discussions. 

\appendix

\section{The dislocation wall potential fits to Theorem \ref{t}}
\label{a:Vwall}

Here we show that $V$ as given by \eqref{Vwall} satisfies Assumptions \ref{a:UV:fg} and \ref{a:UV}. Evenness is obvious. By writing
\begin{align*} 
  V(x) = \frac x{\sinh x} \cosh x + \log \frac x{2\sinh x} - \log|x|
\end{align*}
and noting that the first two terms in the right-hand side are analytic around $0$, the assumptions on the regularity and the singularity follow. To prove the monotonicity, we compute for any $x > 0$
\begin{align} \label{pf:uu}
  V(x) &= \frac{2x}{e^{2x} - 1} - \log (1 - e^{-2x}) > 0, \\ \notag
  V'(x) &= \frac{-x}{\sinh^2 x} < 0, \\\notag
  V''(x) &= \frac{2x \cosh x - \sinh x }{\sinh^3 x} > 0, \\\notag
  V'''(x) &= \frac{4 (\sinh x - x \cosh x)\cosh x - 2 x }{\sinh^4 x} < 0, 
\end{align}
where the inequalities follow from the identity $x \cosh x > \sinh x$. 

To prove the remaining Assumption \ref{a:UV}\ref{a:UV:to0},\ref{a:UV:int}, it is sufficient to show that $V^{(k)}$ has exponentially decaying tails for all $k \geq 0$. For $k=0$ this is obvious from \eqref{pf:uu}. For $k \geq 1$ we obtain by induction that
\[
  V^{(k)}(x) = \frac{ x p_{k-1} (\cosh x, \sinh x) + q_{k-1} (\cosh x, \sinh x) }{\sinh^{k+1} x}
  \qquad \text{for all } k \geq 1
\]
for some $(k-1)$-homogeneous polynomials $p_{k-1}$, $q_{k-1}$ of two variables. Hence, $V^{(k)}$ has exponentially decaying tails.

\end{document}